\documentclass[preprint,12pt]{elsarticle}
\usepackage{soul}
\usepackage{psfrag}
\usepackage{a4wide}
\usepackage{rotating}
\usepackage{color}
\usepackage{amssymb}
\usepackage{amsmath}
\usepackage{amsthm}
\usepackage{verbatim}
\usepackage{soul}

\newcommand{\ab}[1]{\boldsymbol{#1}}
\def\bfm#1{\boldsymbol{#1}}
\newcommand{\f}[1]{\mathbf{#1}}
\newcommand{\bb}[1]{\bfm{#1}}
\newcommand{\N}{\mathbb N}

\newcommand{\R}{\mathbb R}

\newcommand{\V}{\mathcal{V}_h}
\newcommand{\VO}{\mathcal{V}_{0h}}
\newcommand{\W}{\mathcal{W}_{0h}}

\DeclareMathOperator{\Span}{span}

\newtheorem{thm}{Theorem}
\newtheorem{lem}{Lemma}
\newtheorem{prop}{Proposition}

\theoremstyle{definition}
\newtheorem{ex}{Example}

\newtheorem{rem}{Remark}

\newproof{pf}{proof}
\advance\textheight by 0.4cm
\advance\topmargin by -0.2cm

\definecolor{gold}{rgb}{1,0.7,0}
\definecolor{dred}{rgb}{0.92,0,0}
\definecolor{dgreen}{rgb}{0,0.6,0}
\def\MK{\color{black}}
\def\VV{\color{black}}

\begin{document}

\begin{frontmatter}

\title{Solving the triharmonic equation over multi-patch domains \\ using isogeometric analysis}


\author[lnz]{Mario Kapl\corref{cor}}
\ead{mario.kapl@ricam.oeaw.ac.at}

\author[slo1,slo2]{Vito Vitrih}
\ead{vito.vitrih@upr.si}
 
\address[lnz]{Johann Radon Institute for Computational and Applied Mathematics, \\Austrian Academy of Sciences, Linz, Austria}

\address[slo1]{IAM and FAMNIT, University of Primorska, Koper, Slovenia}

\address[slo2]{Institute of Mathematics, Physics and Mechanics, Ljubljana, Slovenia}

\cortext[cor]{Corresponding author}
   
\begin{abstract}
We present a framework for solving the triharmonic equation over bilinearly parameterized planar multi-patch domains by means of isogeometric analysis. Our approach is based on the 
construction of a globally $C^2$-smooth isogeometric spline space which is used as discretization space. The generated $C^2$-smooth space consists of three different types 
of isogeometric functions called patch, edge and vertex functions. All functions are entirely local with a small support, and {\VV numerical examples indicate that they are 
well-conditioned}. The construction of the functions is simple and works uniformly for all multi-patch configurations. While the patch and edge functions are given by a closed form 
representation, the vertex functions are {\VV obtained by computing the null space of a small system of linear equations}. Several examples demonstrate the potential of our approach 
for solving the triharmonic equation. 
\end{abstract}


\end{frontmatter}

\section{Introduction}

In isogeometric analysis (IGA), which was introduced by Hughes et al. \cite{HuCoBa04}, standard CAD functions for describing the geometry, such as polynomial splines or NURBS, are also 
used {\MK for the numerical simulation} of partial differential equations (PDEs), cf. \cite{ANU:9260759,CottrellBook,HuCoBa04}. IGA provides the possibility to solve 
high order PDEs by using standard Galerkin discretization, see e.g.~\cite{BaDe15, TaDe14}, but  which requires isogeometric spline spaces of  high smoothness. 
In case of $4$-th order PDEs, such as the biharmonic equation \cite{BaDe15, CoSaTa16, KaBuBeJu16, KaViJu15, TaDe14}, the Kirchhoff-Love shell problem \cite{
benson2011large, kiendl-bazilevs-hsu-wuechner-bletzinger-10,kiendl-bletzinger-linhard-09, KiHsWuRe15, NgZhZh17}, 
or the Cahn-Hilliard equation \cite{gomez2008isogeometric,  LiDeEvBoHu13}, $C^{1}$-smooth isogeometric functions are needed. 
{\VV Furthermore, $C^1$-smooth isogeometric functions are also needed for plane problems of first strain gradient elasticity \cite{gradientElast2011, KhakaloNiiranenC1} and for a 
locking-free reformulation of Reissner-Mindlin plates \cite{ReissnerMindlin2015}.}
In order to solve $6$-th order PDEs, such as the triharmonic equation \cite{BaDe15, KaVi17b, KaVi17a, TaDe14}, the phase-field crystal 
equation \cite{BaDe15, Gomez2012}, the Kirchhoff plate model based on the Mindlin's gradient elasticity theory {\VV \cite{KhakaloNiiranenC2, Niiranen2016}}, or the 
gradient-enhanced continuum damage model \cite{GradientDamageModels}, even $C^{2}$-smooth functions are required. In particular for the case of $6$-th order PDEs, 
these problems have been mainly considered so far for single-patch domains or simple closed surfaces, where the required smoothness of an isogeometric functions is directly obtained 
by the smoothness of the underlying spline space. In case of multi-patch domains, the construction of
$C^{s}$-smooth ($s\geq 1$) isogeometric spline spaces defined on multi-patch domains is linked to the concept of geometric continuity of multi-patch surfaces (cf. \cite{HoLa93, Pe02}). 
More precisely, an isogeometric function is $C^{s}$-smooth on a multi-patch domain if and only if its graph surface over the multi-patch 
domain is $G^{s}$-smooth (cf. \cite{Pe15, KaViJu15}). The design  of $C^s$-smooth isogeometric spline spaces {\MK over multi-patch domains} is the task of recent research, see e.g. 
\cite{BeMa14, BlMoVi17, CoSaTa16, KaBuBeJu16, KaSaTa17a, KaSaTa17c, KaViJu15, Pe15-2, mourrain2015geometrically, NgPe16, ToSpHu17b,ToSpHu17} for $s=1$ and 
e.g. \cite{KaVi17a, KaVi17b, KaVi17c, ToSpHiHu16} for $s=2$. 

This work focuses on solving the triharmonic equation over bilinearly parameterized planar multi-patch domains by using IGA.  
To our knowledge this problem was handled for the first time in \cite{KaVi17b, KaVi17a}. There, a basis of the entire space of $C^2$-smooth isogeometric functions
is generated. The construction is based on the concept of minimal determining sets (cf.~\cite{LaSch07}) for the involved spline coefficients and requires the symbolic 
{\MK computation of the null space} of a large (global) 
system of linear equations. Further disadvantages of this approach are the following: The resulting functions which are defined across the common interfaces possess in general 
large supports along one or more interfaces. The method is restricted to isogeometric spline functions of bidegree~$(p,p)$ with $p=5,6$ and regularity $r=2$ within the 
single patches. Moreover, the presented examples of solving the triharmonic equation were restricted to one particular level of $h$-refinement. 

{\MK Two further constructions of $C^2$-smooth spline functions over multi-patch domains are \cite{KaVi17c, ToSpHiHu16}, but both methods have not been applied so far to solve $6$th 
order {\VV PDEs}. In~\cite{KaVi17c}, $C^2$-smooth spline spaces over the class of so-called bilinear-like two-patch parameterizations, which contains the subclass of bilinear two-patch 
geometries, were considered. There, the dimension of this space was analyzed and an explicit basis construction was presented, which will serve as a basis for our construction in 
the multi-patch case. In~\cite{ToSpHiHu16}, a polar spline framework is developed to construct $C^2$-smooth isogeometric spaces which is based on a special construction in the 
vicinity of the polar point to ensure $C^2$-smoothness also there.  

Beside multi-patch {\VV quadrangular} domains, triangulations have been used to generate $C^2$-smooth (or even smoother) spline spaces over complex domains. The book~\cite{LaSch07} 
gives an overview of different techniques to model such smooth spline spaces, and provides a detailed bibliography on this topic. There, also the concept of minimal determining sets 
is recalled, which is a common strategy to generate a basis of a smooth spline space over a given triangulation. The minimal determining set implicitly describes a basis of the null 
space of the homogeneous linear system obtained by the corresponding smoothness conditions. We will use this concept for the construction of those basis functions which will be defined 
in the neighborhood of a vertex of the multi-patch domain.} {\MK Some more recent constructions of $C^2$-smooth spline spaces on triangulations are e.g. 
\cite{DavydovYeo2013, Groselj2016, LycheMuntingh2014, Speleers2012, Speleers2013}.}

The present paper improves and extends the approach~\cite{KaVi17b, KaVi17a} in several directions. Instead of constructing the entire space of $C^2$-smooth isogeometric functions, 
which has a complex structure, a simpler subspace~$\W$ is generated. The subspace~$\W$ maintains the full approximation properties of the entire space and is defined as the direct 
sum of spaces corresponding to the single patches, edges and vertices. For each of these spaces the construction of the basis functions is simple and leads to basis functions which 
possess small supports {\MK and} can be described by explicit formulae {\VV or by computing the null space of a small system of linear equations.} {\MK Furthermore, the numerical 
examples indicate that the generated basis functions are well-conditioned}. The basis construction 
of the single spaces is based on and extends the explicit {\MK construction in~\cite{KaVi17c}, 
and can be applied} for any degree $p \geq 5$ and any regularity $2\leq r \leq p-3$ {\MK at the inner knots} 
within the single patches. {\MK Moreover,}
the construction of the space~$\W$ works uniformly for all possible multi-patch configurations. In contrast to 
\cite{KaVi17b, KaVi17a}, the triharmonic equation is solved on several bilinearly parameterized multi-patch domains for different levels of $h$-refinement, where the numerical 
results show the potential of our approach. 

The remainder of the paper is organized as follows. Section~\ref{sec:modelProblem} introduces the model problem which is studied in this work, i.e., 
solving the triharmonic equation over bilinear multi-patch domains by means of IGA. This requires the use of a discretization space consisting of globally $C^2$-smooth 
isogeometric functions. Section~\ref{sec:C2smoothspaces} recalls the concept of $C^2$-smooth isogeometric spline spaces and summarizes the explicit construction~\cite{KaVi17c} for 
the case of two patches which serves as a basis for the multi-patch case. In Section~\ref{sec:C2smoothtestfunctions}, we describe the construction of the discretization space for 
solving the triharmonic equation. This space is a subspace of the entire
space of globally $C^2$-smooth isogeometric spline spaces and is defined as the direct sum of subspaces of three different types
called patch, edge and vertex subspaces. The potential of our method for solving the triharmonic equation is demonstrated on the basis of several examples in 
Section~\ref{sec:triharmonic_examples}, where amongst others the convergence rates and condition numbers obtained under $h$-refinement are numerically studied. 
Finally, we conclude the paper.

\section{The model problem}  \label{sec:modelProblem}

We introduce the model problem which will be considered throughout the paper. The goal is to solve a particular sixth-order partial differential equation, 
namely the triharmonic equation with homogeneous boundary conditions of order 2. 

\subsection{The triharmonic equation}

Let $\Omega = \cup_{\ell=1}^P \Omega^{(\ell)}$ be a planar multi-patch domain. We have to find the function $u : \Omega  \rightarrow \R$ which solves 
 for $f \in H^{0}(\Omega)$ the equation 
\begin{equation} \label{eq:triharmonic_problem}
 \triangle^{3} u(\ab{x})  = - f(\ab{x}), \quad  \ab{x} \in \Omega,
 \end{equation}
 with the boundary conditions
\begin{equation} \label{eq:triharmonic_problem_boundary}
 u(\ab{x}) = \frac{\partial u}{\partial \ab{n}}(\ab{x}) = \triangle u(\ab{x})  = 0 , \quad \ab{x} \in \partial \Omega.
 \end{equation}
Using the weak formulation of \eqref{eq:triharmonic_problem} and \eqref{eq:triharmonic_problem_boundary} we have to find $u \in \mathcal{V}_{0}$, with 
\[
 \mathcal{V}_{0} = \{ v \in H^2(\Omega) : \triangle v \in H^1(\Omega) \mbox{ and }  v(\ab{x}) = 
\frac{\partial v}{\partial \ab{n}}(\ab{x}) = \triangle v(\ab{x})  = 0 \mbox{ for }\ab{x} \in \partial \Omega \},
\]
such that
 \begin{equation} \label{eq:weak_triharmonic}
   \int_{\Omega} \nabla\left(\triangle u(\ab{x})\right) \cdot \nabla \left( \triangle v(\ab{x}) \right) \mathrm{d}\ab{x} = \int_{\Omega} f(\ab{x}) v(\ab{x}) \mathrm{d}\ab{x},
 \end{equation}
 where $\cdot$ denotes the standard inner product, is satisfied  for all $v \in \mathcal{V}_{0}$, cf. \cite{BaDe15, TaDe14}. 
 In order to discretize problem~\eqref{eq:weak_triharmonic} by applying Galerkin projection, a finite dimensional function space $\W \subseteq \mathcal{V}_{0}$ is 
 required. 
Assume that we have such a space $\W$ with a basis $\{w_{i}\}_{i \in I_h}$, where $I_h= \{1, 2, \ldots, \dim \W \}$.  
Then, we have to find 
\begin{equation*} 
 u_{h}(\ab{x})=\sum_{i \in I_h}
 c_{i} w_{i}(\ab{x}),  \quad  c_{i} \in  \R,
\end{equation*}
which solves the system of equations
\begin{equation*} 
  \int_{\Omega}  \nabla\left(\triangle u_{h}(\ab{x}) \right) \cdot \nabla \left( \triangle v_{h}(\ab{x})\right) \mathrm{d}\ab{x} = \int_{\Omega} f(\ab{x}) v_{h}(\ab{x}) 
  \mathrm{d}\ab{x}
 \end{equation*}
for all $v_{h} \in  \W$. This results in a system of linear equations 
\[
S\ab{c}=\ab{f}
\]for the unknown coefficients $\ab{c} = (c_{i})_{i \in I_h}$, 
where the elements of the matrix $S = (s_{i,j} )_{i,j \in I_h}$ and the elements of the right-hand side vector $\ab{f} = (f_i)_{i \in I_h}$ are given by
\begin{equation} \label{eq:sij} 
 s_{i,j}=\int_{\Omega} \nabla\left(\triangle w_{i}(\ab{x}) \right) \cdot \nabla \left(\triangle w_{j}(\ab{x}) \right) \mathrm{d}\ab{x} 
 \quad {\rm  and} \quad  f_{i}= \int_{\Omega} f(\ab{x}) w_{i}(\ab{x}) 
 \;\mathrm{d}\ab{x}.
 \end{equation}

In this work, we will follow the isogeometric approach to solve the triharmonic equation. For this purpose, we will construct an isogeometric 
space $\W \subseteq \mathcal{V}_{0}$ and an associated basis $\{w_{i}\}_{i \in I_h}$, see Section~\ref{sec:C2smoothtestfunctions}. 
Beside the fulfillment of the homogeneous boundary conditions~\eqref{eq:triharmonic_problem_boundary}, the generated basis functions~$w_i$ will be $C^2$-smooth, since 
$C^1$-smoothness is not enough to ensure that $w_{i} \in \mathcal{V}_{0}$. 

\subsection{Using the isogeometric approach}

We describe the isogeometric approach to compute the elements in \eqref{eq:sij}. 
We assume that the planar multi-patch domain $\Omega$ 
consists of 
\begin{itemize}
 \item $P$ patches~$\Omega^{(\ell)}$, $\ell=1,2,\ldots, P$, with $P \in \N$ and $P \geq 2$,
 \item $E$ non-boundary common edges $\Gamma^{(s)}$, $s=1,2,\ldots,E$, and
 \item $V$ inner and boundary vertices $\bfm{v}^{(\rho)}$ of valency $\bar{\nu}_{\rho} \geq 3$, $\rho=1,2,\ldots,V$. 
 \footnote{In this work, a boundary vertex of valency two is not considered as a vertex $\ab{v}^{(\rho)}$.}{}
\end{itemize}
In addition, we assume that 
\begin{itemize}
\item the deletion of any vertex does not split $\Omega$ into subdomains, whose union would be unconnected, 
\item {\VV all subdomains $\Omega^{(\ell)}$ are strictly convex quadrangular patches, whose interiors are mutually disjoint,} 
\item any two patches~$\Omega^{(\ell)}$ and $\Omega^{(\ell')}$ have either an empty intersection, possess exactly one common vertex or share the whole common edge, and 
\item each patch $\Omega^{(\ell)}$ is parameterized by a bilinear, bijective and regular geometry
mapping~$\ab{F}^{(\ell)}$,
\begin{align*}
 \ab{F}^{(\ell)}: [0,1]^{2}  \rightarrow \R^{2}, \quad 
 \bb{\xi}^{(\ell)} =(\xi^{(\ell)}_1,\xi^{(\ell)}_2) \mapsto
  (F^{(\ell)}_1,F^{(\ell)}_2)=
 \ab{F}^{(\ell)}(\bb{\xi}^{(\ell)}), \quad \ell \in \{1, 2,\ldots,P\},
\end{align*}
such that $\Omega^{(\ell)} = \ab{F}^{(\ell)}([0,1]^{2})$, see Fig.~\ref{fig:geommetryToOmega}. 
\end{itemize} 

\begin{figure}[htp] 
\centering
\includegraphics[width=12cm]{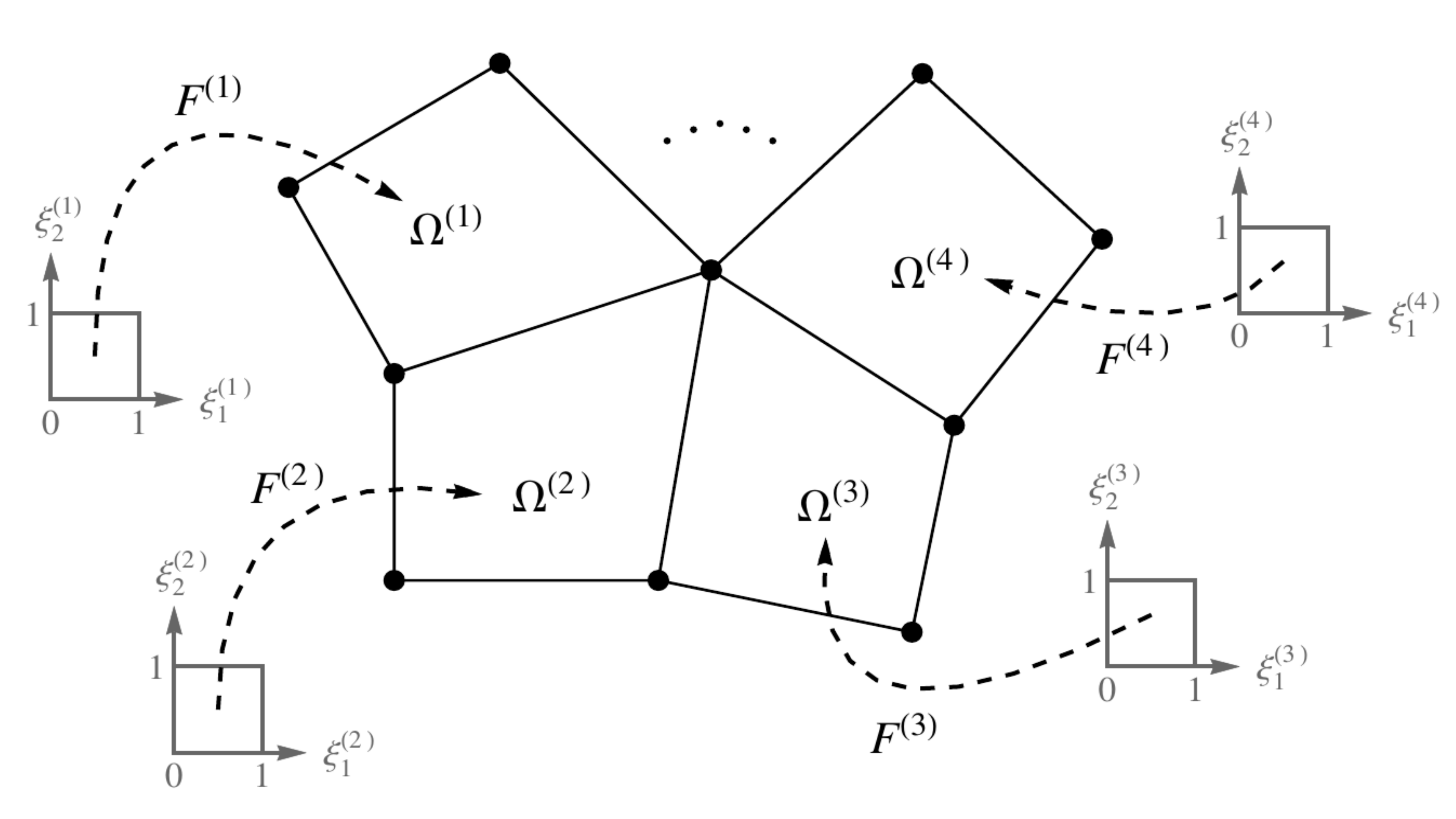}
\caption{The multi-patch domain $\Omega= \cup_{\ell=1}^{P} \Omega^{(\ell)}$ with the corresponding geometry mappings $\bfm{F}^{(\ell)}$, $\ell=1,2,\ldots,P$. } 
\label{fig:geommetryToOmega}
\end{figure}

Let $J^{(\ell)}$ be the Jacobian of $\ab{F}^{(\ell)}$ and let \begin{equation*}
{\VV K^{(\ell)}(\bb{\xi}^{(\ell)})} = \left(J^{(\ell)}(\bb{\xi}^{(\ell)})\right)^{-T}  \left(J^{(\ell)}(\bb{\xi}^{(\ell)})\right)^{-1} |\det J^{(\ell)}(\bb{\xi}^{(\ell)})|.
\end{equation*}
{Furthermore, let $W^{(\ell)}_i = w_i $} $\circ \ab{F}^{(\ell)}$, $i \in I_{h}$. 
Then, we compute the elements in \eqref{eq:sij} patch-wise by 
\begin{equation*}
  s_{i,j} = \sum_{\ell=1}^P s^{(\ell)}_{i,j} \;  \mbox{ and } \quad f_{i} = \sum_{\ell=1}^P f^{(\ell)}_{i},
\end{equation*} 
where 
\begin{align*}  
    s^{(\ell)}_{i,j} = & \int_{[0,1]^{2}} \nabla \left( \frac{1}{|\det J^{(\ell)}(\bb{\xi}^{(\ell)})|}  \,
    \nabla \cdot \left( {\VV K^{(\ell)}(\bb{\xi}^{(\ell)})} \, \nabla W^{(\ell)}_{i}(\bb{\xi}^{(\ell)})\right) \right)
    \cdot \nonumber \\
    & \left(  {\VV K^{(\ell)}(\bb{\xi}^{(\ell)})} \,  \nabla \left( \frac{1}{|\det J^{(\ell)}(\bb{\xi}^{(\ell)})|}  \,
    \nabla \cdot \left( {\VV K^{(\ell)}(\bb{\xi}^{(\ell)})} \,\nabla W^{(\ell)}_{j}(\bb{\xi}^{(\ell)})\right) \right) \right)
  \, \mathrm{d}\bb{\xi}^{(\ell)},
\end{align*}
and
\begin{equation*}
f^{(\ell)}_{i}  = \int_{[0,1]^{2}} f(\ab{F}^{(\ell)}(\bb{\xi}^{(\ell)})) W^{(\ell)}_{i}(\bb{\xi}^{(\ell)}) |\det J^{(\ell)}(\bb{\xi}^{(\ell)})| \; 
\mathrm{d}\bb{\xi}^{(\ell)},
\end{equation*}
cf. \cite{BaDe15, KaVi17a}.

\section{$C^2$-smooth isogeometric spline spaces} \label{sec:C2smoothspaces}

In Section~\ref{sec:C2smoothtestfunctions}, the isogeometric discretization space~$\W$ will be generated as a subspace of the space of $C^2$-smooth isogeometric spline 
functions on $\Omega$. Before, we recall the concept of $C^2$-smooth isogeometric spline spaces, cf.~\cite{KaVi17b, KaVi17c}, and adapt the notations appropriately.

\subsection{The space of $C^2$-smooth isogeometric spline functions}

 In order to define the space of $C^2$-smooth isogeometric spline functions on $\Omega$, we need some additional definitions and notations. 
{\VV Let $p \geq 5$, $k \in \N_{0}$ and for $k\geq 1$ let $2 \leq r \leq p-3$. Moreover let $h=\frac{1}{k+1}$.} 
We denote by $\mathcal{S}_{h}^{p,r}([0,1])$ the univariate spline space on the interval~$[0,1]$ of 
degree~$p$ and regularity $C^r$ possessing the open knot vector  
\begin{equation*}  
(\underbrace{0,0,\ldots,0}_{(p+1)-\mbox{\scriptsize times}},
\underbrace{\textstyle \tau_1,\tau_1,\ldots ,\tau_1}_{(p-r) - \mbox{\scriptsize times}}, 
\underbrace{\textstyle \tau_2,\tau_2,\ldots ,\tau_2}_{(p-r) - \mbox{\scriptsize times}},\ldots, 
\underbrace{\textstyle \tau_k,\tau_k,\ldots ,\tau_k}_{(p-r) - \mbox{\scriptsize times}},
\underbrace{1,1,\ldots,1}_{(p+1)-\mbox{\scriptsize times}}),
\end{equation*}
where the $k$ different inner knots $\tau_j$, $j \in \{1,2,\ldots,  k\}$, are equally distributed, i.e., $\tau_j = \frac{j}{k+1}=jh$. Let $N_{i}^{p,r}$, 
$i=0,1,\ldots,p+k(p-r)$, be the B-splines of the spline space $\mathcal{S}^{p,r}_{h}([0,1])$, and let $\mathcal{S}^{p,r}_{h}([0,1]^2)$ be the bivariate tensor-product spline space on 
the unit-square~$[0,1]^2$ spanned by the B-splines $N^{p,r}_{i,j} = N^{p,r}_{i}N^{p,r}_{j}$, $i,j=0,1,\ldots,p + k (p-r)$. Note that $h$ is the mesh-size of the spline 
spaces~$\mathcal{S}^{p,r}_{h}([0,1])$ and $\mathcal{S}^{p,r}_{h}([0,1]^2)$. {\MK In addition, in case of $k=0$ (i.e. $h=1$), the spaces $\mathcal{S}^{p,r}_{1}([0,1])$ and 
$\mathcal{S}^{p,r}_{1}([0,1]^2)$ are for any $r$ just the corresponding spaces of polynomials of degree~$p$ and bidegree~$(p,p)$, respectively.} 
Below, we assume that the number of inner knots satisfies $k \geq \frac{9-p}{p-r-2}$, which implies $h \leq \frac{p-r-2}{7-r} $.

Recall that the geometry mappings $\ab{F}^{(\ell)}$, $\ell = 1,2,\ldots, P $, are bilinear parameterizations, which also implies that $\ab{F}^{(\ell)} \in 
\mathcal{S}_{h}^{p,r}([0,1]^2) \times \mathcal{S}^{p,r}_{h}([0,1]^2)$. Then, the space of globally $C^{2}$-smooth isogeometric spline functions on $\Omega$ 
(with respect to the spline space~$\mathcal{S}_{h}^{p,r}([0,1]^2)$) is  defined as 
\begin{equation*}
\V = \left\{ \phi \in C^{2}(\Omega) : \; \phi |_{\Omega^{(\ell)}} \in {\mathcal{S}_{ h}^{p,r}([0,1]^{2})} \circ (\ab{F}^{(\ell)})^{-1}, \; \ell \in \{1,2, \ldots, P\}   \right\}.
\end{equation*}
The graph surface $\ab{\Sigma}:[0,1]^2 \to \Omega \times \R$  of an isogeometric function $\phi \in \V$ is given patch-wise by the graph surface patches 
$$
\ab{\Sigma}^{(\ell)}(\bfm{\xi}^{(\ell)}) = \left(\ab{F}^{(\ell)}(\bfm{\xi}^{(\ell)}), g^{(\ell)}(\bfm{\xi}^{(\ell)}) \right)^T, \quad g^{(\ell)} \in {\mathcal{S}_{ h}^{p,r}([0,1]^2)}, 
 \quad \ell=1,2,\ldots,P, 
$$
with
\begin{equation}  \label{eq:g_ell}
g^{(\ell)}(\bfm{\xi}^{(\ell)})= \phi \circ \ab{F}^{(\ell)}(\bfm{\xi}^{(\ell)})  = 
\sum_{i=0}^{p+k(p-r)} \sum_{j=0}^{p+k(p-r)} d^{(\ell)}_{i,j} N_{i,j}^{p,r} (\bfm{\xi}^{(\ell)}), \quad d^{(\ell)}_{i,j} \in \R.
\end{equation}

\begin{figure}[htp] 
\centering
\includegraphics[width=9cm]{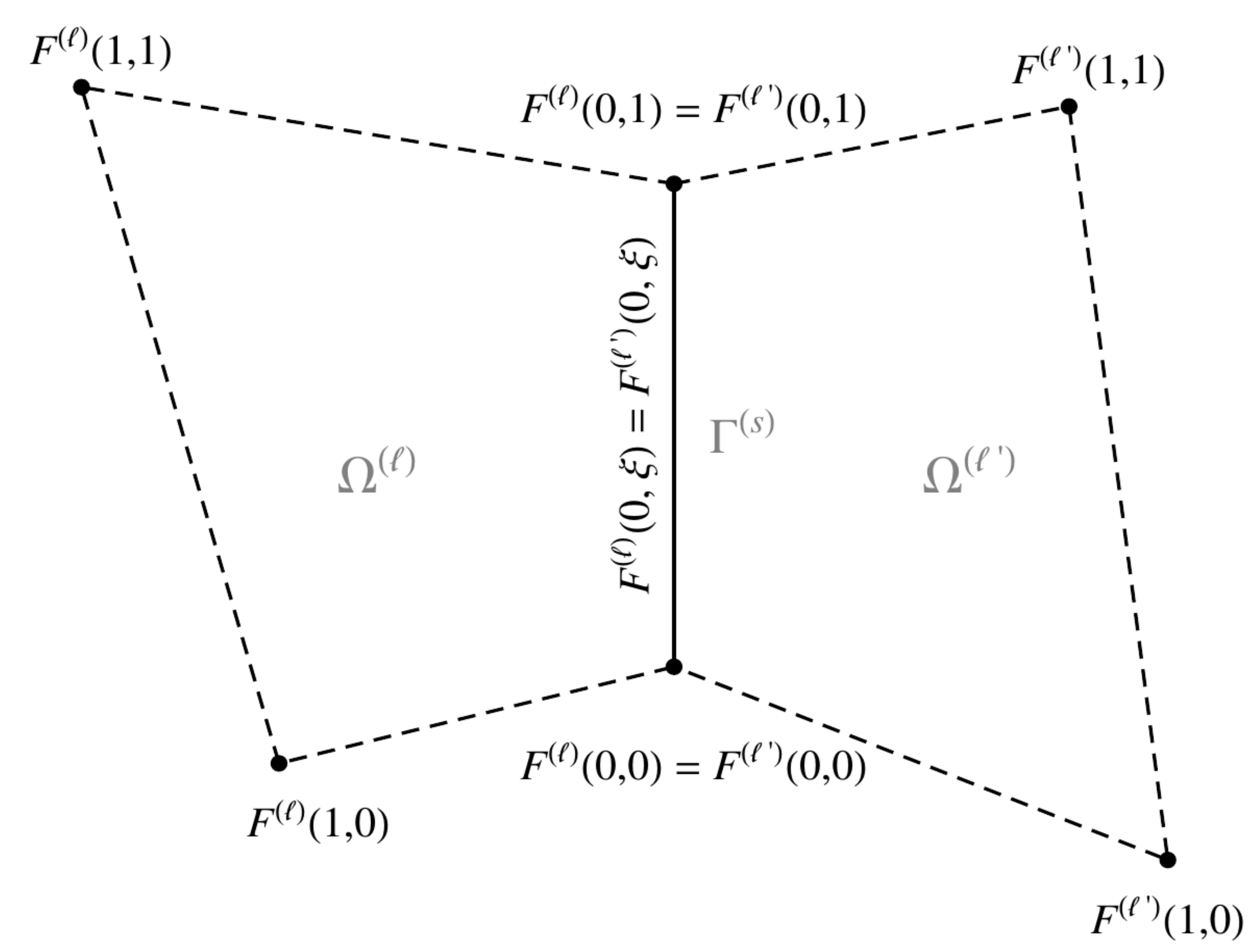}
\caption{Considering two neighboring patches~$\Omega^{(\ell)}$ and $\Omega^{(\ell)}$, we can always assume (without loss of generality) that the two corresponding geometry mappings 
$\ab{F}^{(\ell)}$ and $\ab{F}^{(\ell')}$ are parameterized as shown.} \label{fig:situation_two_patches}
\end{figure}

The functions in $\V$ can be characterized by using the concept of geometric continuity (cf. \cite{Pe15, KaViJu15}): 
\emph{An isogeometric function $\phi$ belongs to the space $\V$ if and only if for all neighboring patches $\Omega^{(\ell)}$ and $\Omega^{(\ell')}$ sharing an 
interface~$\Gamma^{(s)} = \Omega^{(\ell)} \cap \Omega^{(\ell')}$ (where $s \in \{1,2,\ldots,E\}$), the two graph surface patches 
$\ab{\Sigma}^{(\ell)}$ and $\ab{\Sigma}^{(\ell')}$ meet at the common interface $\Gamma^{(s)}$
with $G^2$ continuity.} Since the geometry mappings $\ab{F}^{(\ell)}$ and $\ab{F}^{(\ell')}$ are given in advance, the $G^2$ continuity conditions for the graph surface 
patches~$\ab{\Sigma}^{(\ell)}$ and $\ab{\Sigma}^{(\ell')}$ lead to conditions for the spline functions $g^{(\ell)}$ and $g^{(\ell')}$, which  
determine again linear constraints on the spline coefficients $d_{i,j}^{(\ell)}$ and $d_{i,j}^{(\ell')}$. 
These conditions were studied in \cite{KaVi17c} for the class of so-called bilinear-like $G^2$ geometries, which includes the class of bilinearly parameterized 
geometries.  Let us shortly recall the conditions for the two neighboring patches~$\Omega^{(\ell)}$ and $\Omega^{(\ell')}$. For the sake of simplicity, 
we can always reparameterize (if needed) the two geometry mappings $\ab{F}^{(\ell)}$ and $\ab{F}^{(\ell')}$ to have the situation as given in 
Fig.~\ref{fig:situation_two_patches},
i.e., 
$$
\VV
\ab{F}^{(\ell)}(0,{\xi_2})=
\ab{F}^{(\ell')}(0,{\xi_2}),\quad \xi_2 =\xi_2^{(\ell)}=\xi_2^{(\ell')}\in[0,1].
$$ 
To simplify the notation, let us denote the common interface $\Gamma^{(s)}$ in this section by $\Gamma$ and let
\begin{equation*}  \label{eq:alphaLRbar}
 \bar{\alpha}_{\Gamma}^{(\tau)}(\xi) = \det [D_{\xi_1^{(\tau)}}\ab{F}^{(\tau)}(0,\xi), D_{\xi}\ab{F}^{(\tau)}(0,\xi)] , \;
  \alpha^{(\tau)}_\Gamma(\xi) = \gamma_1(\xi) \bar{\alpha}^{(\tau)}_\Gamma(\xi),\; \widehat{\alpha}^{(\tau)}_\Gamma(\xi) = \gamma_2(\xi) \bar{\alpha}_\Gamma^{(\tau)}(\xi),
\end{equation*}
for $\tau\in \{\ell,\ell'\}$ and
\[
\bar{\beta}_\Gamma(\xi) = \det[ D_{\xi_1^{(\ell)}}\ab{F}^{(\ell)}(0,\xi) , D_{\xi_1^{(\ell')}}\ab{F}^{(\ell')}(0,\xi) ], \quad
 \beta_\Gamma(\xi) = \gamma_1(\xi) \bar{\beta}_\Gamma (\xi)
\]
for $\gamma_i: [0,1]\to \R,\; i=1,2.$ Note that $\bar{\alpha}^{(\ell)}_{\Gamma}$ and $\bar{\alpha}^{(\ell')}_{\Gamma}$ are linear polynomials with $\bar{\alpha}^{(\ell)}_{\Gamma} < 0$ 
and $\bar{\alpha}^{(\ell')}_{\Gamma} > 0$, respectively, and $\bar{\beta}_{\Gamma}$ is a quadratic polynomial.
We can write the function $\beta_\Gamma$ also as 
\begin{equation*} \label{eq:beta}
 \beta_\Gamma(\xi) = \alpha_\Gamma^{(\ell)}(\xi) \beta_\Gamma^{(\ell')}(\xi) - \alpha_\Gamma^{(\ell')}(\xi) \beta_\Gamma^{(\ell)}(\xi),
\end{equation*}
where $\beta_\Gamma^{(\ell)}$, $\beta_\Gamma^{(\ell')}: [0,1] \rightarrow \R$ are given as
\begin{equation*} \label{eq:betaS}
 \beta_\Gamma^{(\tau)}(\xi) = \frac{D_{\xi_1^{(\tau)}} \ab{F}^{(\tau)}(0,\xi) \cdot D_{\xi}\ab{F}^{(\tau)}(0,\xi)}{||D_{\xi}\ab{F}^{(\tau)}(0,\xi)||^{2}}, \quad \tau \in \{\ell,\ell'\}.
\end{equation*}
Moreover let
\begin{equation*} \label{eq:eta}
 \eta_\Gamma(\xi) = 2 \gamma_2(\xi) ( \alpha_\Gamma^{(\ell)})'(\xi) \alpha_\Gamma^{(\ell')}(\xi) \beta_\Gamma(\xi),
\end{equation*}
\begin{equation*} \label{eq:theta}
 \theta_\Gamma(\xi)= 2 \gamma_2(\xi) \left(\alpha_\Gamma^{(\ell)}(\xi)  (\beta_\Gamma^{(\ell)})'(\xi) - (\alpha_\Gamma^{(\ell)})'(\xi) \beta_\Gamma^{(\ell)}(\xi)\right) \alpha_\Gamma^{(\ell')}(\xi) \beta_\Gamma(\xi).
\end{equation*}
Then, we have: 
\emph{$\phi \in \V$ if and only if
\begin{equation} \label{eq: C0}
 g^{(\ell)}(0,\xi) = g^{(\ell')}(0,\xi),
\end{equation}
\begin{equation} \label{eq: C1}
 \alpha_\Gamma^{(\ell')}(\xi) D_{\xi_1^{(\ell)}} g^{(\ell)}(0,\xi) - \alpha_\Gamma^{(\ell)}(\xi) D_{\xi_1^{(\ell')}} g^{(\ell')}(0,\xi) + \beta_\Gamma(\xi) D_{\xi} g^{(\ell)}(0,\xi) = 0,
\end{equation}
and
\begin{equation} \label{eq: C2}
 \widehat{\alpha}_\Gamma^{(\ell)}(\xi) w_\Gamma(\xi) + \eta_\Gamma(\xi) D_{\xi_1^{(\ell)}} g^{(\ell)}(0,\xi) + \theta_\Gamma(\xi) D_{\xi}g^{(\ell)}(0,\xi) = 0,
\end{equation}
where
\begin{align*}
 w_\Gamma(\xi) =& \;( \alpha_\Gamma^{(\ell)}(\xi))^{2} D_{\xi_1^{(\ell')}\xi_1^{(\ell')}} g^{(\ell')}(0,\xi) - ( (\alpha_\Gamma^{(\ell')}(\xi))^{2} D_{\xi_1^{(\ell)} \xi_1^{(\ell)}} g^{(\ell)}(0,\xi) \nonumber\\[-0.3cm]
 & \label{eq:w} \\[-0.3cm]
 & + 2 \alpha_\Gamma^{(\ell')}(\xi) \beta_\Gamma(\xi) D_{\xi_1^{(\ell)} \xi} g^{(\ell)}(0,\xi) + (\beta_\Gamma(\xi))^{2}  D_{\xi \xi} g^{(\ell)}(0,\xi) ) . \nonumber
\end{align*}
}

Note that  condition~\eqref{eq: C0} guarantees that $\phi$ is $C^{0}$-smooth, condition~\eqref{eq: C1} additionally ensures that $\phi$ is $C^{1}$-smooth, and 
condition~\eqref{eq: C2} finally implies that $\phi$ is $C^{2}$-smooth. 

\begin{rem}
Below, we choose $\gamma_2(\xi) = 1$, implying $\widehat{\alpha}^{(\tau)} = \bar{\alpha}^{(\tau)} $, $\tau \in \{\ell, \ell'\}$. 
In addition, we select $\gamma_1(\xi) = c_1 \in \R$ 
such that 
\[
 || \alpha_\Gamma^{(\ell)}+1 ||^2_{L^2} + || \alpha_\Gamma^{(\ell')}-1 ||^2_{L^2}
\]
is minimized, cf. \cite{KaSaTa17c}. 
\end{rem}

\subsection{The two-patch case} \label{subsec:two-patch-case}

In this subsection we restrict ourselves to the two-patch case~$\Omega = \Omega^{(\ell)} \cup \Omega^{(\ell')}$ for two neighboring patches~$\Omega^{(\ell)}$ and $\Omega^{(\ell')}$ 
having the common interface~$\Gamma^{(s)} = \Omega^{(\ell)} \cap \Omega^{(\ell')}$. Without loss of generality, we can assume that the two geometry mappings~$\ab{F}^{(\ell)}$ and 
$\ab{F}^{(\ell')}$ are parameterized as in Fig.~\ref{fig:situation_two_patches}. We recall now the construction of  a $C^2$-smooth 
isogeometric spline space $\widetilde{\mathcal{W}}_h \subseteq \V$, which was described in \cite{KaVi17c}, by using now adapted notations. The subspace $\widetilde{\mathcal{W}}_h$ 
is advantageous compared to the entire space~$\V$, since 
its basis construction is simpler and works uniformly for all possible configurations. In addition, it was numerically demonstrated in \cite{KaVi17c} that already the 
subspace~$\widetilde{\mathcal{W}}_h$ possesses optimal approximation properties. For a detailed investigation of the spaces~$\widetilde{\mathcal{W}}_h$ and $\V$ we refer to~\cite{KaVi17c}. 

The space~$\widetilde{\mathcal{W}_{h}}$ is the direct sum of three subspaces, i.e., 
$$
\widetilde{\mathcal{W}}_h = \mathcal{\widetilde{W}}_{h;\Omega^{(\ell)}} \oplus \mathcal{\widetilde{W}}_{h;\Omega^{(\ell')}} \oplus \mathcal{\widetilde{W}}_{h; \Gamma^{(s)}}.
$$
The subspaces $\mathcal{\widetilde{W}}_{h;\Omega^{(\ell)}}$ and $\mathcal{\widetilde{W}}_{h;\Omega^{(\ell')}}$
are given by
\begin{align*}
\mathcal{\widetilde{W}}_{h;\Omega^{(\tau)}} &= \Span \{ \widetilde{\phi}_{\Omega^{(\tau)};i,j} |\; i=3,4,\ldots,p+k(p-r),\; j=0,1,\ldots, p+k(p-r)\},\quad \tau \in \{\ell,\ell'\},
\end{align*}
 with the functions
\begin{equation}  \label{eq:PhiOmega}
\widetilde{\phi}_{\Omega^{(\tau)};i,j}(\bfm{x})  = 
\begin{cases}
   (N_{i,j}^{p,r}\circ (\ab{F}^{(\tau)})^{-1})(\bfm{x}) \;
\mbox{ if }\f \, \bfm{x} \in\Omega^{(\tau)},
\\ 0 \quad \mbox{ if }\f \, \bfm{x} \in \Omega \backslash \Omega^{(\tau)}.
\end{cases} 
\end{equation}
 In order to define the subspace~$\mathcal{\widetilde{W}}_{h;\Gamma^{(s)}}$, we need some additional definitions. 
Let
\begin{equation}  \label{eq:Mj}
  M_0(\xi) = \sum_{i=0}^2 N_i^{p,r}(\xi) , \;
  M_1(\xi) = \frac{ h}{p} \left( N_1^{p,r}(\xi) + 2 N_2^{p,r}(\xi)  \right), \;
  M_2(\xi) = \frac{h^2}{p(p-1)} N_2^{p,r}(\xi), 
\end{equation}  
and let 
\begin{equation*}  \label{eq:ni}
 n_0 = \dim\left( \mathcal{S}_{h}^{p,r+2}([0,1])\right),\;
 n_1 = \dim \left( \mathcal{S}_{ h}^{p-d_\alpha,r+1}([0,1])\right),\;
 n_2 = \dim \left( \mathcal{S}_{h}^{p-2d_\alpha,r}([0,1])\right),
\end{equation*}
where $d_\alpha = \max\left( \deg(\alpha_{\Gamma^{(s)}}^{(\ell)}), \deg(\alpha_{\Gamma^{(s)}}^{(\ell')})\right) \in \{0,1\}$.
The space $\widetilde{\mathcal{W}}_{h;\Gamma^{(s)}}$ is given by 
\begin{equation*}  \label{eq:edgeSubspaceTwoPatch}
\widetilde{\mathcal{W}}_{h;\Gamma^{(s)}} = \Span \{ \widetilde{\phi}_{\Gamma^{(s)};i,j} | \; i=0,1,2, \;  j=0,1,\ldots,n_i-1\}, 
\end{equation*}
with the functions 
\begin{equation} 
\widetilde{\phi}_{\Gamma^{(s)};i,j}(\bfm{x})  = 
\begin{cases}
  (g_{\Gamma^{(s)}; i,j}^{(\ell)} \circ (\ab{F}^{(\ell)})^{-1})(\bfm{x}) \;
\mbox{ if }\f \, \bfm{x} \in\Omega^{(\ell)},
\\[0.15cm] 
  (g_{\Gamma^{(s)}; i,j}^{(\ell')} \circ (\ab{F}^{(\ell')})^{-1} )(\bfm{x}) \;
\mbox{ if }\f \, \bfm{x} \in\Omega^{(\ell')},
\end{cases}
 \label{eq:basisFunctionsGenericEdge}
\end{equation}
where
\begin{align}  \label{eq:basisFunctionsGenericG}
\scalebox{0.93}{$ 
g_{\Gamma^{(s)}; 0,j}^{(\tau)} (\bfm{\xi}^{(\tau)})$} & \scalebox{0.93}{$= 
N_j^{p,r+2}(\xi_2^{(\tau)}) M_0(\xi_1^{(\tau)}) + \beta_{\Gamma^{(s)}}^{(\tau)}(\xi_2^{(\tau)}) (N_j^{p,r+2})'(\xi_2^{(\tau)})  M_1(\xi_1^{(\tau)})  $} \nonumber \\
& \; + \left( \beta^{(\tau)}_{\Gamma^{(s)}}(\xi_2^{(\tau)})\right)^2 (N_j^{p,r+2})''(\xi_2^{(\tau)})  M_2(\xi_1^{(\tau)}),\nonumber \\
\scalebox{0.93}{$  
g_{\Gamma^{(s)}; 1,j}^{(\tau)} (\bfm{\xi}^{(\tau)})$} & \scalebox{0.93}{$ = \displaystyle \frac{p}{ h}  \left( {\alpha}^{(\tau)}_{\Gamma^{(s)}}(\xi_2^{(\tau)}) {N}_j^{p-d_{{\alpha}},r+1}(\xi_2^{(\tau)})  M_1(\xi_1^{(\tau)})  \right. $}  \\
& \left. \; + 2\,  {\alpha}^{(\tau)}_{\Gamma^{(s)}} (\xi_2^{(\tau)})\beta^{(\tau)}_{\Gamma^{(s)}}(\xi_2^{(\tau)}) ({N}_j^{p-d_{{\alpha}},r+1})'(\xi_2^{(\tau)})  M_2(\xi_1^{(\tau)})\right), \nonumber \\
\scalebox{0.93}{$
 g_{\Gamma^{(s)}; 2,j}^{(\tau)}  (\bfm{\xi}^{(\tau)})$} & \scalebox{0.93}{$=  \displaystyle  \frac{p(p-1)}{h^2} \left( {\alpha}^{(\tau)}_{\Gamma^{(s)}}(\xi_2^{(\tau)})\right)^2 
{N}_j^{p-2d_{{\alpha}},r}(\xi_2^{(\tau)}) M_2(\xi_1^{(\tau)}) $}  \nonumber,
\end{align}
for $\tau \in \{\ell,\ell'\} $.
\begin{rem}
 The functions in \eqref{eq:basisFunctionsGenericG} are scaled in comparison to the ones in \cite{KaVi17c}.
\end{rem}
The following proposition gives an estimate for the support of the function~$g_{\Gamma^{(s)}; i,j}^{(\tau)}$, and will be needed later. 
\begin{prop}  \label{thm:mainC4}
Let $d= \dim\left( \mathcal{S}_{h}^{p,r}([0,1])\right) = p + k(p-r) +1$. The functions $ g_{\Gamma^{(s)}; i,j}^{(\tau)}$, $j=0,1,\ldots,n_i-1,\, i =0,1,2,\, \tau \in \{\ell,\ell'\},$  
can be represented as 
\begin{equation}  \label{eq:defgljS}
  g_{\Gamma^{(s)}; i,j}^{(\tau)} (\bfm{\xi}^{(\tau)}) =  \sum_{m=0}^2 \; 
   \sum_{n=\max (0,i+j-m)}^{\min(d-1,d-n_i+j-i+m) } d^{(\tau)}_{\Gamma^{(s)}; m,n} N_{m,n}^{p,r} (\bfm{\xi}^{(\tau)}), \qquad d^{(\tau)}_{\Gamma^{(s)}; m,n}\in \R.
\end{equation}  
\end{prop}
\begin{proof}
 See \ref{app:proofProposition}.
\end{proof}
In the next section we will use the $C^2$-smooth isogeometric functions for the two-patch case to construct 
a $C^2$-smooth isogeometric spline space~$\W$ for the multi-patch case.

\section{$C^2$-smooth discretization space $\W$} \label{sec:C2smoothtestfunctions}

A $C^2$-smooth discretization space $\W$ will be constructed which can be used for solving the triharmonic equation \eqref{eq:triharmonic_problem} 
with homogeneous boundary conditions \eqref{eq:triharmonic_problem_boundary}, see Section~\ref{sec:triharmonic_examples}. This space will be a subspace of $\V$ or more precisely of 
the space~$\VO$ given by 
\begin{equation*} 
\VO = \{ \phi \in \V : \;  \phi(\ab{x}) = \frac{\partial \phi}{\partial \ab{n}}(\ab{x}) = \triangle \phi(\ab{x})  = 0 , \quad \ab{x} \in \partial \Omega  \}, 
\end{equation*}
which contains all $C^2$-smooth functions on $\Omega$ fulfilling the homogeneous boundary conditions~\eqref{eq:triharmonic_problem_boundary}.

\subsection{Structure of the space~$\W$}

The discretization space~$\W$ is the direct sum of smaller subspaces corresponding to the single patches~$\Omega^{(\ell)}$, 
edges~$\Gamma^{(s)}$ and vertices $\ab{v}^{(\rho)}$, i.e., 
\begin{equation} \label{eq:Wh}
    \W =  \left(\bigoplus_{\ell=1}^P \mathcal{W}_{0h;\Omega^{(\ell)}}\right) \oplus \left(\bigoplus_{s=1}^E \mathcal{W}_{0h;\Gamma^{(s)}}\right) 
    \oplus \left(\bigoplus_{\rho=1}^V \mathcal{W}_{0h;\bfm{v}^{(\rho)}}\right).
\end{equation}
This decomposition is a common strategy to generate smooth spline spaces, e.g. \cite{KaSaTa17c, KaVi17b}.
The construction of the single subspaces will be presented in the following subsections and will be based on functions from the subspaces 
$\mathcal{\widetilde{W}}_{h;\Omega^{(\ell)}}$ and $\mathcal{\widetilde{W}}_{h;\Gamma^{(s)}}$ for the two-patch case in Section~\ref{subsec:two-patch-case}. 

\subsection{The patch subspace~$\mathcal{W}_{0h;\Omega^{(\ell)}}$}

Let $\ell \in \{1, 2, \ldots, P \}$. We denote by $\phi_{\Omega^{(\ell)};i,j}: \Omega \to \R$, ${\MK i,j=0,1,\ldots,p+k(p-r)}$, the functions 
\begin{equation} \label{eq:defphiOmega}
    \phi_{\Omega^{(\ell)};i,j} (\bfm{x}) = 
    \begin{cases}
 \widetilde{ \phi}_{\Omega^{(\ell)};i,j} (\bfm{x})\;
\mbox{ if }\f \, \bfm{x} \in\Omega^{(\ell)} ,
\\[0.15cm] 
  0\quad \mbox{otherwise},
\end{cases}
\end{equation}
with $\widetilde{\phi}_{\Omega^{(\ell)};i,j}$ given in \eqref{eq:PhiOmega}, 
{\MK and }define the patch subspace $\mathcal{W}_{0h; \Omega^{(\ell)}}$ as
\begin{equation*} \label{eq:spaceW0hOmega}
\mathcal{W}_{0h; \Omega^{(\ell)}} =  \Span \{ \phi_{\Omega^{(\ell)};i,j} |\;  i,j=3,4,\ldots,p+k(p-r)-3  \}.
\end{equation*}
\begin{lem}  \label{lem:patchspace}
We have
\[
\mathcal{W}_{0h; \Omega^{(\ell)}}  \subseteq \VO. 
\] 
\end{lem}
\begin{pf}
By \eqref{eq:defphiOmega}, the functions~$\phi_{\Omega^{(\ell)};i,j}$, $i,j=3,4,\ldots,p+k(p-r)-3$, possess a support
\[
\mbox{supp}(\phi_{\Omega^{(\ell)};i,j}) \subseteq \Omega^{(\ell)},
\]
they are clearly $C^2$-smooth on $\Omega^{(\ell)}$, and have vanishing values, gradients and Hessians on $\partial \Omega^{(\ell)}$. This implies 
that $\phi_{\Omega^{(\ell)};i,j} \in \VO$. \qed
\end{pf}

\subsection{The edge subspace~$\mathcal{W}_{0h;\Gamma^{(s)}}$}

Let $s \in \{1,2, \ldots , E\}$ and let $\ell,\ell' \in \{1,2, \ldots, P\}$, $\ell \neq \ell'$, be the corresponding indices of the two patches such that 
$\Gamma^{(s)} =  \Omega^{(\ell)} \cap \Omega^{(\ell')}$. Without loss of generality, we can assume that the two geometry mappings~$\ab{F}^{(\ell)}$ and $\ab{F}^{(\ell')}$ 
are parameterized as in Fig.~\ref{fig:situation_two_patches}. Otherwise, suitable linear reparameterizations of the two patches can be applied to fulfill this situation.

We denote by 
$\phi_{\Gamma^{(s)};i,j}: \Omega \to \R, \; i=0,1,2,\; j=0,1,\ldots, n_i-1$, the functions
\begin{equation}  \label{eq:defphiGamma}
    \phi_{\Gamma^{(s)};i,j} (\bfm{x}) = 
    \begin{cases}
 \widetilde{ \phi}_{\Gamma^{(s)};i,j} (\bfm{x})\;
\mbox{ if }\f \, \bfm{x} \in\Omega^{(\ell)} \cup \Omega^{(\ell')} ,
\\[0.15cm] 
  0\quad otherwise,
\end{cases}
\end{equation}
with $\widetilde{\phi}_{\Gamma^{(s)};i,j}$ given in \eqref{eq:basisFunctionsGenericEdge}. 
Then, the edge subspace~$\mathcal{W}_{0h;\Gamma^{(s)}}$ is defined as
\begin{equation*} \label{eq:spaceW0hGamma}
\mathcal{W}_{0h;\Gamma^{(s)}} =  \Span \{\phi_{\Gamma^{(s)};i,j}| \;\; j=5-i,6-i,\ldots,n_i+i-6;\; i=0,1,2\}.
\end{equation*} 
\begin{lem}  \label{lem:edgespace}
It holds that 
\[
\mathcal{W}_{0h; \Gamma^{(s)}}  \subseteq  \VO. 
\] 
\end{lem}
\begin{proof}
Let $ i=0,1,2$ and $j=5-i,6-i,\ldots, n_i+i-6$. By \eqref{eq:defphiGamma}, the functions $\phi_{\Gamma^{(s)};i,j}$ possess a support
 \[
  \mbox{supp}(\phi_{\Gamma^{(s)};i,j}) \subseteq \Omega^{(\ell)} \cup \Omega^{(\ell')}. 
 \]
 Furthermore, it was shown in~\cite{KaVi17c}, that the functions $\phi_{\Gamma^{(s)};i,j}$ are $C^2$-smooth on~$\Omega^{(\ell)} \cup \Omega^{(\ell')}$. Since 
 Proposition~\ref{thm:mainC4} ensures that the functions $ \phi_{\Gamma^{(s)};i,j}$ have vanishing values, gradients and Hessians on $\partial (\Omega^{(\ell)} \cup \Omega^{(\ell')} )$,
we obtain $\phi_{\Gamma^{(s)};i,j} \in \VO$.
\end{proof}

\subsection{The vertex subspace $\mathcal{W}_{0h;\bfm{v}^{(\rho)}}$}
We consider an inner or boundary vertex~$\bfm{v}^{(\rho)}$, $\rho \in \{1,2, \ldots, V \}$, possessing the valency $\bar{\nu}_{\rho} \geq 3$. 
We define $\nu_{\rho}$ as
\[
\nu_{\rho} =
\begin{cases}
   \bar{\nu}_{\rho}, \; \mbox{ if }\bfm{v}^{(\rho)}\mbox{ is an inner vertex} ,\\
\bar{\nu}_{\rho}-1, \; \mbox{ if }\bfm{v}^{(\rho)}\mbox{ is a boundary vertex} .
\end{cases}
\]
For the sake of simplicity, we relabel 
the patches containing the vertex $\ab{v}^{(\rho)}$
in counterclockwise order by $\Omega^{(0)}, \Omega^{(1)}, \ldots, \Omega^{(\nu_{\rho}-1)}$. 
Furthermore, we assume without loss of generality that the corresponding geometry mappings $\ab{F}^{(\ell)}$, $\ell=0,1,\ldots,\nu_\rho -1$, are parameterized 
as shown in Fig.~\ref{fig:situation_multi_patches}, 
which assures that 
$$
  \bfm{F}^{(0)}(\bfm{0}) = \bfm{F}^{(1)}(\bfm{0}) = \cdots = \bfm{F}^{(\nu_\rho -1)}(\bfm{0}) = \bfm{v}^{(\rho)}.
$$
\begin{figure}[htp] 
\centering
\includegraphics[width=12cm]{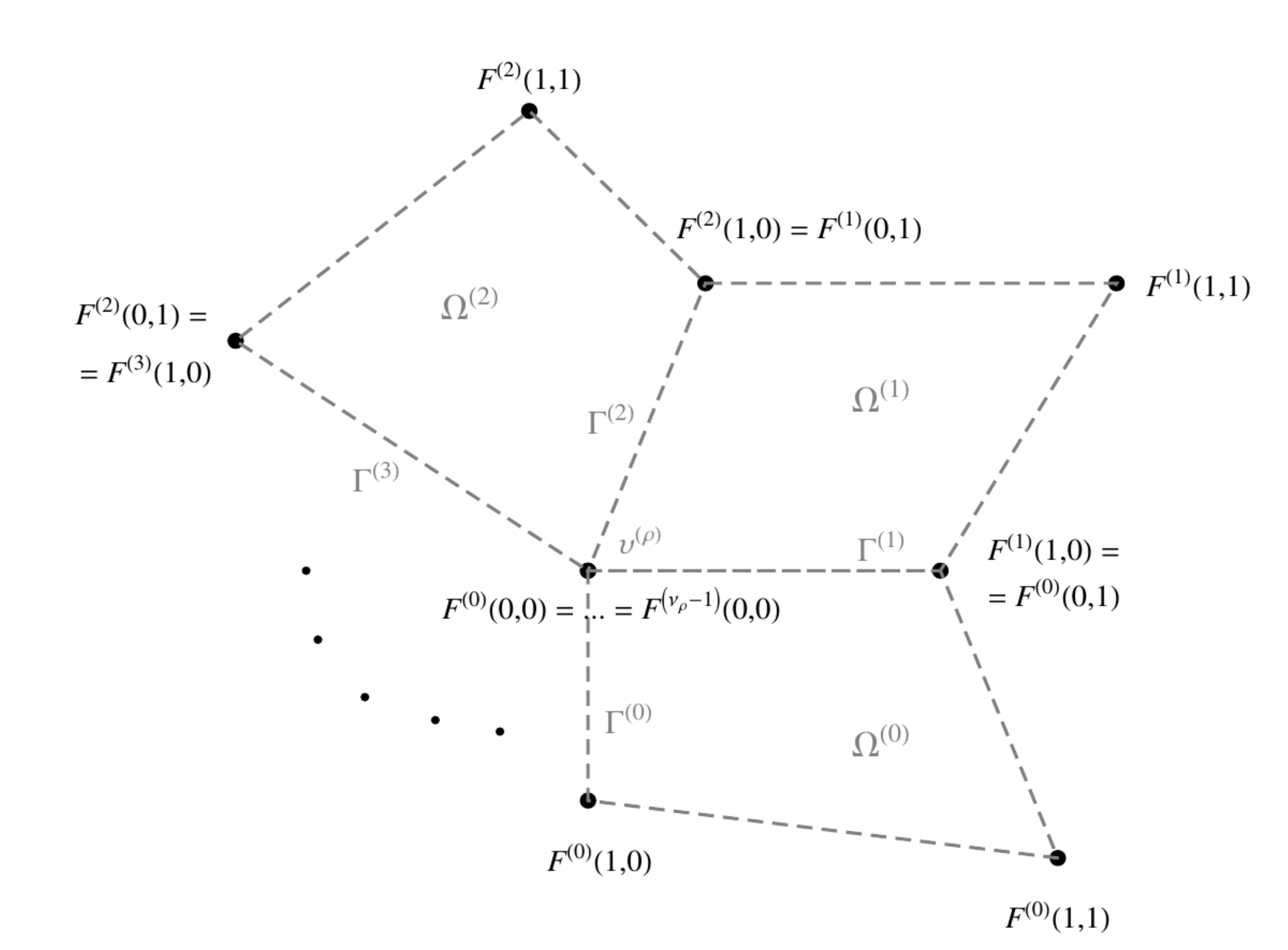}
\caption{The geometry mappings~$\ab{F}^{(\ell)}$ of the patches~$\Omega^{(\ell)}$, $\ell =0, 1,\ldots, \nu_{\rho} -1$, which contain the vertex~$\bfm{v}^{(\rho)}$, can be always 
reparameterized as shown.} 
\label{fig:situation_multi_patches}
\end{figure}
Moreover, we relabel the common interface of every two-patch subdomain
$
\Omega^{(\ell)} \cup \Omega^{(\ell+1)},  \, \ell=0,1,\ldots,\nu_\rho -1, $
by $\Gamma^{(\ell+1)}$. In case of an inner vertex~$\bfm{v}^{(\rho)}$, we consider the upper index~$\ell$ of $\Omega^{(\ell)}$ and $\Gamma^{(\ell)}$ modulo~$\nu_{\rho}$, and in 
case of a boundary vertex~$\bfm{v}^{(\rho)}$, we denote by $\Gamma^{(0)}$ the edge of $\Omega^{(0)}$ corresponding to 
$\ab{F}^{(0)}([0,1]\times \{0\})$, and by $\Gamma^{(\nu_{\rho})}$  the edge of $\Omega^{(\nu_{\rho}-1)}$ corresponding to 
$\ab{F}^{(\nu_{\rho}-1)}(\{0 \} \times [0,1])$. In addition, we denote by $\bar{\bfm{\xi}}^{(\ell)} $ the pair of parameters $\bar{\bfm{\xi}}^{(\ell)} = 
({\xi}_2^{(\ell)},\xi_1^{(\ell)})$.

The idea is to construct the vertex subspace $\mathcal{W}_{0h;\bfm{v}^{(\rho)}}$ as the space of functions which can be represented by suitable linear 
combinations of functions $\phi_{\Omega^{(\ell)};i,j}$, ${\MK 0 \leq i,j \leq 2},\, 0\leq\ell\leq {\nu}_\rho-1,$ and of functions $\phi_{\Gamma^{(\ell)};i,j}$, 
$0\leq i\leq 2,\, 0\leq j\leq 4-i,\, 0 \leq\ell\leq \bar{\nu}_\rho-1$. Note that none of these functions are contained in any of the spaces 
$\mathcal{W}_{0h;\Omega^{(\ell)}}$ and $\mathcal{W}_{0h;\Gamma^{(s)}}$. {\MK Furthermore, these are exactly those functions~$\phi_{\Omega^{(\ell)};i,j}$ and $\phi_{\Gamma^{(\ell)};i,j}$, 
which are involved in the continuity constraints at the vertex, since they can possess nonzero spline coefficients (with respect to the representation~\eqref{eq:g_ell}), 
which are affected by the $C^2$-continuity conditions of more than one edge~$\Gamma^{(\ell)}$. {\VV These corresponding spline coefficients are the ones in the grey region in 
Fig.~\ref{fig:supports} }}. Recall that the functions $\phi_{\Gamma^{(\ell)};i,j}$ are $C^2$-smooth on the two-patch subdomain~$\Omega^{(\ell-1)} \cup \Omega^{(\ell)}$.

For each patch $\Omega^{(\ell)}$, $\ell=0,1,\ldots, {\nu}_{\rho} -1$, we define the function $f_{\ell} : [0,1]^2 \rightarrow \R$ as
\[
f_{\ell}  (\bfm{\xi}^{(\ell)}) = f_{\ell}^{\Gamma^{(\ell)}}  (\bfm{\xi}^{(\ell)}) +  f_{\ell}^{\Gamma^{(\ell+1)}}  
(\bfm{\xi}^{(\ell)}) -  f_{\ell}^{\Omega^{(\ell)}}  (\bfm{\xi}^{(\ell)}),
\]
where the functions $f_{\ell}^{\Gamma^{(\ell)}} , f_{\ell}^{\Gamma^{(\ell+1)}}, f_{\ell}^{\Omega^{(\ell)}} : [0,1]^2 \to \R$ are given by
{\VV
\begin{align*}
  f_{\ell}^{\Gamma^{(\ell)}} (\bfm{\xi}^{(\ell)}) & = \sum_{i=0}^2 \sum_{j=0}^{4-i}  a^{\Gamma^{(\ell)}}_{i,j} \, 
  g_{\Gamma^{(\ell); i,j}}^{(\ell)} (\bar{\bfm{\xi}}^{(\ell)})  
, \nonumber \\
  f_{\ell}^{\Gamma^{(\ell+1)}} (\bfm{\xi}^{(\ell)}) & = \sum_{i=0}^2 \sum_{j=0}^{4-i}  a^{\Gamma^{(\ell+1)}}_{i,j} \, 
  g_{\Gamma^{(\ell+1); i,j}}^{(\ell)} (\bfm{\xi}^{(\ell)}),  
\\ \label{eq:functionsonpatch} 
   f_{\ell}^{\Omega^{(\ell)}} (\bfm{\xi}^{(\ell)}) & = \sum_{i=0}^2 \sum_{j=0}^{2} a^{(\ell)}_{i,j} 
   N_{i,j}^{p,r} (\bfm{\xi}^{(\ell)}) , \nonumber 
\end{align*}
}
with $a^{\Gamma^{(\ell)}}_{i,j}, a^{\Gamma^{(\ell+1)}}_{i,j}, a_{i,j}^{(\ell)} \in \R$.
Furthermore, we define the function $\phi_{\bfm{v}^{(\rho)}} : \Omega \to \R$ as
\begin{equation}  \label{eq:defphiXi}
  \phi_{\bfm{v}^{(\rho)}} (\bfm{x}) = 
  \begin{cases}
   (f_{\ell} \circ (\ab{F}^{(\ell)})^{-1})(\bfm{x}) \;
\mbox{ if }\f \, \bfm{x} \in \Omega^{(\ell)},\; \ell=0,1,\ldots,{\nu}_\rho -1,
\\
0 \quad \mbox{ otherwise}.
\end{cases}
\end{equation}

The idea for the construction of the function $\phi_{\bfm{v}^{(\rho)}} $ is as follows. On each patch $\Omega^{(\ell)}$, $\ell=0,1,\ldots, {\nu}_{\rho} -1$, the 
function~$ \phi_{\bfm{v}^{(\rho)}} $ 
is determined by the spline function~$f_{\ell}$, where the sum of the functions $f_{\ell}^{\Gamma^{(\ell)}}$ and $f_{\ell}^{\Gamma^{(\ell+1)}}$ should ensure $C^2$-smoothness across the 
interfaces~$\Gamma^{(\ell)}$ and $\Gamma^{(\ell+1)}$, and the function~$f_{\ell}^{\Omega^{(\ell)}}$ is used to subtract those B-splines $N_{i,j}^{p,r}$ (with respect to the spline 
space $\mathcal{S}^{p,r}_{h}([0,1]^2)$), which have been added twice, see Fig.~\ref{fig:supports}. 

\begin{figure}[htp] 
\centering
\includegraphics[width=15cm]{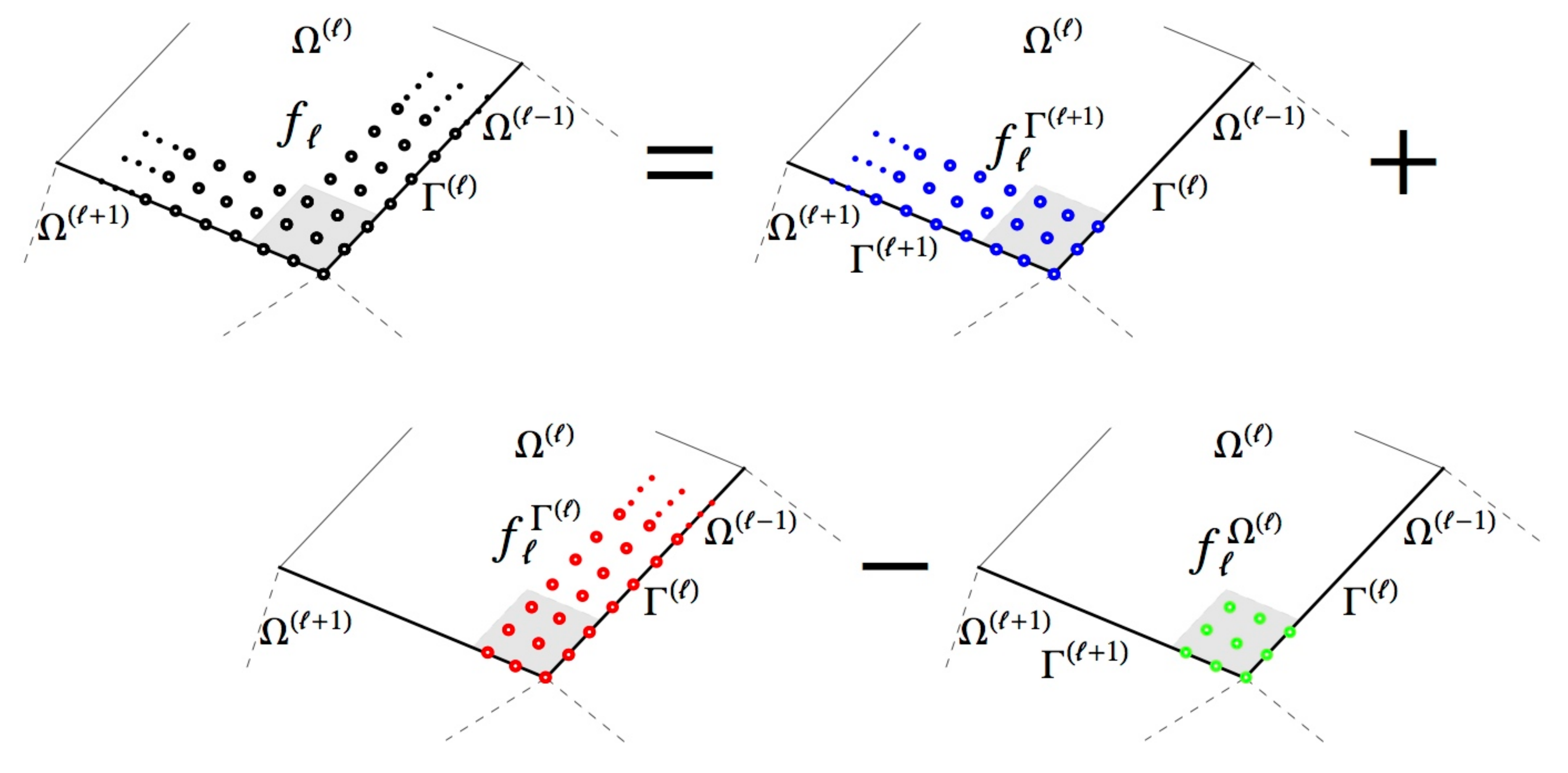}
\caption{On each patch $\Omega^{(\ell )}$, the function~$f_{\ell}$ is obtained by summing up the two functions $f_{\ell}^{\Gamma^{(\ell)}}$ and $f_{\ell}^{\Gamma^{(\ell+1)}}$ and 
by subtracting the function $f_{\ell}^{\Omega^{(\ell)}}$. The nonzero spline coefficients of the single functions {\MK with respect to the spline representation~\eqref{eq:g_ell}} 
are visualized in different colors. To ensure that the function $\phi_{\bfm{v}^{(\rho)}}$ is $C^2$-smooth on $\Omega$, the values of the corresponding spline control points of the 
functions $f_{\ell}^{\Gamma^{(\ell)}}$, $f_{\ell}^{\Gamma^{(\ell+1)}}$ and $f_{\ell}^{\Omega^{(\ell)}}$ in the grey regions have to coincide (compare Lemma~\ref{lem:C2smoothvertex}).} 
\label{fig:supports}
\end{figure}

Clearly, not any choice of the 
coefficients~$a^{\Gamma^{(\ell)}}_{i,j}$ and $a_{i,j}^{(\ell)}$, $ \ell=0,1,\ldots,{\nu}_\rho-1$,  guarantees $\phi_{\bfm{v}^{(\rho)}} \in \V$ or as needed in our case 
even~$\phi_{\bfm{v}^{(\rho)}} 
\in \VO$. The following lemma characterizes when the function~$\phi_{\bfm{v}^{(\rho)}}$ belongs to the space~$\VO$:
\begin{lem} \label{lem:C2smoothvertex}
 $\phi_{\bfm{v}^{(\rho)}} \in \VO$ if the corresponding functions $f_{\ell}^{\Gamma^{(\ell)}}$, 
 $f_{\ell}^{\Gamma^{(\ell+1)}}$, and $f_{\ell}^{\Omega^{(\ell)}}$, $\ell =0,1, \ldots,$  $\nu_{\rho}-1$, satisfy
 \begin{equation}  \label{eq:system1}
   \partial_{\xi_1^{(\ell)}}^i  
   \partial_{\xi_2^{(\ell)}}^j 
   \left( f_{\ell}^{\Gamma^{(\ell+1)}} - f_{\ell}^{\Gamma^{(\ell)}} \right)  (\bfm{0}) = 0, \quad 
   {\VV 0 \leq i,j \leq 2},
\end{equation}
 and
 \begin{equation}  \label{eq:system2}
   \partial_{\xi_1^{(\ell)}}^i  
   \partial_{\xi_2^{(\ell)}}^j 
   \left( f_{\ell}^{\Gamma^{(\ell+1)}} - f_{\ell}^{\Omega^{(\ell)}} \right)  (\bfm{0}) = 0, \quad 
   0 \leq i,j \leq 2,  
\end{equation}
and in case of a boundary vertex~$\bfm{v}^{(\rho)}$, additionally
\begin{equation} \label{eq:system3}
 a_{i,j}^{\Gamma^{(0)}} = 0 , \mbox{ and } a_{i,j}^{\Gamma^{({\nu}_{\rho})}} = 0, \quad 0 \leq i \leq 2, \, 0\leq j\leq 4-i,
\end{equation}
and
\begin{equation} \label{eq:system4}
 a_{i,j}^{(\ell)} = 0 , \quad {\MK 0 \leq i,j \leq 2, \, 0 \leq \ell \leq \nu_{\rho}-1}.
\end{equation}
\end{lem}
\begin{proof}
By \eqref{eq:defphiXi}, the function~$\phi_{\bfm{v}^{(\rho)}}$ possesses a support
 \[
  \mbox{supp}({\phi_{\bfm{v}^{(\rho)}}}) \subseteq \cup_{\ell=0}^{{\nu}_{\rho} -1} \Omega^{(\ell)}. 
 \]
Equations~\eqref{eq:system1} and \eqref{eq:system2} ensure that the coefficients~$a^{\Gamma^{(\ell)}}_{i,j}$ and $a_{i,j}^{(\ell)}$ are well-defined, which implies that the 
 function~$\phi_{\bfm{v}^{(\rho)}}$ is well-defined. The function~$\phi_{\bfm{v}^{(\rho)}}$ is now $C^2$-smooth across the interfaces~$\Gamma^{(\ell)}$, since its values, 
 gradients and Hessians along the interfaces~$\Gamma^{(\ell)}$ are given by
 \[
  \phi_{\bfm{v}^{(\rho)}}(\Gamma^{(\ell)}) = \sum_{i=0}^2 \sum_{j=0}^{4-i}  a^{\Gamma^{(\ell)}}_{i,j} \, \phi_{\Gamma^{(\ell)};i,j}(\Gamma^{(\ell)}),
 \]
 \[
 \nabla  \phi_{\bfm{v}^{(\rho)}}(\Gamma^{(\ell)}) = \sum_{i=0}^2 \sum_{j=0}^{4-i}  a^{\Gamma^{(\ell)}}_{i,j} \, \nabla \phi_{\Gamma^{(\ell)};i,j}(\Gamma^{(\ell)})
 \]
 and
  \[
 \mbox{Hess} (\phi_{\bfm{v}^{(\rho)}})(\Gamma^{(\ell)}) = \sum_{i=0}^2 \sum_{j=0}^{4-i}  a^{\Gamma^{(\ell)}}_{i,j} \, \mbox{Hess}(\phi_{\Gamma^{(\ell)};i,j})(\Gamma^{(\ell)}),
 \]
respectively, Finally, we obtain  $\phi_{\bfm{v}^{(\rho)}} \in \VO$, since Proposition~\ref{thm:mainC4} and equations~\eqref{eq:system3} {\MK and \eqref{eq:system4}} (in case of a 
boundary vertex ~$\bfm{v}^{(\rho)}$) ensure that the function $\phi_{\bfm{v}^{(\rho)}}$ has vanishing values, gradients and Hessians already on the boundary of the multi-patch subdomain 
$\cup_{\ell=0}^{\nu_{\rho}-1} \Omega^{(\ell)}$.
\end{proof}

The equations~\eqref{eq:system1} and \eqref{eq:system2}, and additionally equations~\eqref{eq:system3} {\MK and \eqref{eq:system4}} in case of a boundary vertex~$\bfm{v}^{(\rho)}$, 
form a homogeneous linear system
\begin{equation} \label{eq:whole_system}
 H^{(\rho)} \ab{a}^{(\rho)} = \ab{0},
\end{equation}
where $\ab{a}^{(\rho)}$ is the vector of all involved coefficients $a_{i,j}^{\Gamma^{(\ell)}}$ and $a_{i,j}^{(\ell)}$.
Any basis of the null space (i.e., the kernel) of the matrix~$H^{(\rho)}$, determines $ \dim (\ker H^{(\rho)})$ linearly independent functions~$\phi_{\bfm{v}^{(\rho)}} \in \VO$, which 
will be denoted by $\phi_{\bfm{v}^{(\rho)};m}$, $m=1,2,\ldots, \dim (\ker H^{(\rho)})$. One possible strategy is to find a basis by constructing minimal determining sets 
(cf.~\cite{BeMa14, LaSch07}) for the unknown coefficients of the homogeneous linear system~\eqref{eq:whole_system}. In our examples in Section~\ref{sec:triharmonic_examples}, we 
use the minimal determining set algorithm introduced in \cite[Section 6.1]{KaVi17a}, 
which works well and yields well-conditioned functions, {\MK cf. Examples~\ref{ex:example1} and \ref{ex:example2}}.

Finally, the vertex subspace $\mathcal{W}_{0h;\bfm{v}^{(\rho)}}$ is defined as
\begin{equation*}
 \mathcal{W}_{0h;\bfm{v}^{(\rho)}} = \Span \{ \phi_{\bfm{v}^{(\rho)};m}\,  |\; m = 1,2,\ldots, \dim (\ker H^{(\rho)})  \}.
\end{equation*}

\begin{lem}  \label{lem:vertexspace}
We have
\[
 \mathcal{W}_{0h;\bfm{v}^{(\rho)}}   \subseteq \VO. 
\] 
\end{lem} 
\begin{proof}
Recall \eqref{eq:defphiXi}. The functions $\phi_{\bfm{v}^{(\rho)};m}$, $m=1,2,\ldots, \dim (\ker H^{(\rho)})$, are constructed in such a way that they satisfy $\phi_{\bfm{v}^{(\rho)};m} \in \VO$.
\end{proof}

\begin{rem}
A further possible way for the computation of suitable vertex subspaces could be the extension of the method~\cite{KaSaTa17c} proposed for the case of $C^1$-smooth isogeometric 
functions to our case of $C^2$-smooth isogeometric functions. In~\cite{KaSaTa17c}, the vertex subspace is defined by globally $C^1$-smooth functions which are $C^2$-smooth at the vertex. 
However, the extension of this approach to our case would require globally $C^2$-smooth functions which have to be $C^4$-smooth at the vertex.
\end{rem}

\subsection{The space~$\W$}
Recall that the space~$\W$ is the direct sum~\eqref{eq:Wh}. 
\begin{thm} \label{thm:discretizationspace}
It holds that
\[
 \W \subseteq \VO,
\]
and the collection of functions
\begin{align}  \label{eq:setsOfFunctions} 
&\phi_{\Omega^{(\ell)};i,j} , \quad  i,j=3,4,\ldots,p+k(p-r)-3,\; \ell=1,2,\ldots,P, \nonumber \\
&\phi_{\Gamma^{(s)};i,j} , \quad  i=0,1,2,\;  {\MK j=5-i,6-i,\ldots,n_{i}+i-6,} \;s=1,2,\ldots,E,\\ 
&\phi_{\bfm{v}^{(\rho)};m}, \quad  m=1,2,\ldots, \dim (\ker H^{(\rho)}), \; \rho=1,2,\ldots,V, \nonumber
\end{align} 
forms a basis of the space $\W$.
\end{thm}
\begin{proof}
$\W \subseteq \VO$ is a direct consequence of Lemma~\ref{lem:patchspace}, \ref{lem:edgespace} and \ref{lem:vertexspace}, and the definition of the space~$\W$, 
see~\eqref{eq:Wh}. {\MK By construction, the collection of functions~\eqref{eq:setsOfFunctions} spans the space~$\W$, and all functions are linearly independent. The latter 
property follows directly from the following tree facts. First, the functions $\phi_{\Omega^{(\ell)};i,j}$, $\phi_{\Gamma^{(s)};i,j}$ and $\phi_{\bfm{v}^{(\rho)};m}$ are linearly 
independent in their particular sets. Second, the selected functions~$\phi_{\Omega^{(\ell)};i,j}$ do not have a common set of nonzero coefficients with the corresponding 
functions~$\phi_{\Gamma^{(s)};i,j}$ and $\phi_{\bfm{v}^{(\rho)};m}$ with respect to spline representation~\eqref{eq:g_ell}. Third, the functions~$\phi_{\bfm{v}^{(\rho)};m}$ 
are linear combinations only of functions~ $\phi_{\Omega^{(\ell)};i,j}$ and $\phi_{\Gamma^{(s)};i,j}$, which are not contained in any of the spaces~$\mathcal{W}_{0h;\Omega^{(\ell)}}$
and $\mathcal{W}_{0h;\Gamma^{(s)}}$.}
\end{proof}

\begin{rem}
 The functions $\phi_{\Omega^{(\ell)};i,j}$, $\phi_{\Gamma^{(s)};i,j}$ and $\phi_{\bfm{v}^{(\rho)};m}$ are called patch, edge and vertex functions, respectively. All these functions 
 possess a small local support, and are obtained by {\MK computing the null space of} a small system of linear equations and/or by simple explicit formulae. 
 {\MK The patch functions~$\phi_{\Omega^{(\ell)};i,j}$ are just the ``standard'' isogeometric functions 
 {\VV whose supports are contained} in one patch {\VV only}. The small, local supports of 
 the edge and vertex functions are contained in two or in at least two patches, respectively. More precisely, the edge functions~$\phi_{\Gamma^{(s)};i,j}$  {\VV have their supports contained} 
 in a small 
 region across the common interface, and the vertex functions~$\phi_{\bfm{v}^{(\rho)};m}$ possess a support in the vicinity of the vertex. While, the edge functions interpolate 
 values and specific first and second derivatives along the common interface, cf.~\cite{KaVi17c}, the vertex functions are just built up from functions~$\phi_{\Omega^{(\ell)};i,j}$ 
 and $\phi_{\Gamma^{(s)};i,j}$, which are not contained in any patch subspace~$\mathcal{W}_{0h;\Omega^{(\ell)}}$ and in any edge subspace~$\mathcal{W}_{0h;\Gamma^{(s)}}$, respectively.} 

 By means of interpolation, {\MK the edge and vertex functions, or more precisely, their} spline functions  $\phi_{\Gamma^{(s)};i,j} \circ \ab{F}^{(\ell)}$ and 
 $\phi_{\bfm{v}^{(\rho)};m} \circ \ab{F}^{(\ell)}$ can be represented as {\MK a linear combination of} the spline functions $\phi_{\Omega^{(\ell)};i,j} \circ \ab{F}^{(\ell)}$, 
 {\MK i.e.} with respect to the spline representation~\eqref{eq:g_ell} (compare e.g., \cite{KaVi17c}).
\end{rem}

\begin{ex} \label{ex:functions}
We consider the three-patch domain~(a) visualized in Fig.~\ref{fig:example1}~(first row). The space~$\W$ is defined as
\[ 
\W =  \left(\bigoplus_{\ell=1}^3 \mathcal{W}_{0h;\Omega^{(\ell)}}\right) \oplus \left(\bigoplus_{s=1}^3 \mathcal{W}_{0h;\Gamma^{(s)}}\right) 
    \oplus \left(\bigoplus_{\rho=1}^4 \mathcal{W}_{0h;\bfm{v}^{(\rho)}}\right)
\]
with the vertices $\bfm{v}^{(1)}=(\frac{17}{3},2)$, $\bfm{v}^{(2)}=(\frac{35}{4},\frac{15}{7})$, $\bfm{v}^{(3)}=(\frac{13}{3},4)$ and $\bfm{v}^{(4)}=(5,0)$, 
and the edges $\Gamma^{(1)}=\Omega^{(1)} \cap \Omega^{(2)}$, $\Gamma^{(2)}=\Omega^{(2)} \cap \Omega^{(3)}$ and $\Gamma^{(3)}=\Omega^{(3)} \cap \Omega^{(1)}$. 
For $p=5$, $r=2$ and $h=\frac{1}{6}$, the dimensions of the single subspaces are given by 
\[
\dim \mathcal{W}_{0h;\Omega^{(\ell)}} = 225, \mbox{ } \dim \mathcal{W}_{0h;\Gamma^{(s)}} =6 , \mbox{ } \dim \mathcal{W}_{0h;\bfm{v}^{(1)}} = 16 \mbox{ and } 
\dim \mathcal{W}_{0h;\bfm{v}^{(\rho)}} = 3,
\]
for $s,\ell=1,2,3$ and $\rho=2,3,4$. Furthermore, the functions of the edge space 
$\mathcal{W}_{0h;\Gamma^{(1)}}$ and the functions of the vertex spaces $\mathcal{W}_{0h;\bfm{v}^{(1)}}$ and $\mathcal{W}_{0h;\bfm{v}^{(2)}}$ are 
shown in Fig.~\ref{fig:ex_functions_edge} and Fig.~\ref{fig:ex_functions_vertex}, respectively. {\MK Recall that the functions of the edge spaces 
are determined by the explicit representation~\eqref{eq:basisFunctionsGenericG}, and that the functions of the vertex spaces are defined via appropriate  
bases of the null spaces of the corresponding homogeneous linear systems~\eqref{eq:whole_system}, which are computed by means of the minimal determining set 
algorithm~\cite[Section 6.1]{KaVi17a}.}

\begin{figure}[htp]
\centering\footnotesize
\begin{tabular}{ccc}
\includegraphics[width=4.0cm,clip]{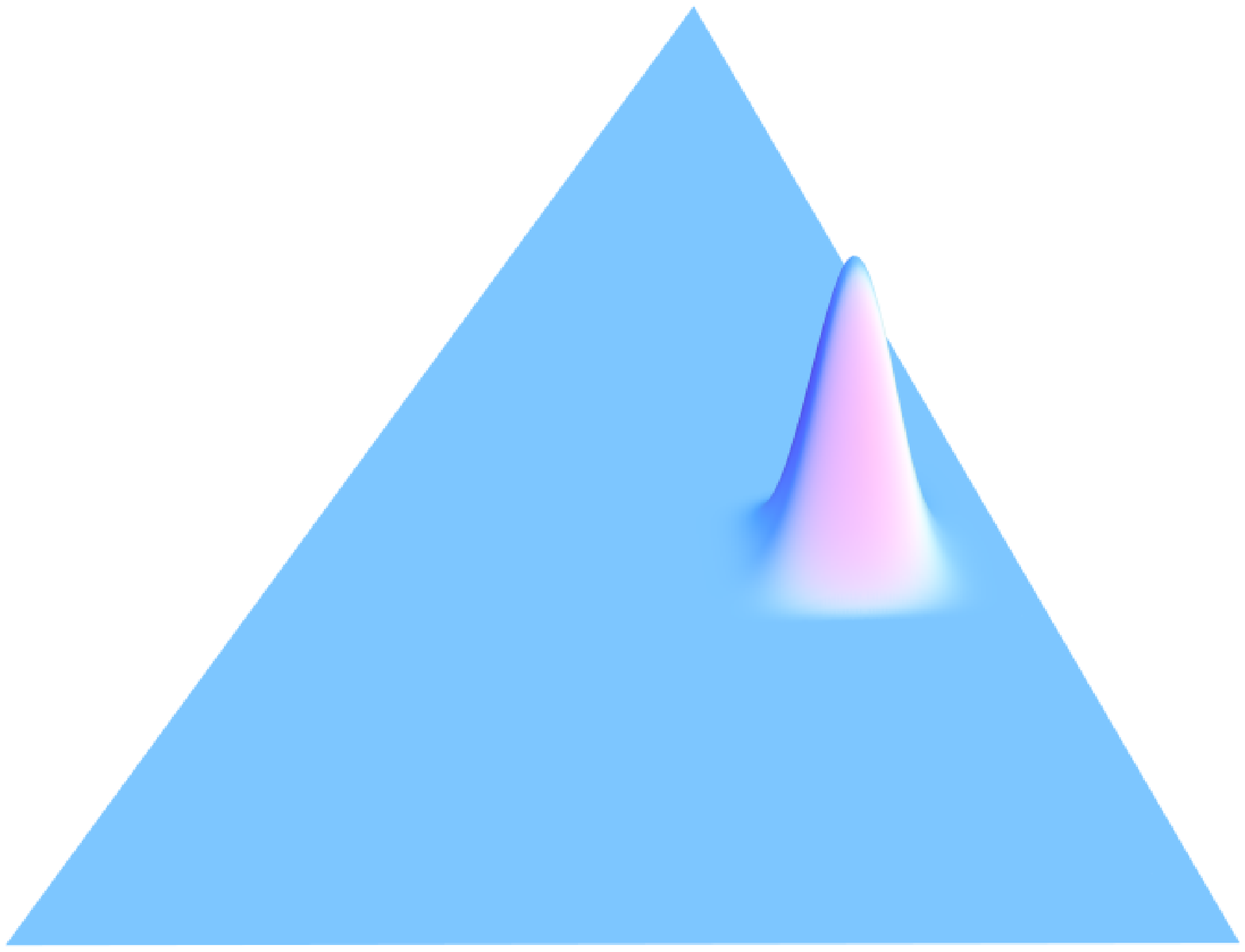} &
\includegraphics[width=4.0cm,clip]{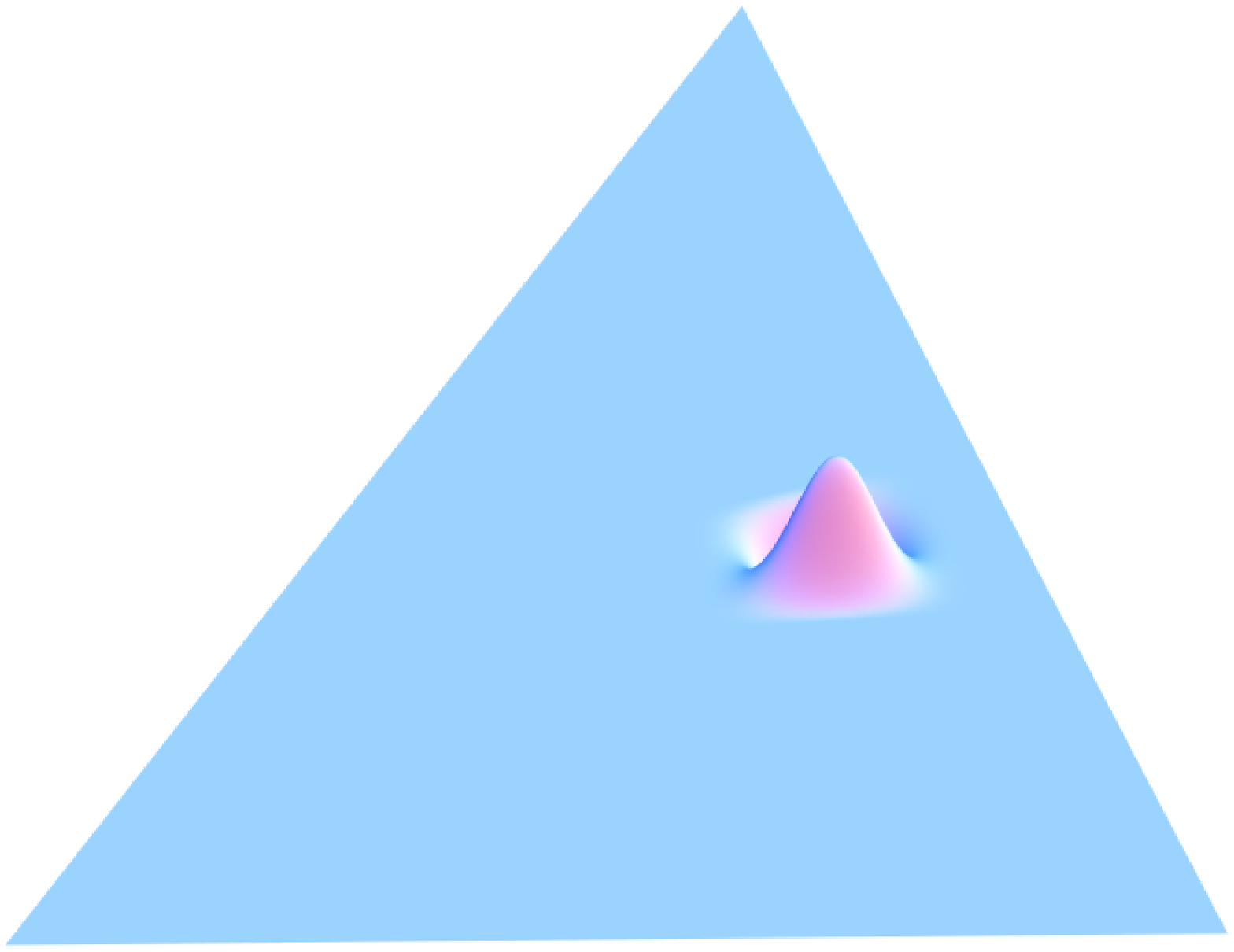} &
\includegraphics[width=4.0cm,clip]{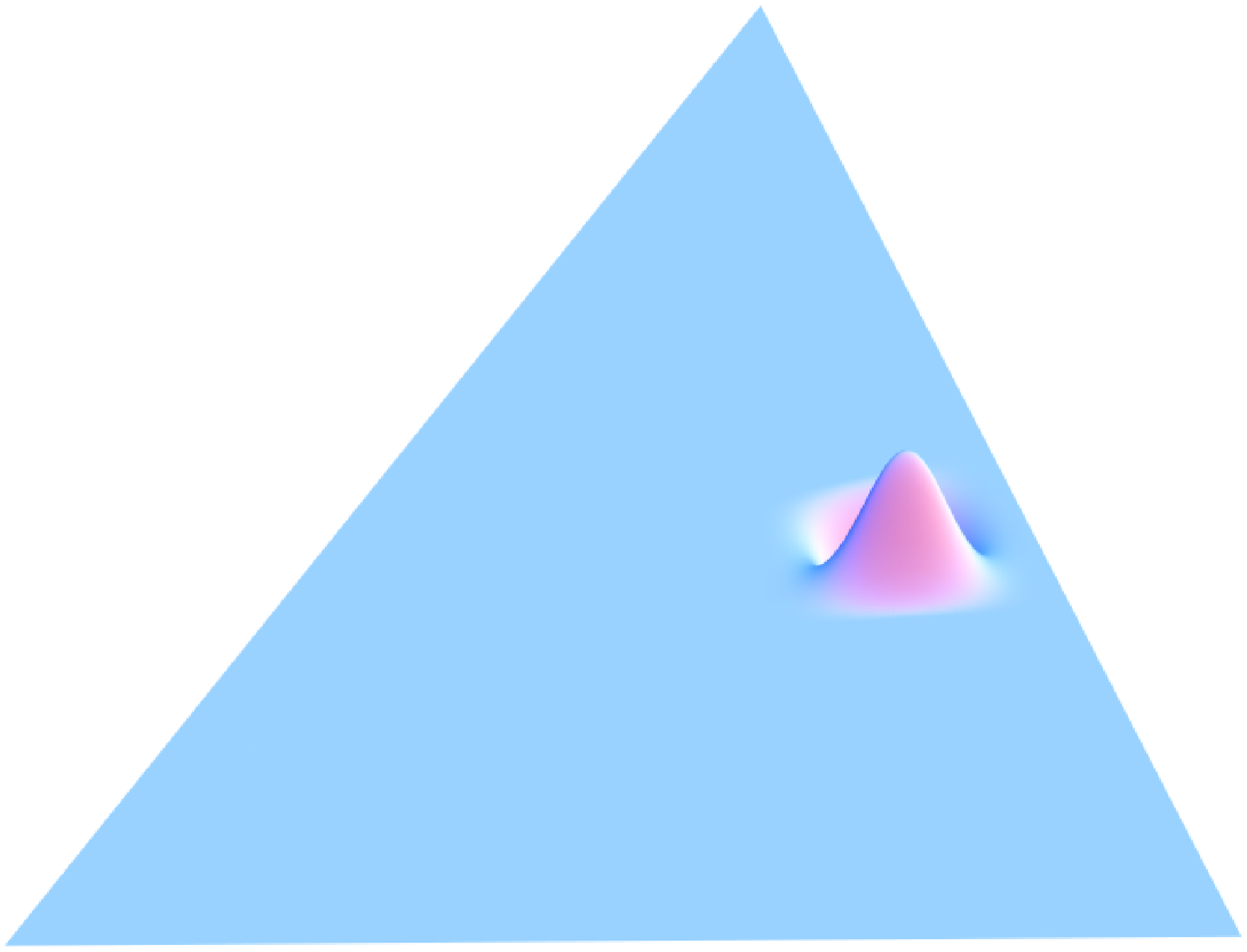} \\
$\phi_{\Gamma^{(1)};0,5}$ & $\phi_{\Gamma^{(1)};1,4}$ & $\phi_{\Gamma^{(1)};1,5}$  \\
\includegraphics[width=4.0cm,clip]{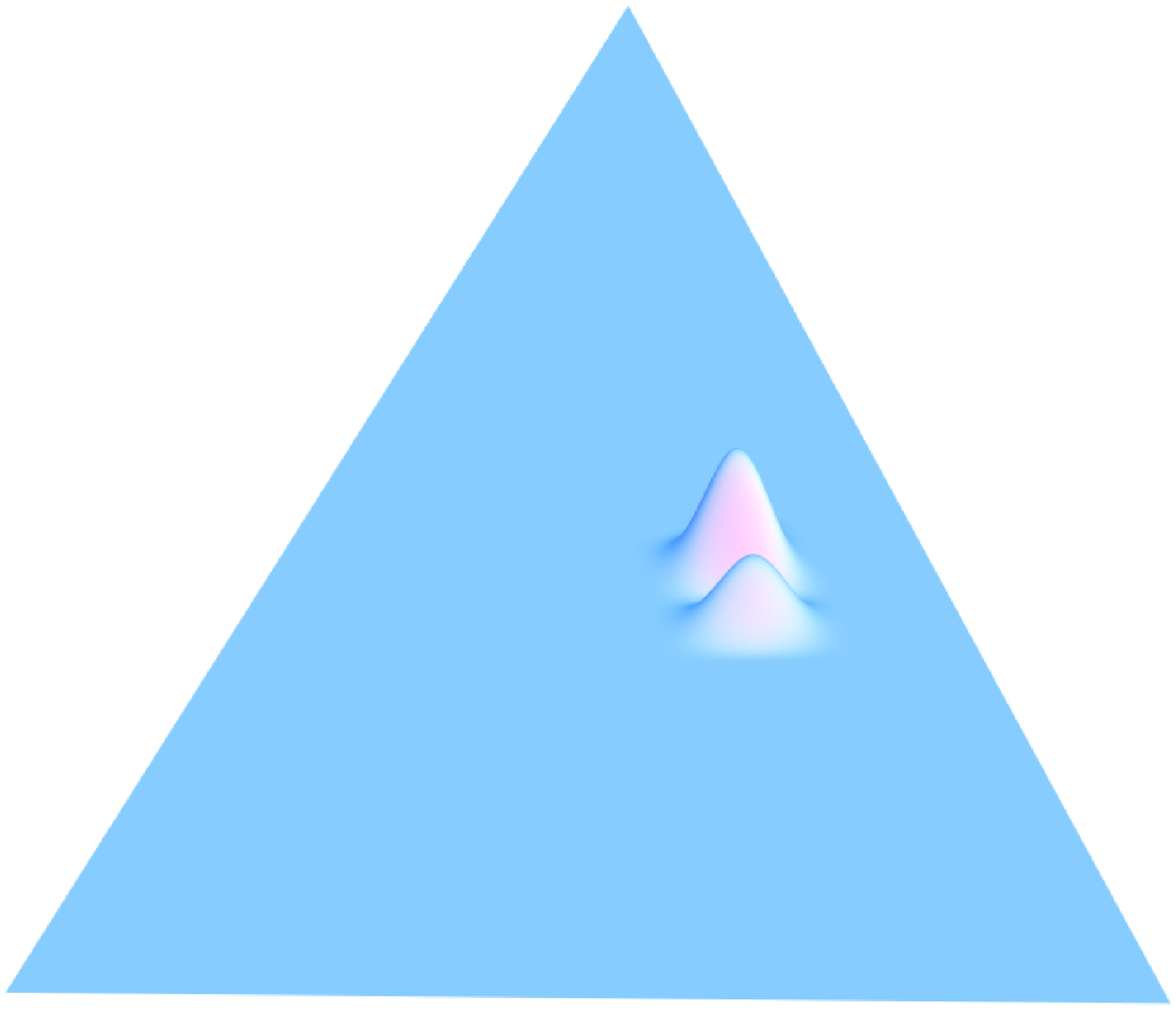} &
\includegraphics[width=4.0cm,clip]{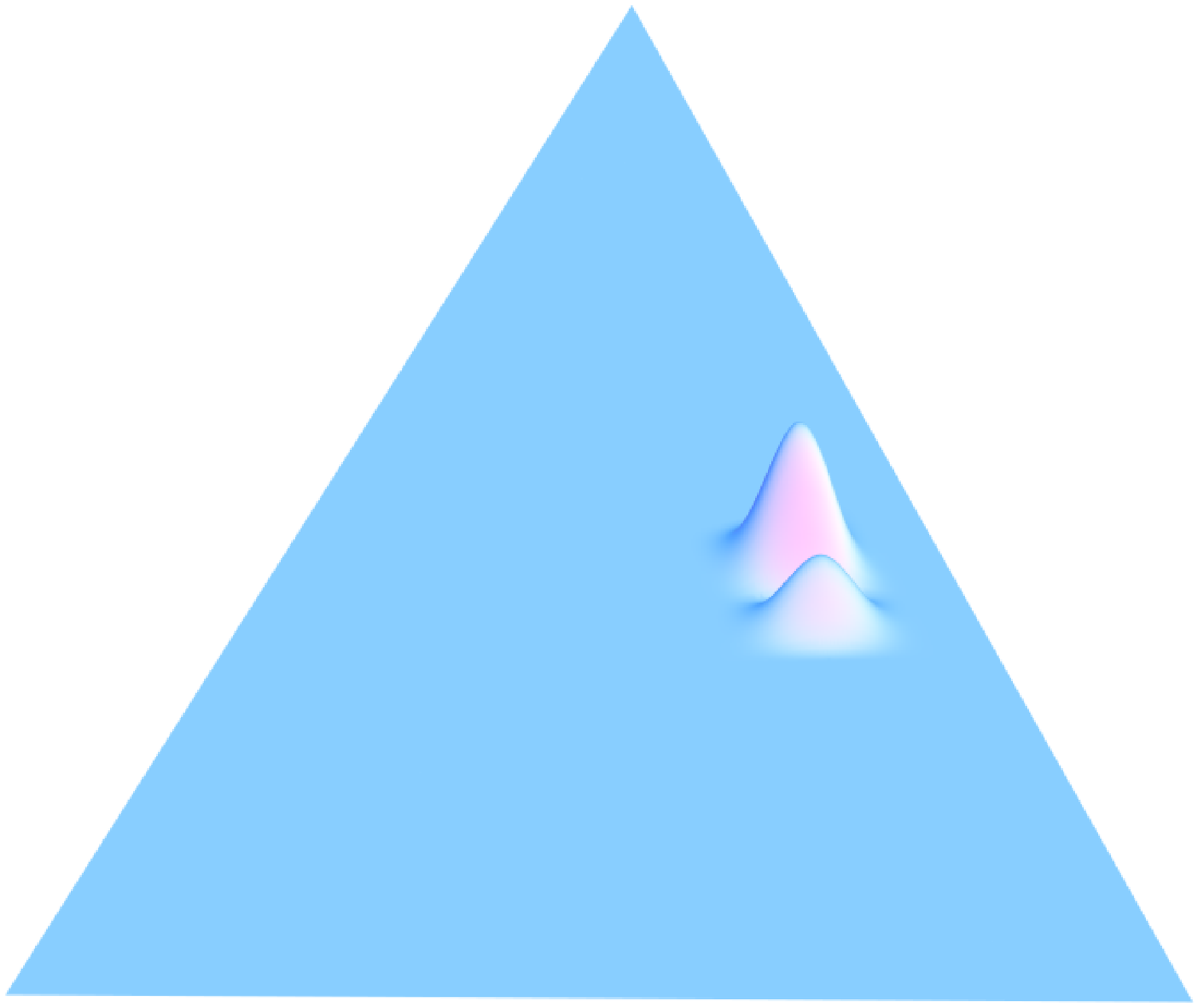} &
\includegraphics[width=4.0cm,clip]{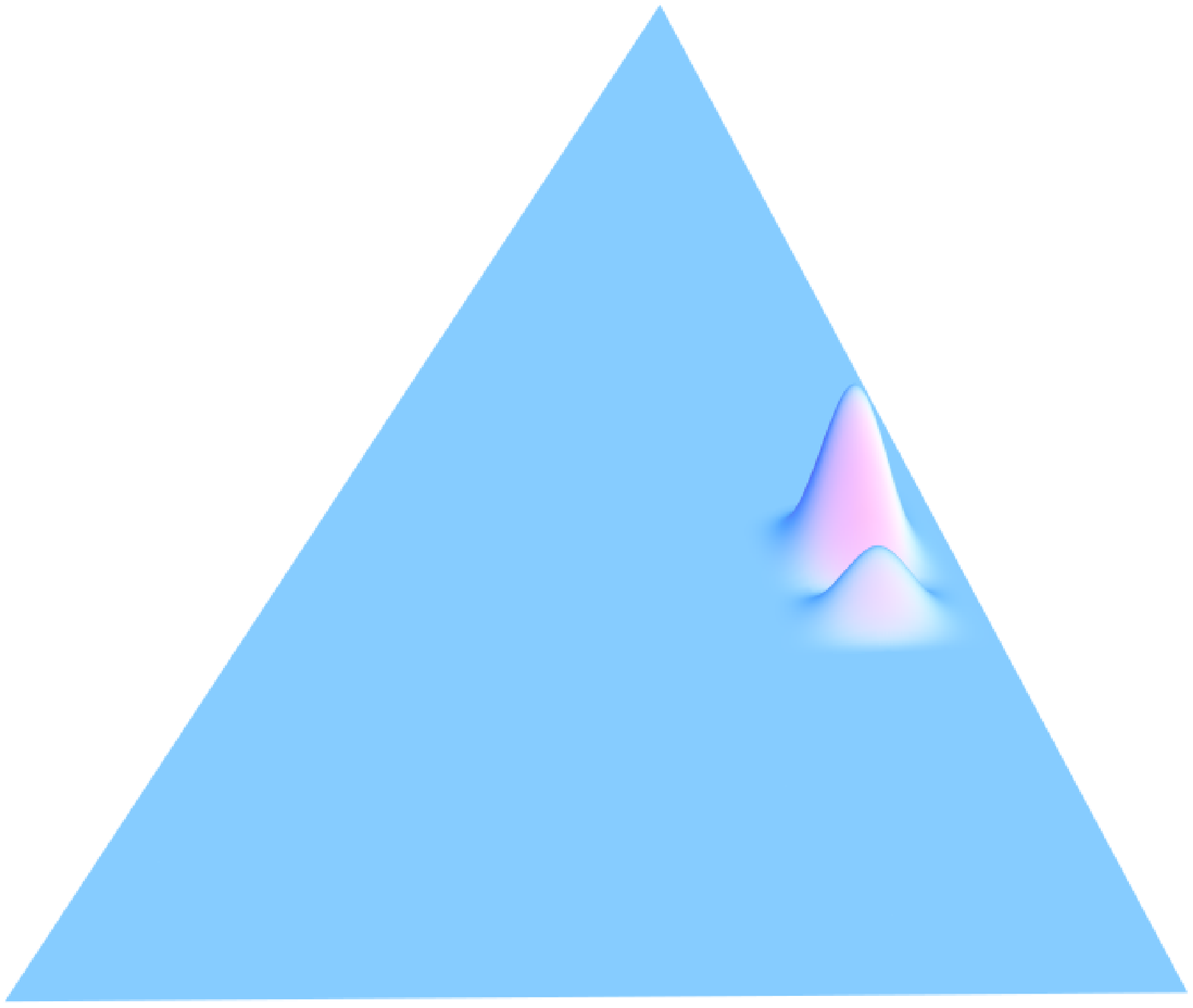} \\
$\phi_{\Gamma^{(1)};2,3}$  & $\phi_{\Gamma^{(1)};2,4}$  & $\phi_{\Gamma^{(1)};2,5}$  
\end{tabular}
\caption{Graphs of the functions of the edge space $\mathcal{W}_{0h;\Gamma^{(1)}}$ for $p=5$, $r=2$ and $h=\frac{1}{6}$ (cf. Example~\ref{ex:functions}).}
\label{fig:ex_functions_edge}
\end{figure}

\begin{figure}[htp]
\centering\footnotesize
\begin{tabular}{cccc}
\includegraphics[width=3.6cm,clip]{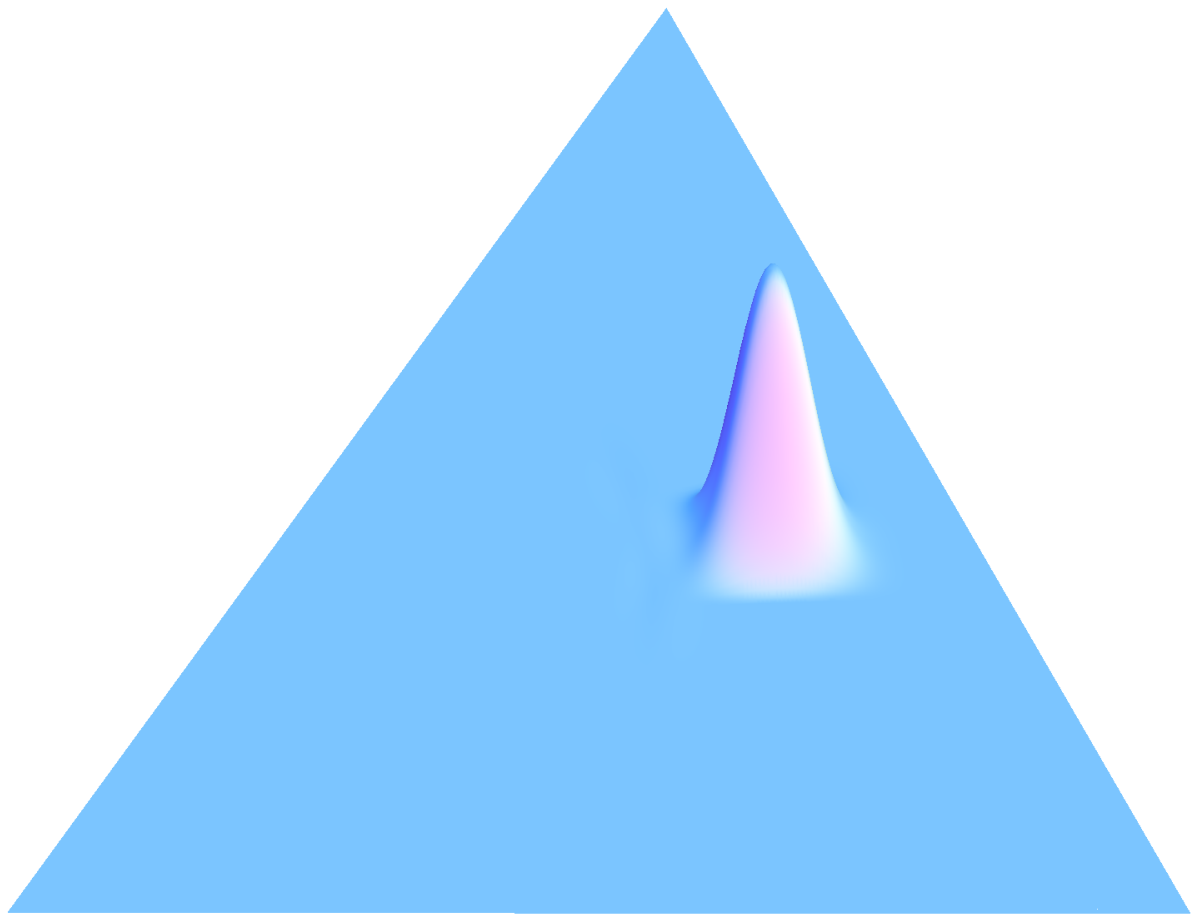} &
\includegraphics[width=3.6cm,clip]{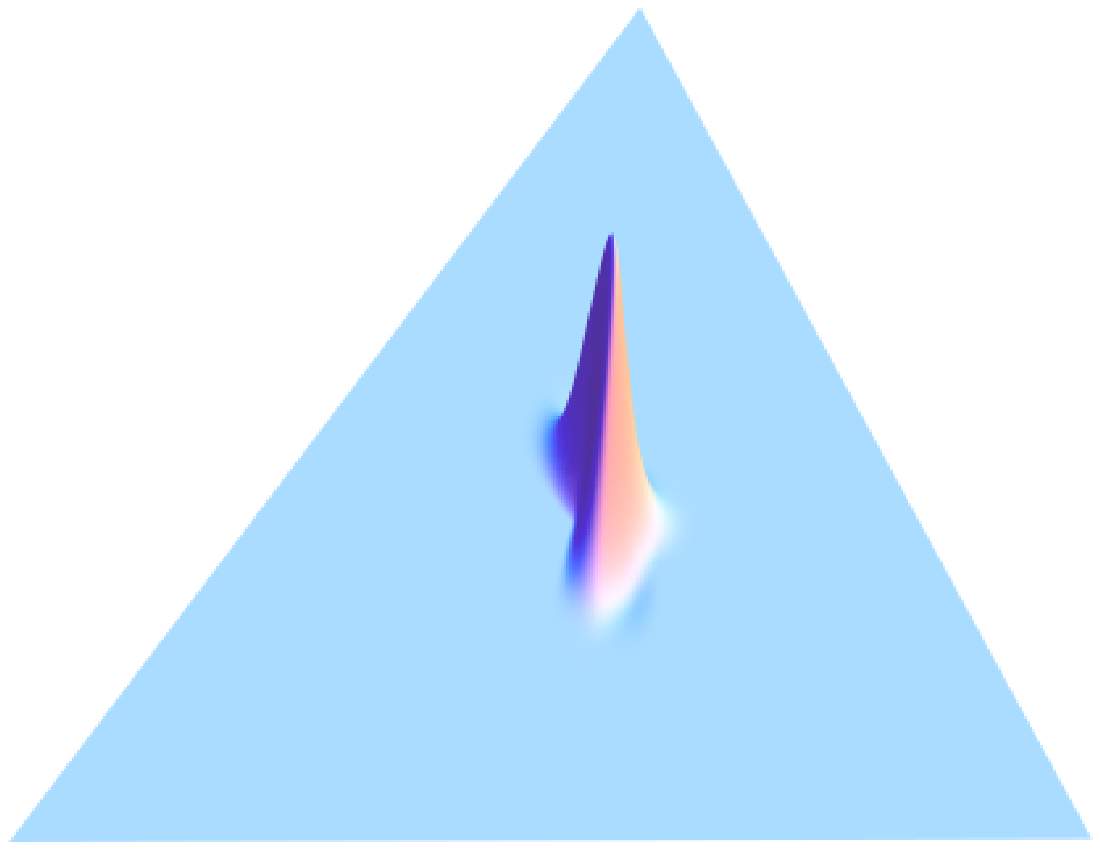} &
\includegraphics[width=3.6cm,clip]{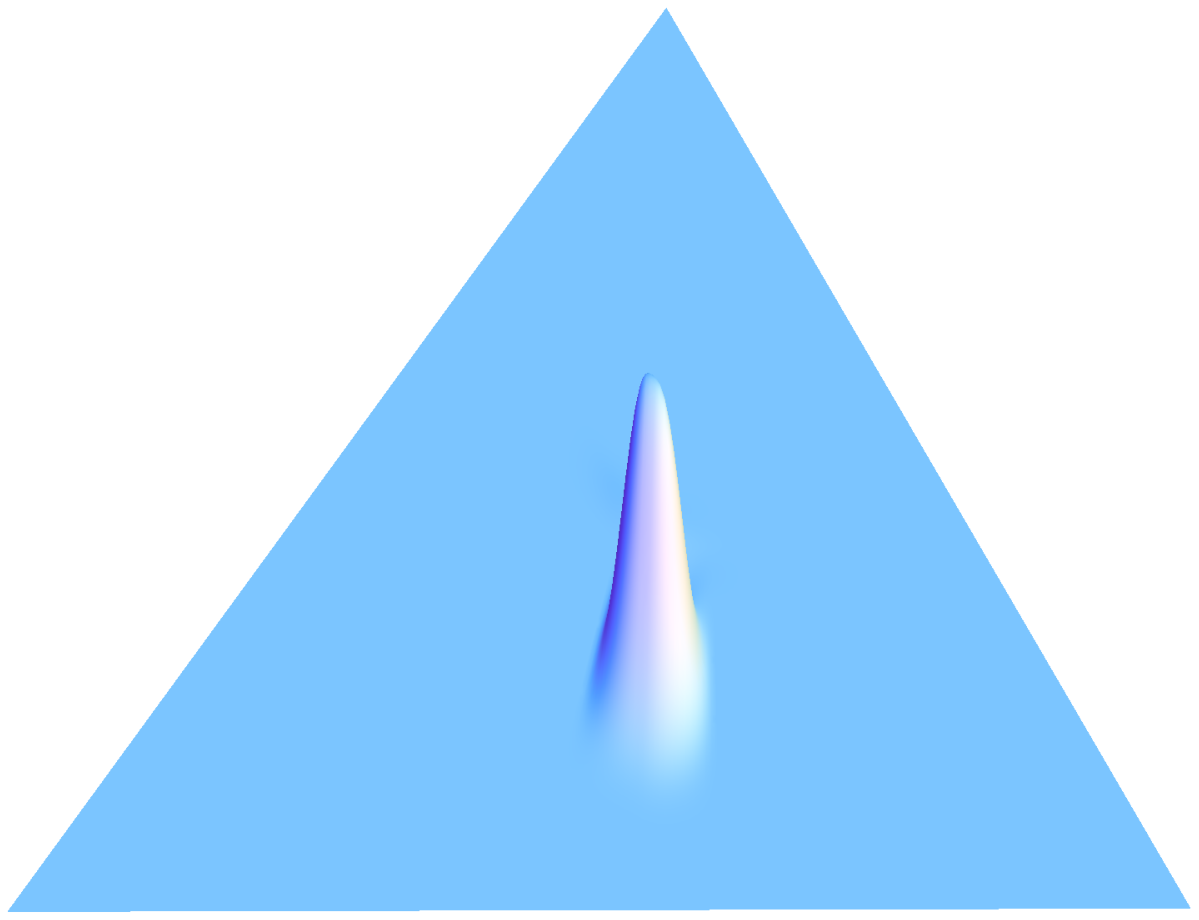} &
\includegraphics[width=3.6cm,clip]{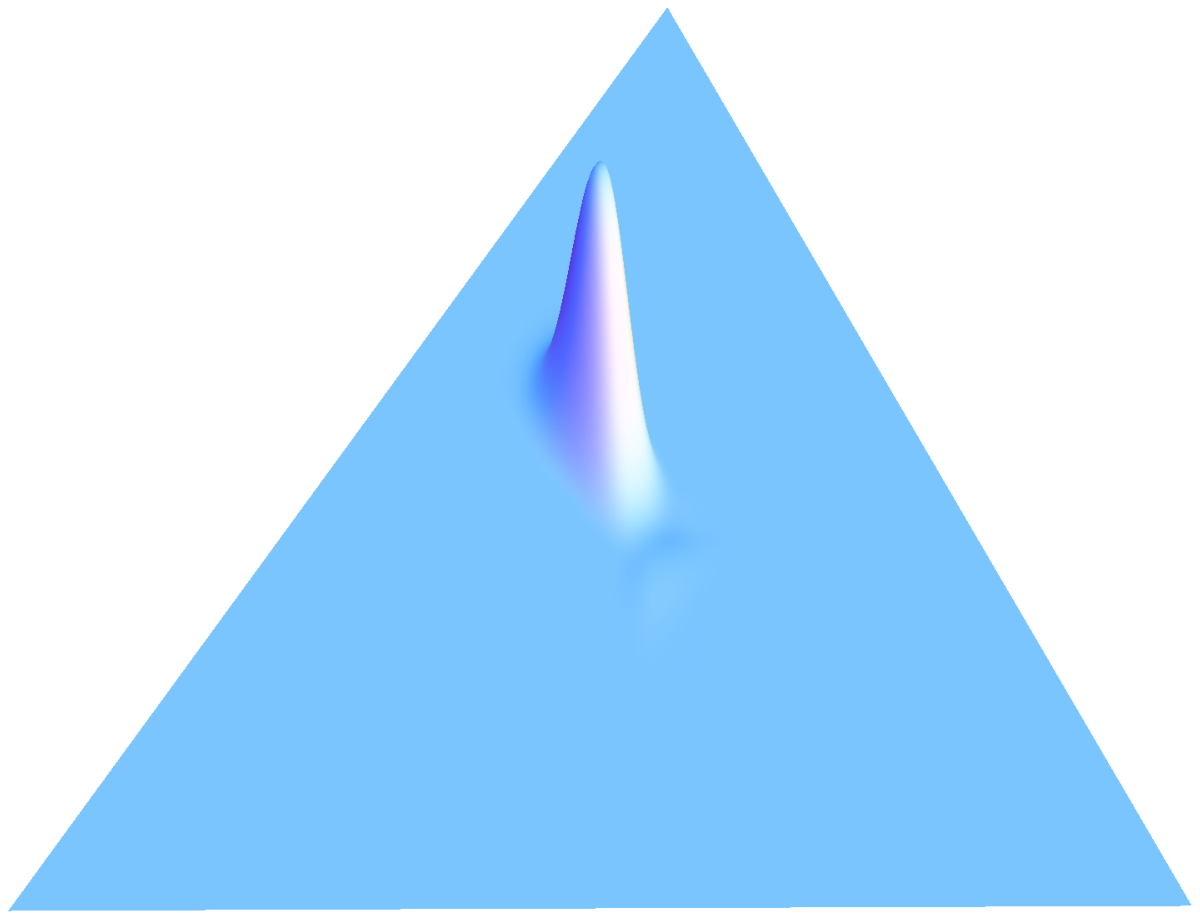} \\
$\phi_{\bfm{v}^{(1)};1}$ & $\phi_{\bfm{v}^{(1)};2}$ & $\phi_{\bfm{v}^{(1)};3}$ & $\phi_{\bfm{v}^{(1)};4}$\\
\includegraphics[width=3.6cm,clip]{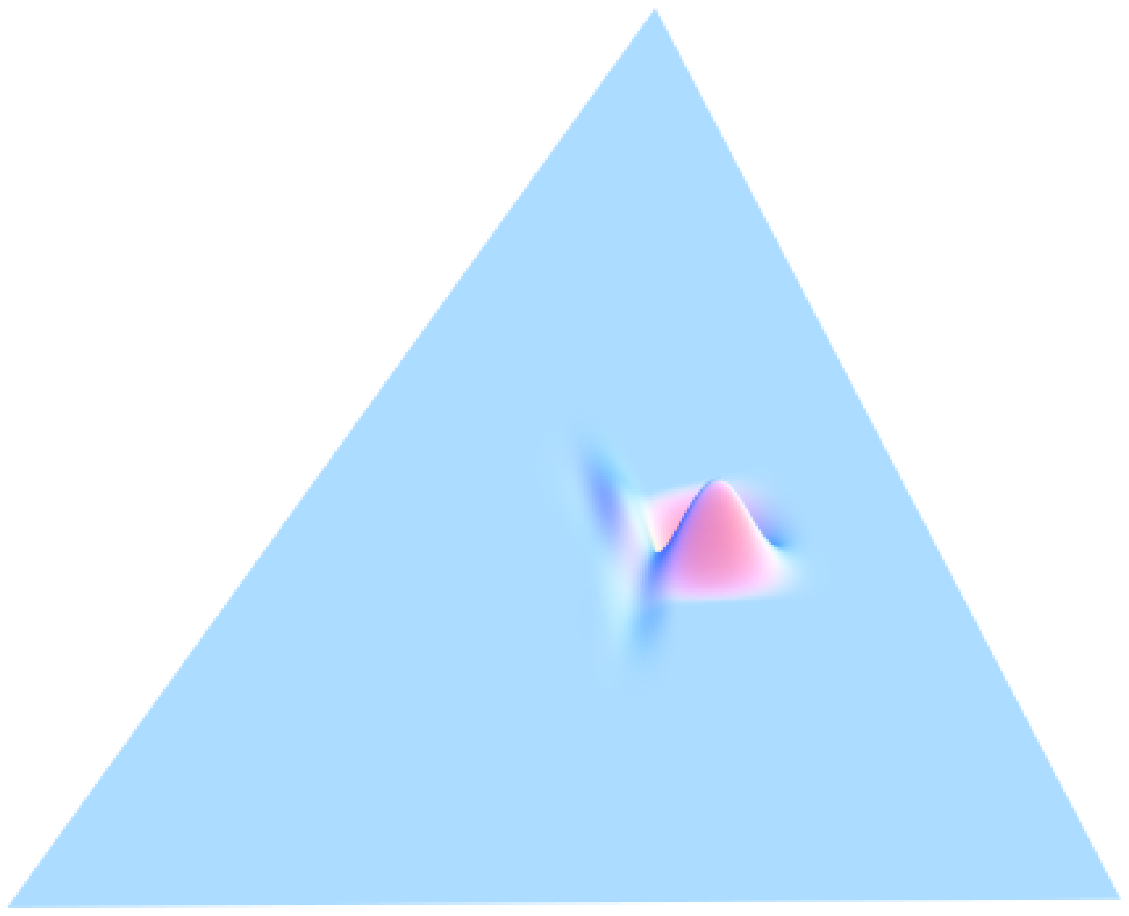}  &
\includegraphics[width=3.6cm,clip]{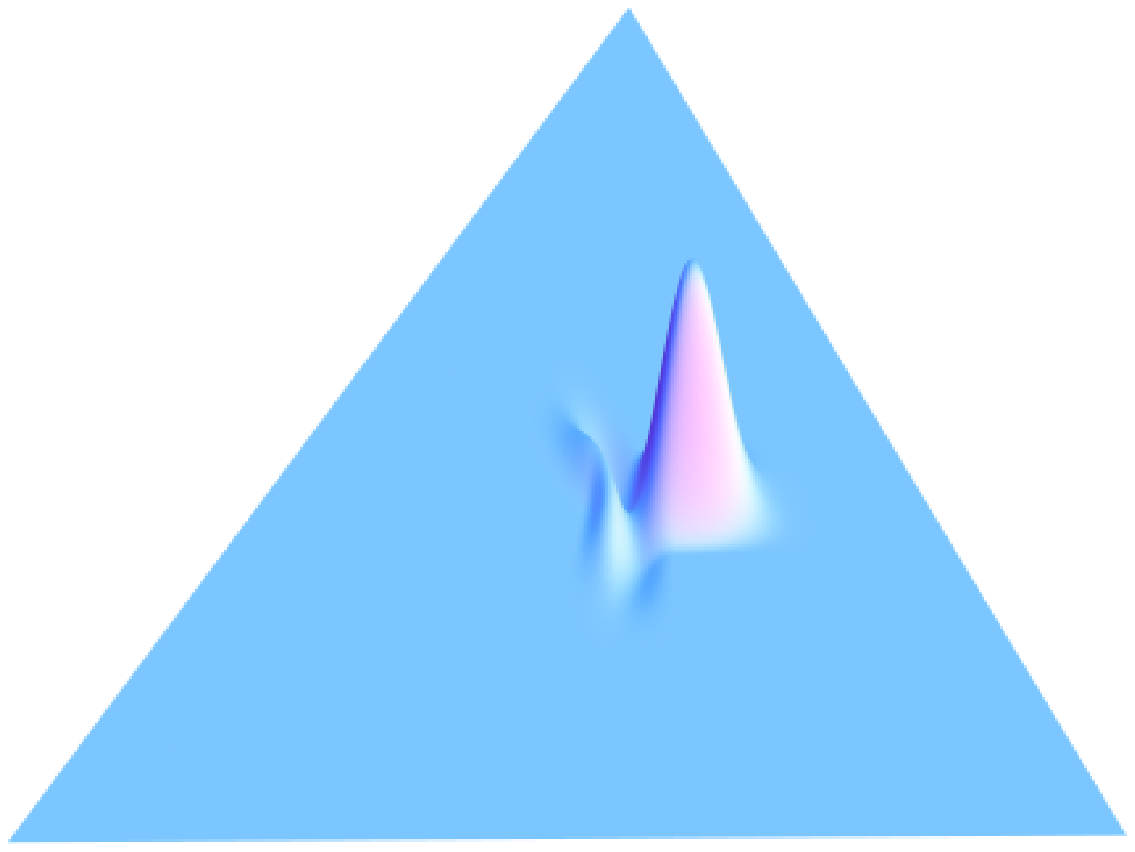} &
\includegraphics[width=3.6cm,clip]{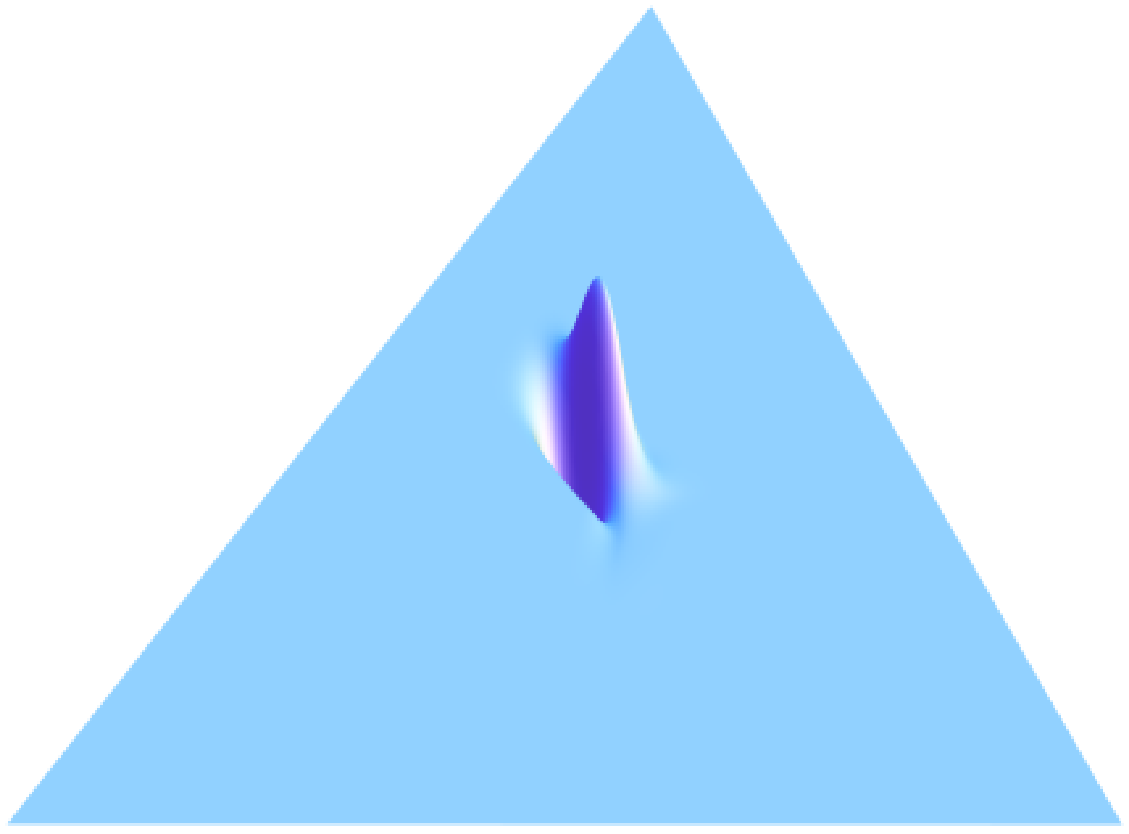} &
\includegraphics[width=3.6cm,clip]{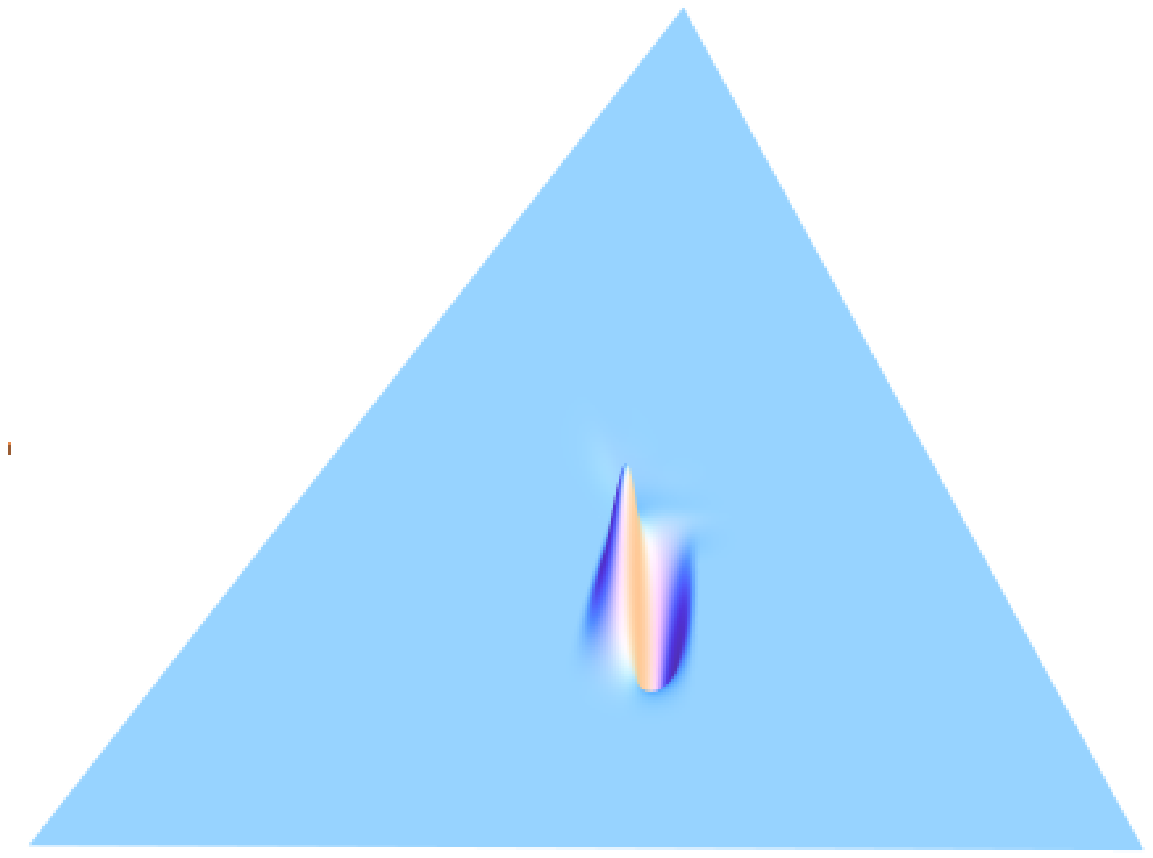} \\
$\phi_{\bfm{v}^{(1)};5}$ & $\phi_{\bfm{v}^{(1)};6}$ & $\phi_{\bfm{v}^{(1)};7}$ & $\phi_{\bfm{v}^{(1)};8}$ \\
\includegraphics[width=3.6cm,clip]{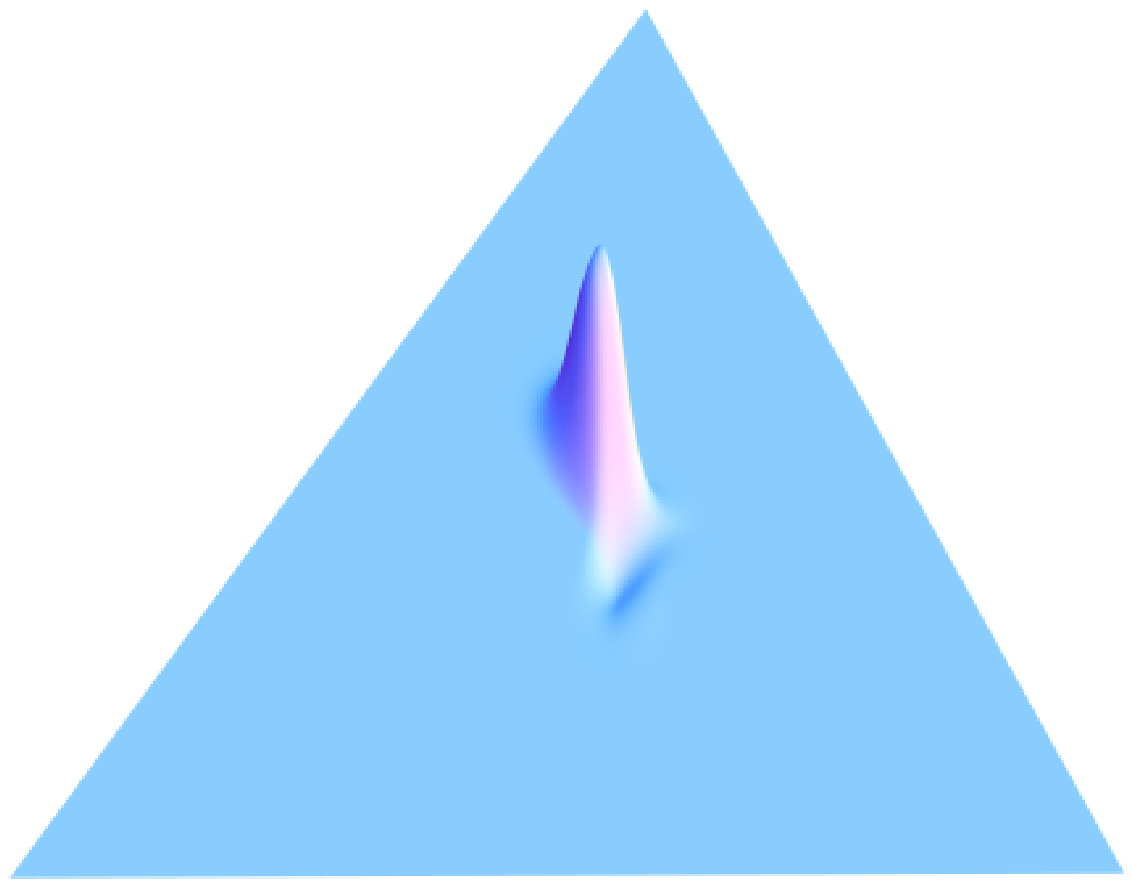}  &
\includegraphics[width=3.6cm,clip]{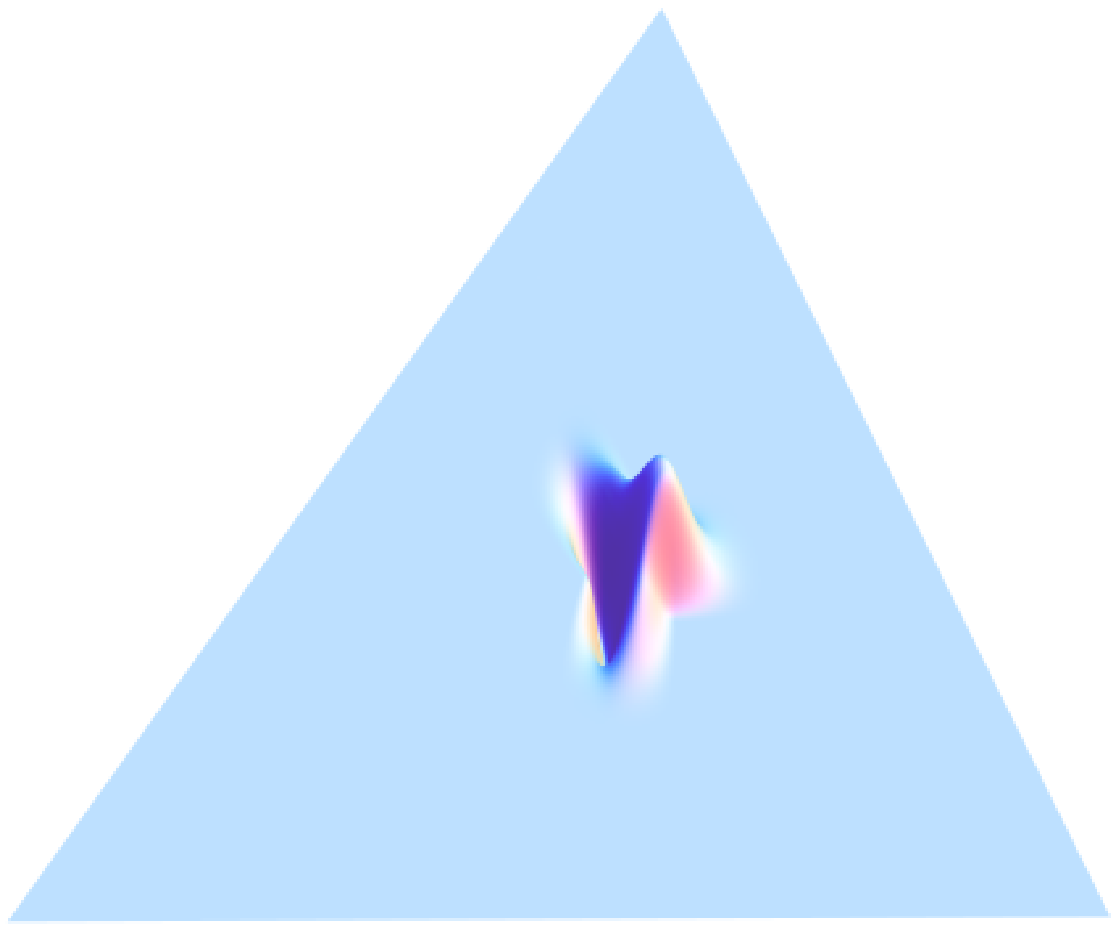} &
\includegraphics[width=3.6cm,clip]{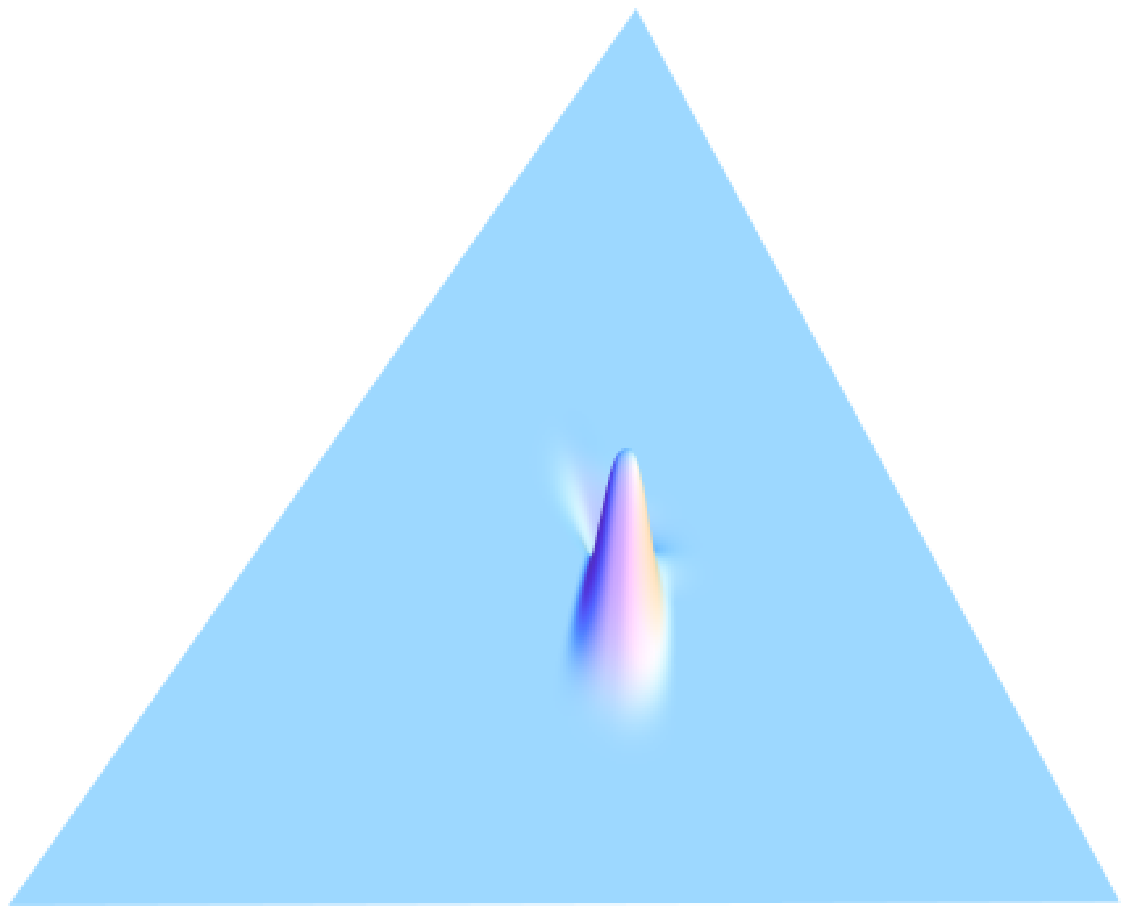} &
\includegraphics[width=3.6cm,clip]{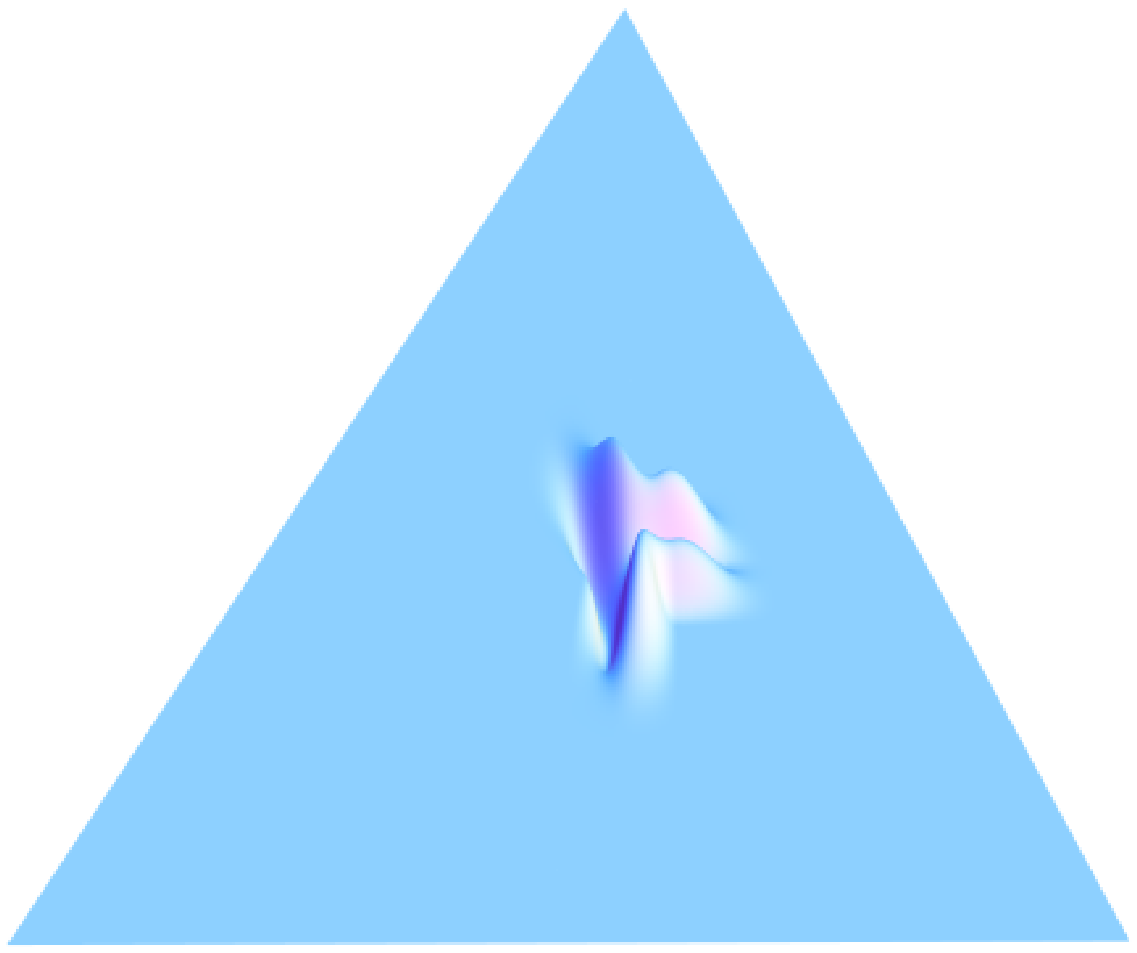} \\
$\phi_{\bfm{v}^{(1)};9}$ & $\phi_{\bfm{v}^{(1)};10}$ & $\phi_{\bfm{v}^{(1)};11}$ & $\phi_{\bfm{v}^{(1)};12}$ \\
\includegraphics[width=3.6cm,clip]{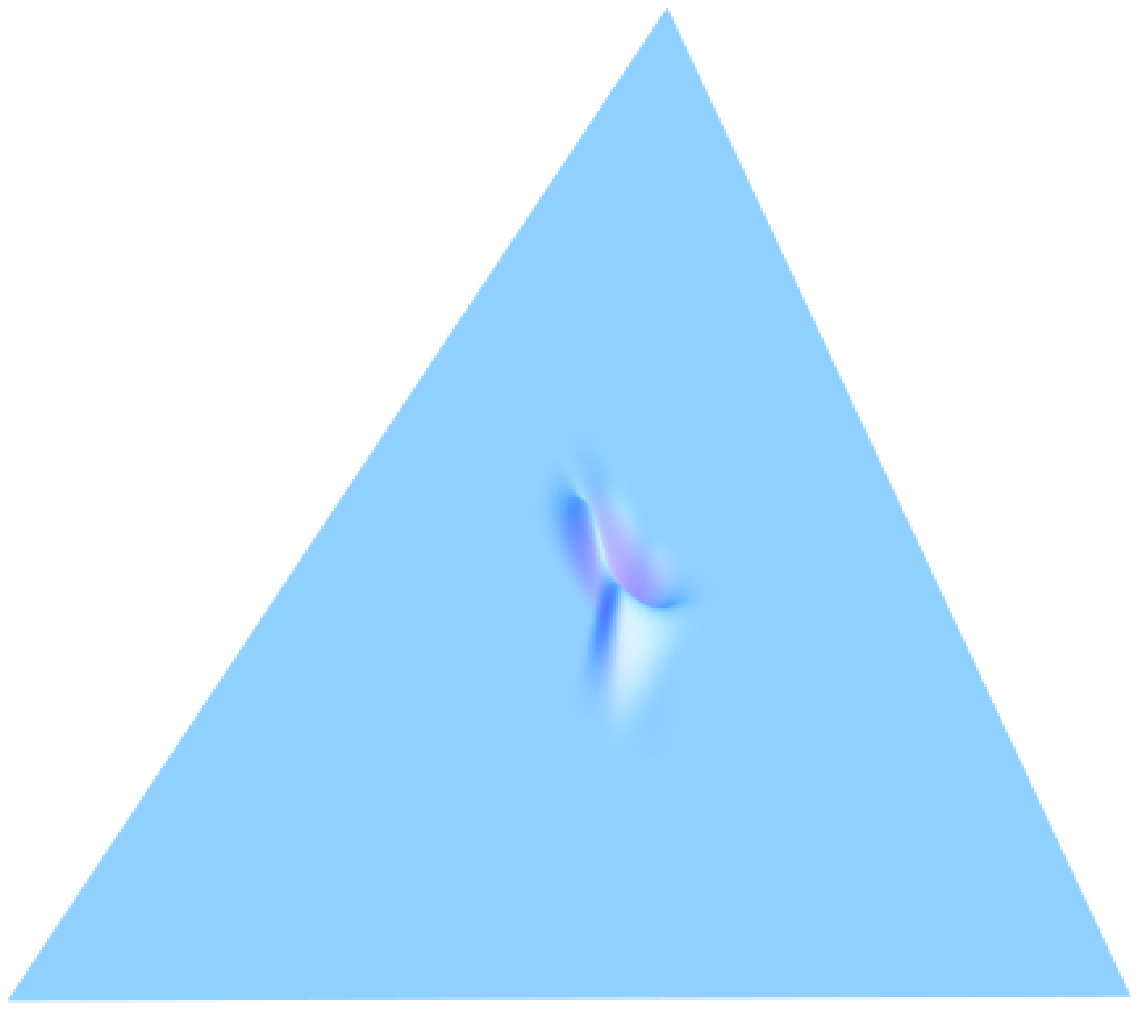}  &
\includegraphics[width=3.6cm,clip]{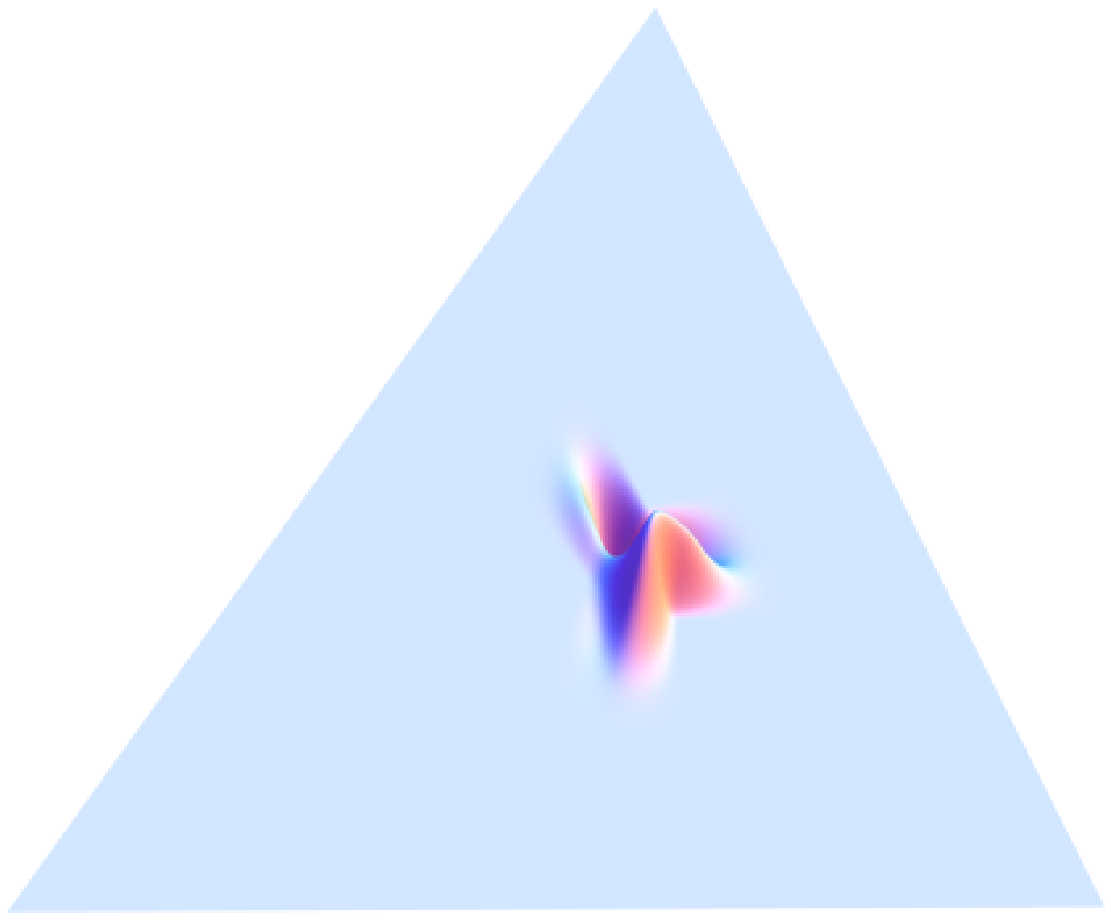} &
\includegraphics[width=3.6cm,clip]{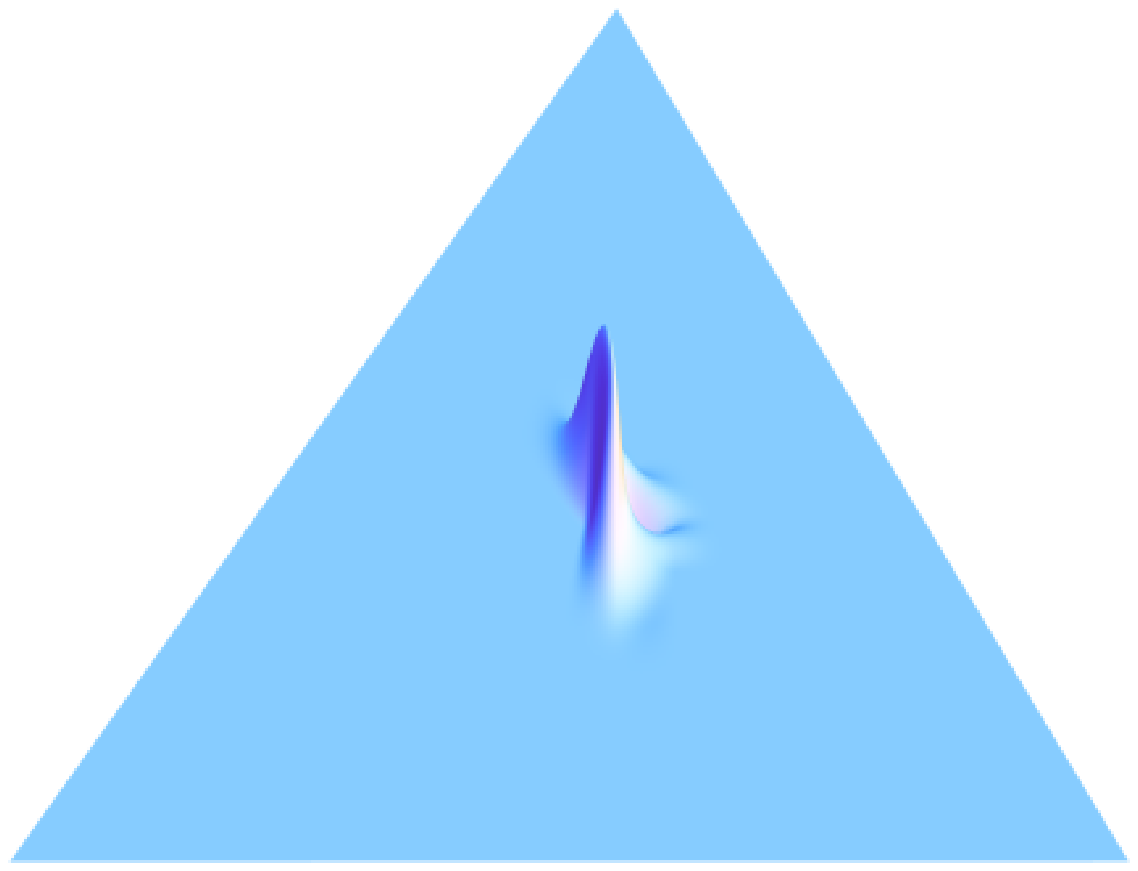} &
\includegraphics[width=3.6cm,clip]{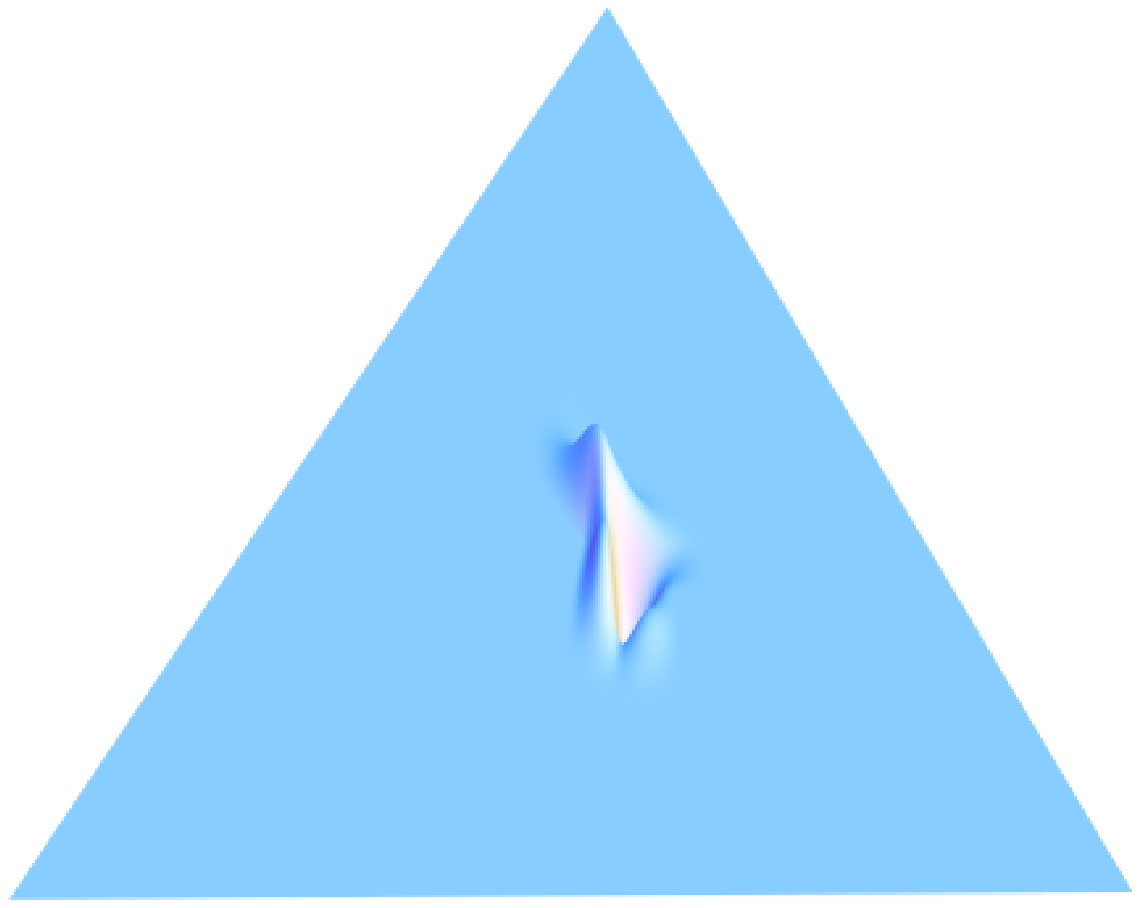} \\
$\phi_{\bfm{v}^{(1)};13}$ & $\phi_{\bfm{v}^{(1)};14}$ & $\phi_{\bfm{v}^{(1)};15}$ & $\phi_{\bfm{v}^{(1)};16}$ 
\end{tabular}
\begin{tabular}{ccc}
\includegraphics[width=3.6cm,clip]{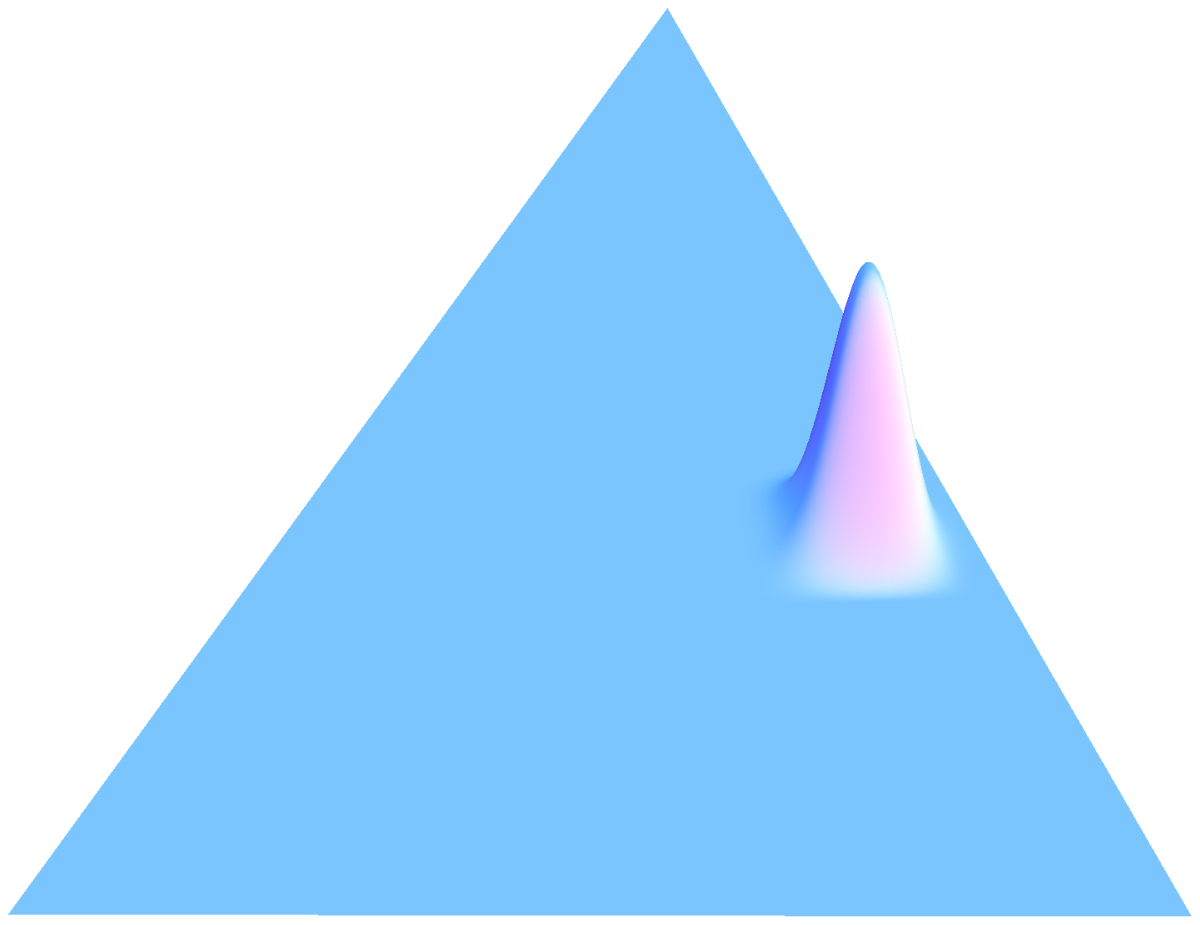} &
\includegraphics[width=3.6cm,clip]{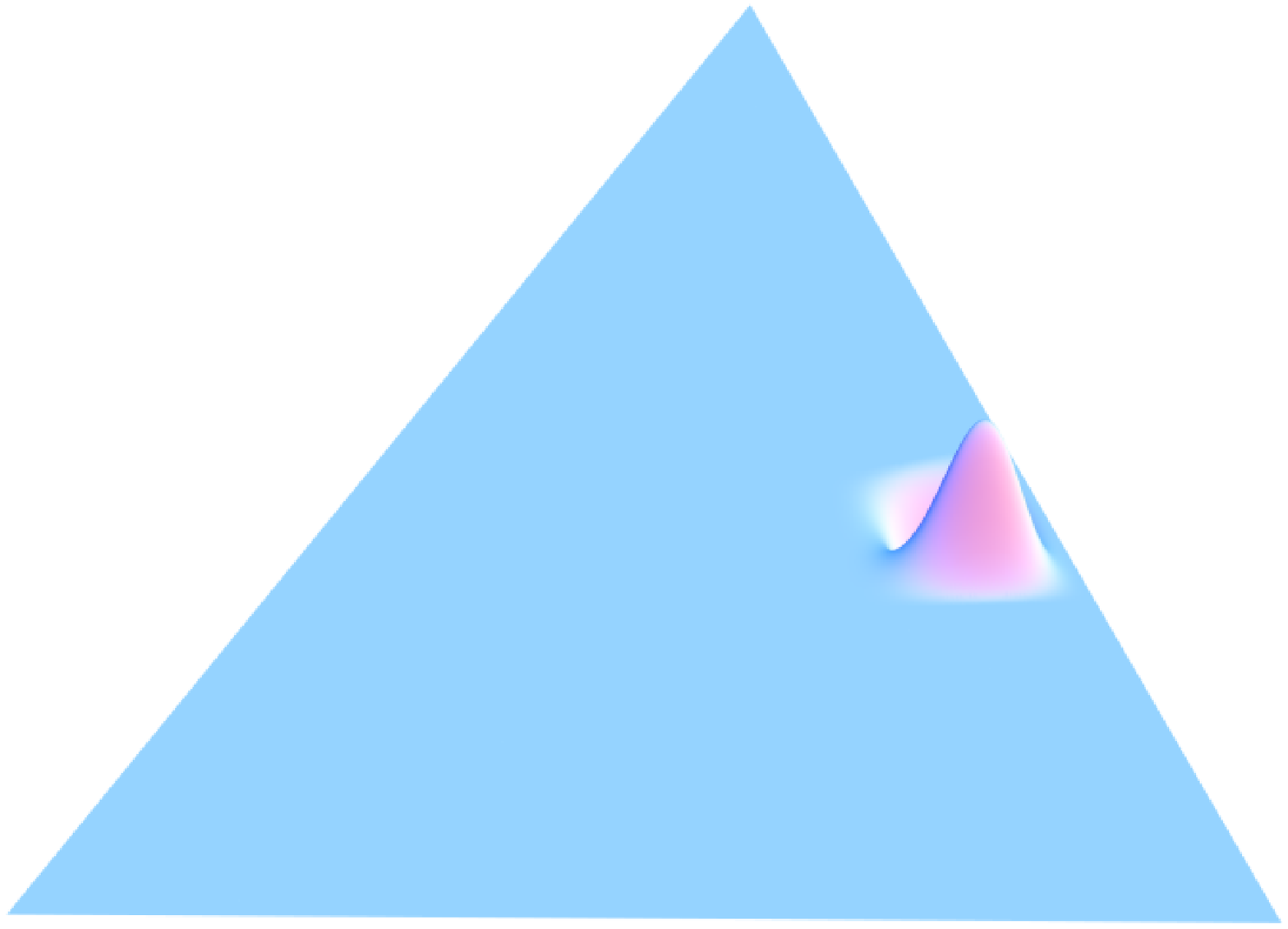} &
\includegraphics[width=3.6cm,clip]{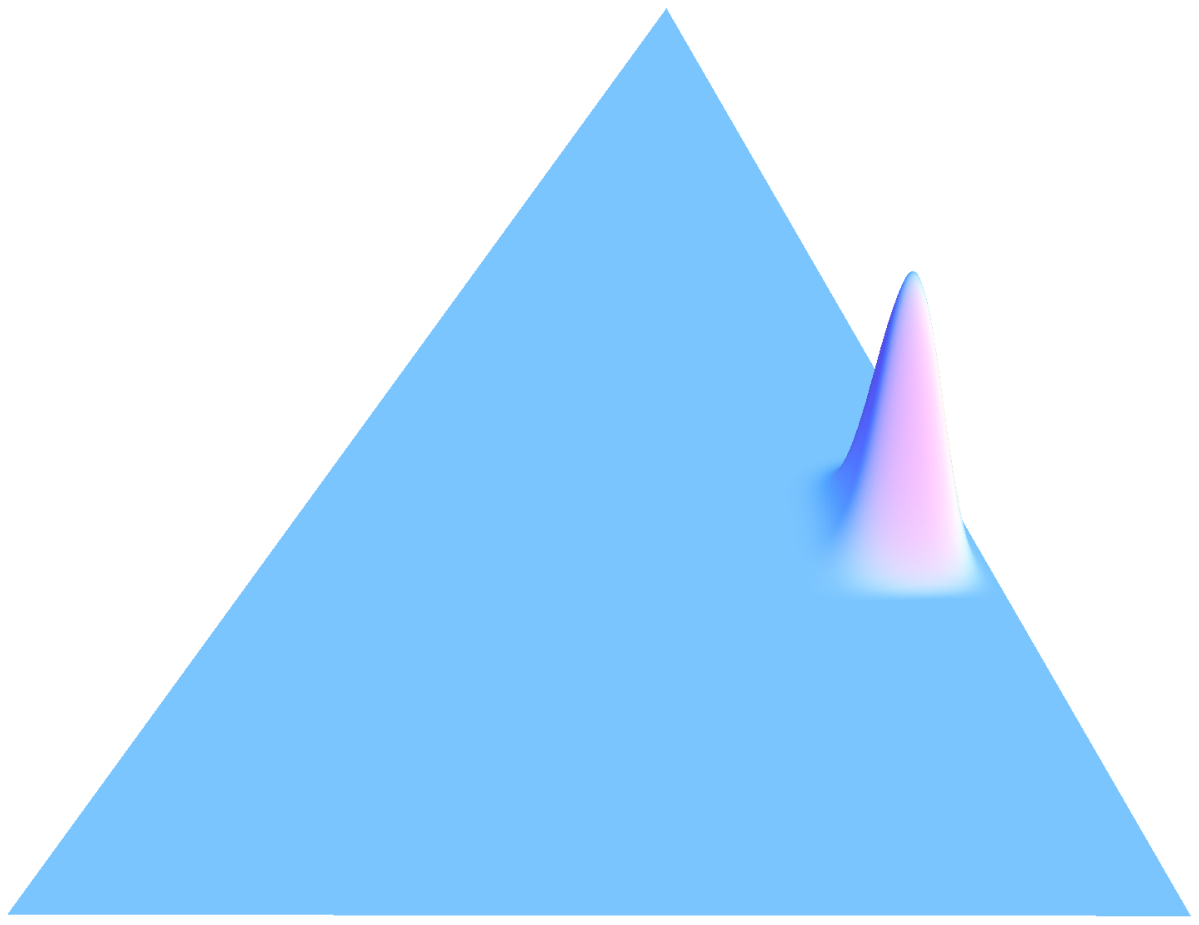} \\
$\phi_{\bfm{v}^{(2)};1}$ & $\phi_{\bfm{v}^{(2)};2}$ & $\phi_{\bfm{v}^{(2)};3}$ 
\end{tabular}
\caption{Graphs of the functions of the vertex spaces $\mathcal{W}_{0h;\bfm{v}^{(1)}}$ and $\mathcal{W}_{0h;\bfm{v}^{(2)}}$ for $p=5$, $r=2$ and 
$h=\frac{1}{6}$ (cf. Example~\ref{ex:functions}).}
\label{fig:ex_functions_vertex}
\end{figure}
\end{ex}

\begin{rem}
For the sake of simplicity we restricted ourselves to the case of bilinearly parameterized multi-patch domains. The construction of the space~$\W$ and of its basis 
should be extendable in a straightforward way to the class of bilinear-like geometries~\cite{KaVi17c}. However, the construction and the study of bilinear-like geometries themselves 
are limited to the case of two-patch domains~\cite{KaVi17c} so far. But an extension to the case of multi-patch domains is of vital interest for the future research.
\end{rem}

\section{Solving the triharmonic equation -- Examples} \label{sec:triharmonic_examples}

We present several examples to demonstrate the potential of our approach for solving the triharmonic equation over bilinear multi-patch domains. 

\begin{ex} \label{ex:example1}
We consider the three bilinearly parameterized multi-patch domains given in Figure~\ref{fig:example1}~(first row), which possess extraordinary vertices of valency~$3$, $5$ or $6$ {\MK and 
describe a triangular, pentagonal and hexagonal domain, respectively.} 
For all three domains (a)-(c), we construct nested isogeometric spline spaces 
$\W$ of degree~$p=5$ and regularity~$r=2$ for the mesh-sizes $h=\frac{1}{k+1}$, $k \in \{3,7,15,31 \}$. Note that for the case of $h=\frac{1}{4}$, the construction of the 
space $\W$ has to be slightly modified. More precisely, the vertex subspace $\mathcal{W}_{0h;\bfm{v}^{(\rho)}}$ is constructed without the use of the 
functions~$\phi_{\Gamma^{(\ell)};0,4}$. Instead, these functions are added to the corresponding edge subspaces~$\mathcal{W}_{0h;\Gamma^{(\ell)}}$ after subtracting suitable linear 
combinations of 
functions $\phi_{\Gamma^{(\ell)};i,j}$, $0 \leq i \leq 2$, $0 \leq j \leq \min(4-i,3)$ to obtain functions~$\widehat{\phi}_{\Gamma^{(\ell)};0,4}$ which have 
vanishing values, gradients and Hessians on $\partial (\Omega^{(\ell-1)} \cup \Omega^{(\ell)} )$. 

We solve the triharmonic equation~\eqref{eq:triharmonic_problem} with the homogeneous boundary conditions~\eqref{eq:triharmonic_problem_boundary} over 
the domains~(a)-(c) for right side functions~$f$ obtained by the exact solutions
\[ \footnotesize
 u_{a}(\ab{x}) = \left(\frac{1}{20} x_2 (\frac{12 x_1}{13}-x_2)(\frac{120 - 12 x_1}{7} -x_2) \right)^3,
\]
\[ \footnotesize 
 u_{b}(\ab{x}) = \left(\frac{1}{20000} (\frac{121 +8 x_1}{15} - x_2)( \frac{7 x_1}{2} + x_2)x_2(\frac{52}{3} - \frac{13 x_1}{6} + 
    x_2)( \frac{31}{2} - \frac{9 x_1}{11} - x_2) \right)^3
\]
and
\[ \footnotesize
 u_{c}(\ab{x}) = \left( \frac{1}{200000}(\frac{55- 5 x_1}{2} - x_2)(\frac{1799}{160} - \frac{7 x_1}{64} - x_2)(\frac{652 +78 x_1}{61} - x_2)
 (-\frac{18 x_1}{11} - x_2)x_2(\frac{5 x_1}{3} -10 - x_2)\right)^3,
 \]
see Fig.~\ref{fig:example1}~(second row). The resulting relative $H^{i}$-errors, $i=0,1,2,3$, are visualized in Fig.~\ref{fig:example1}~(third row) and indicate convergence rates 
of order $\mathcal{O}(h^{6-i})$ in the corresponding norms.~\footnote{Note that for the spaces~$\W$ the norms $|| \cdot ||_{H^3(\Omega)}$ and 
$||\nabla \triangle (\cdot) ||_{L^2(\Omega)}$ are equivalent.}{} {\MK Furthermore, Fig.~\ref{fig:example1}~(fourth row) shows the resulting condition numbers $\kappa$ of the 
stiffness matrices~$S$ by using diagonally scaling (cf. \cite{Br95}) and by employing no preconditioner. In case of the non-preconditioned stiffness matrices, 
the errors are slightly higher, but for both cases} 
the estimated growth rates are of order $\mathcal{O}(h^{-6})$, which demonstrate that the constructed basis functions are well-conditioned.

\begin{figure}[htp]
\centering \footnotesize
\begin{tabular}{ccc}
(a) & (b) & (c)\\
\includegraphics[width=4.8cm,clip]{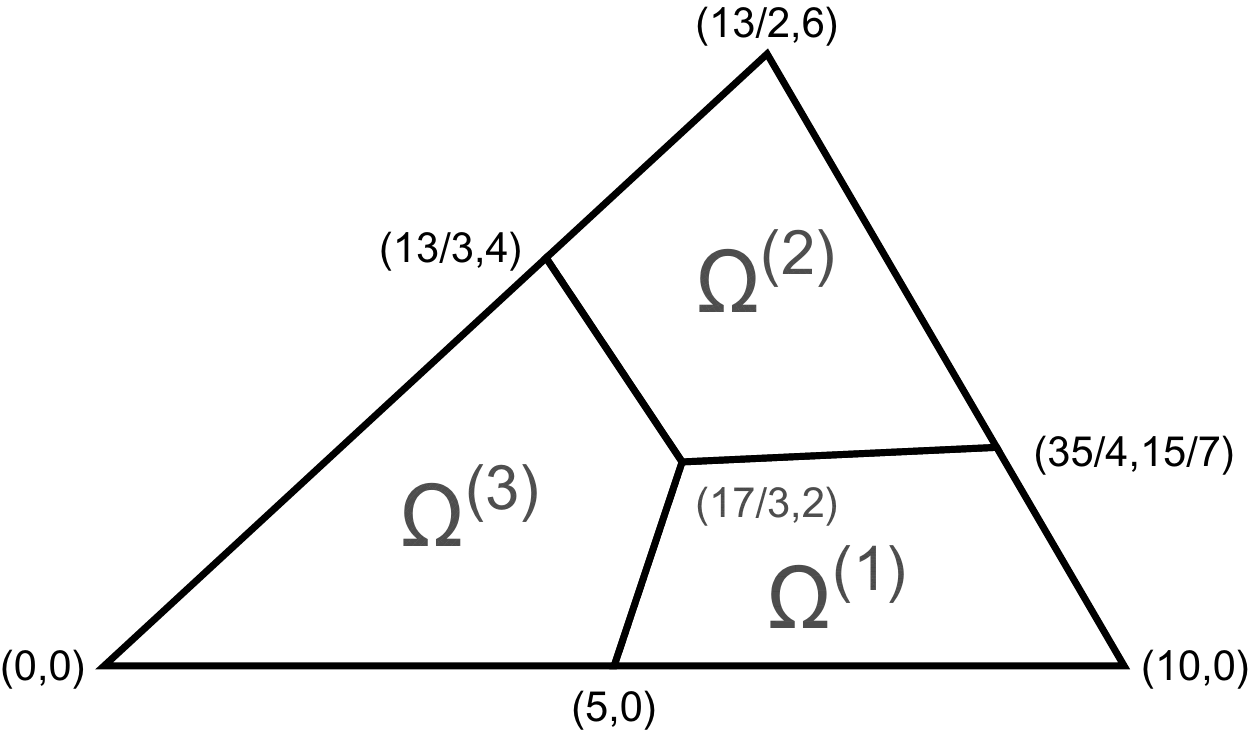} &  
\includegraphics[width=4.8cm,clip]{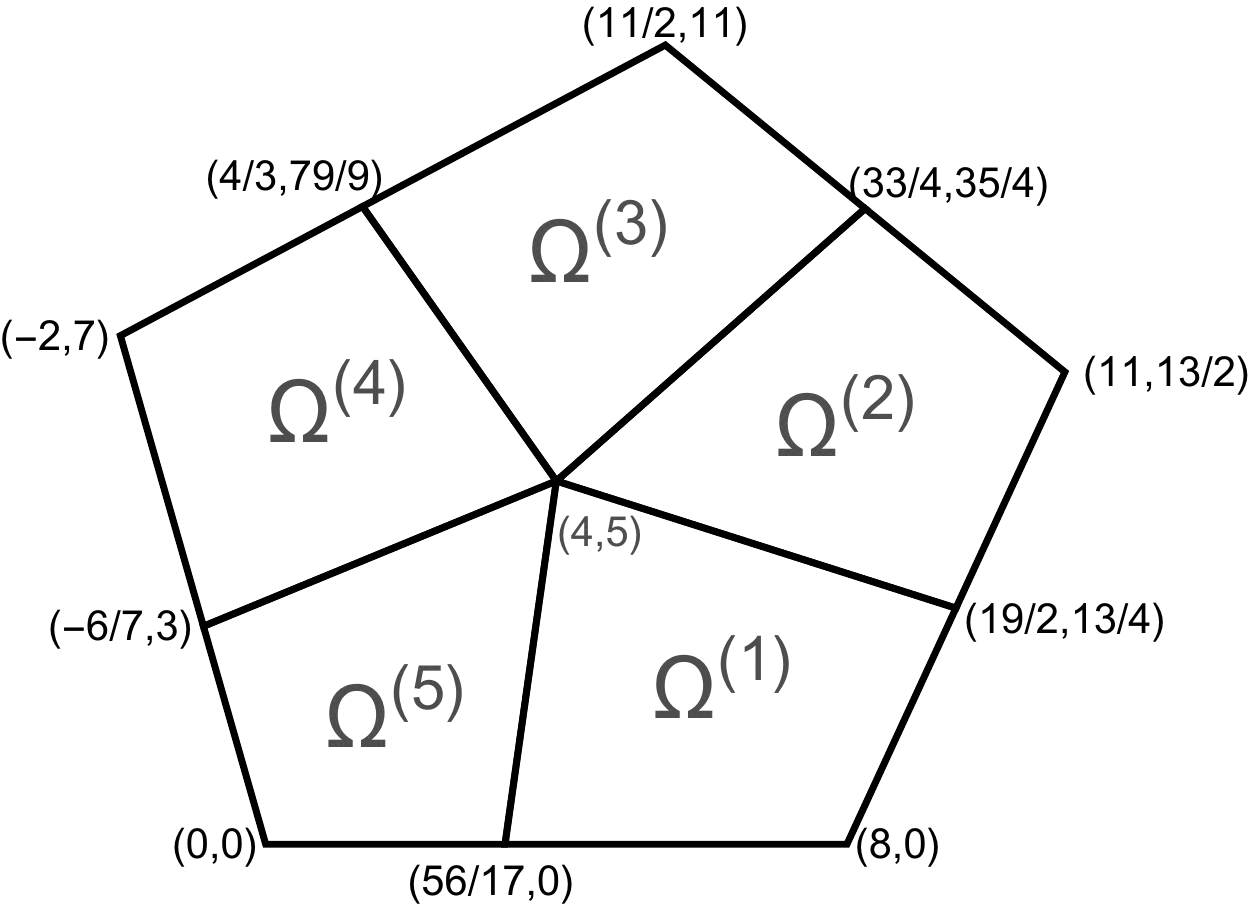} & 
\includegraphics[width=4.8cm,clip]{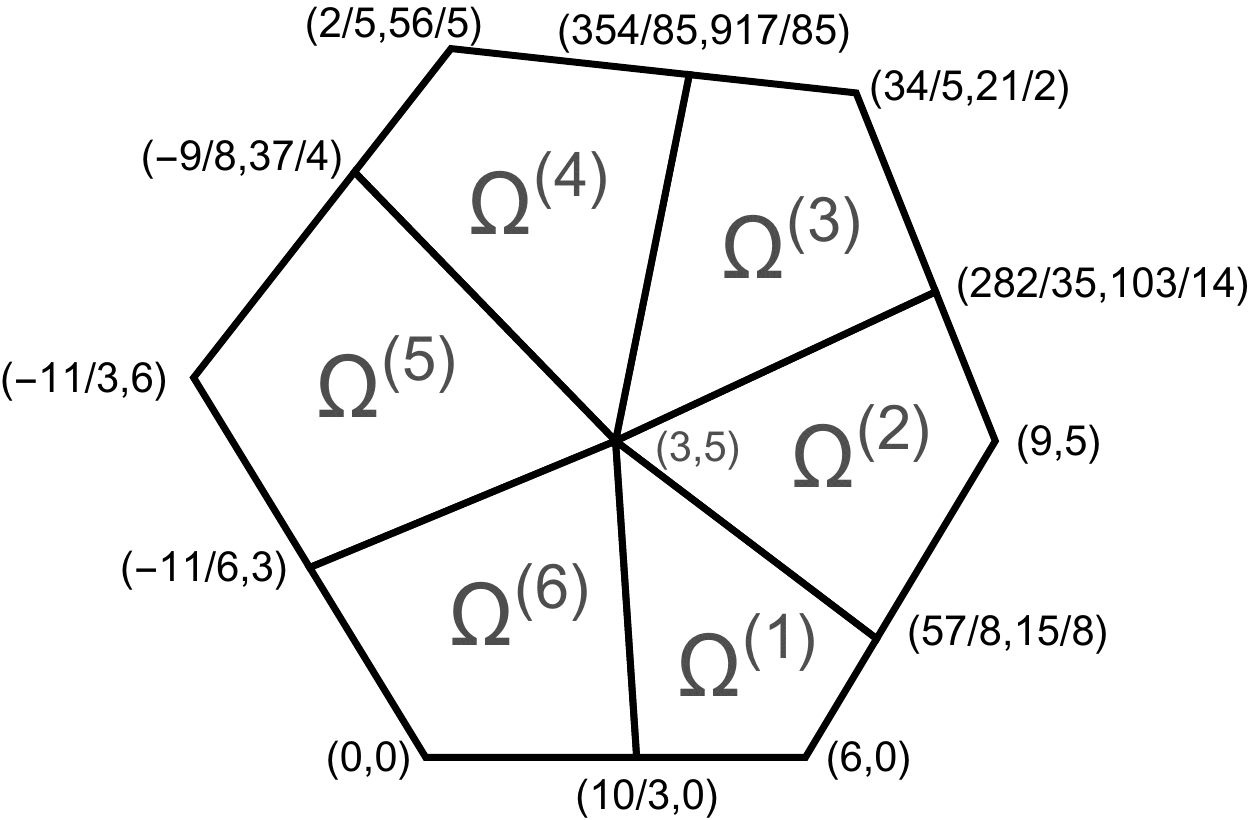} \\ \\
\multicolumn{3}{c}{Computational domains~$\Omega$} \\[0.3cm]
\includegraphics[width=4.8cm,clip]{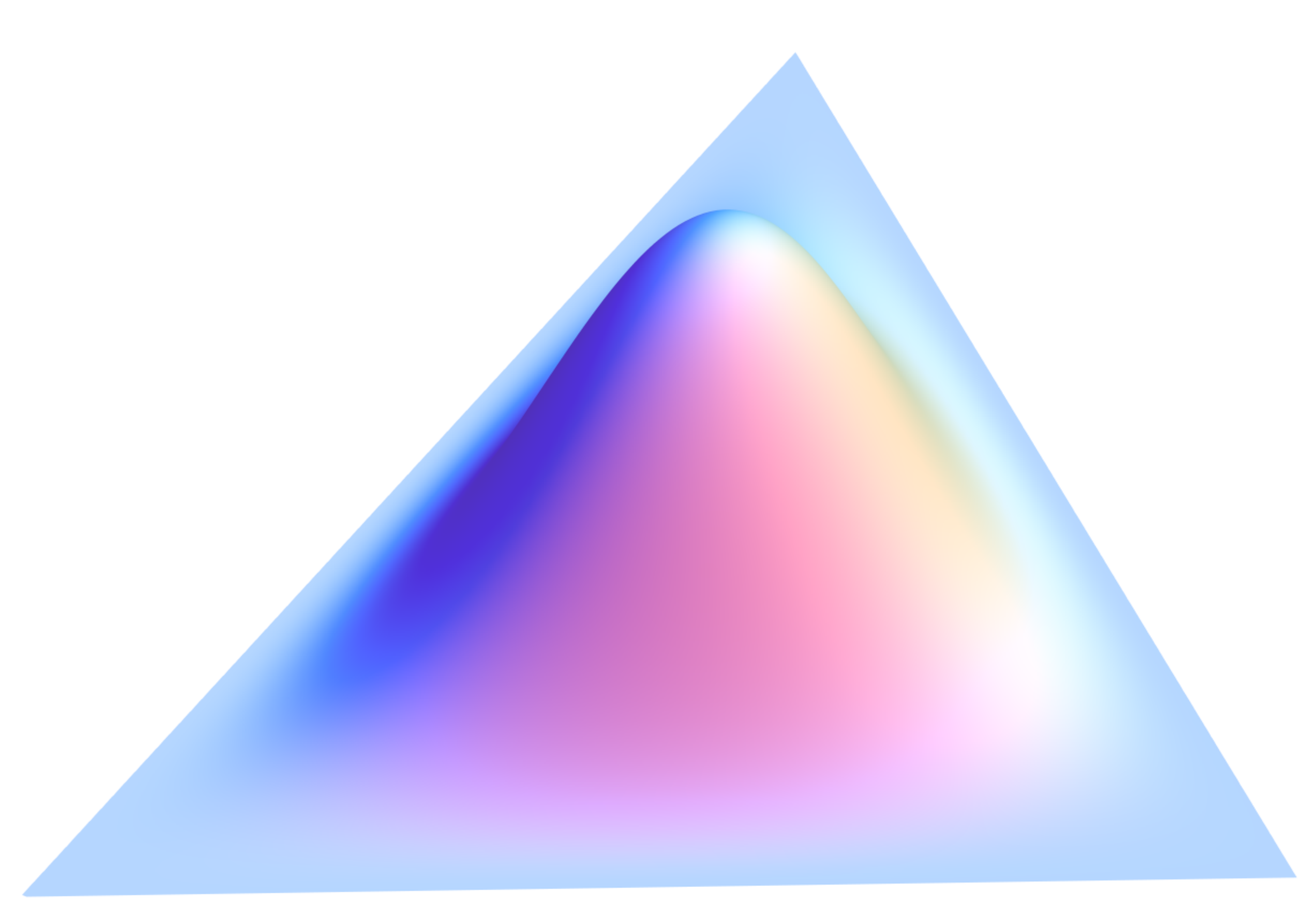} &  
\includegraphics[width=4.8cm,clip]{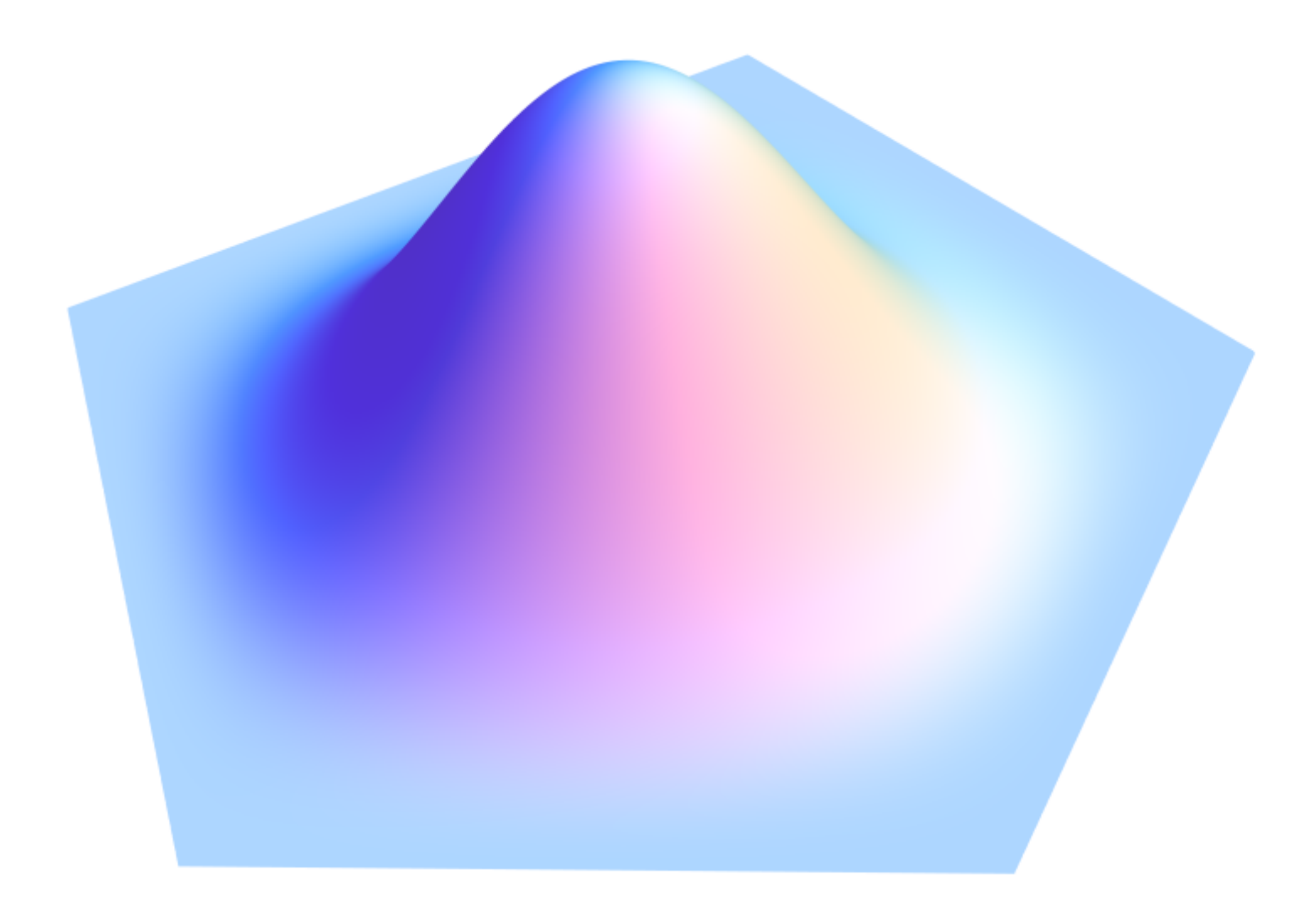} & 
\includegraphics[width=4.8cm,clip]{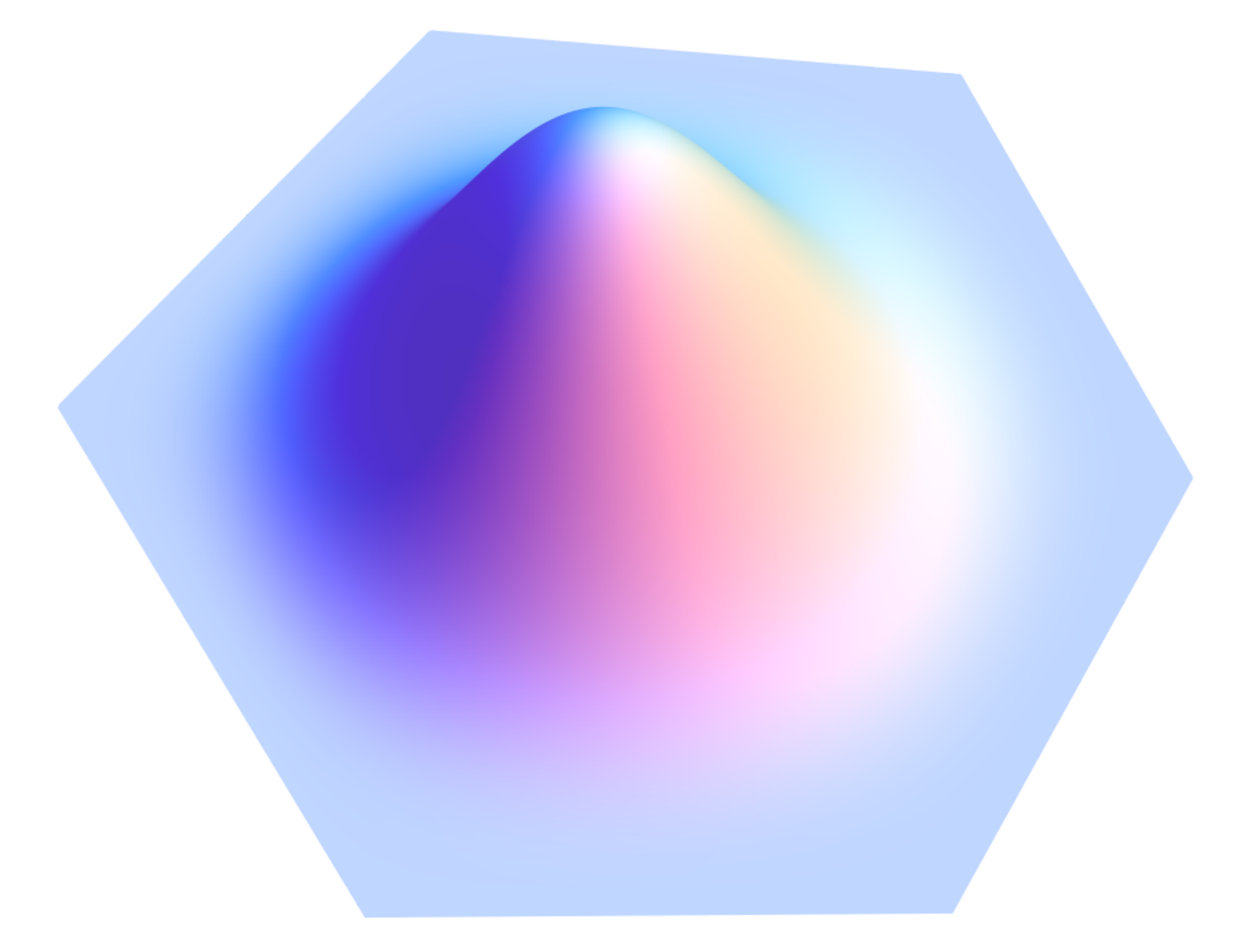} \\
\multicolumn{3}{c}{Exact solutions} \\[0.3cm]
\includegraphics[width=4.8cm,clip]{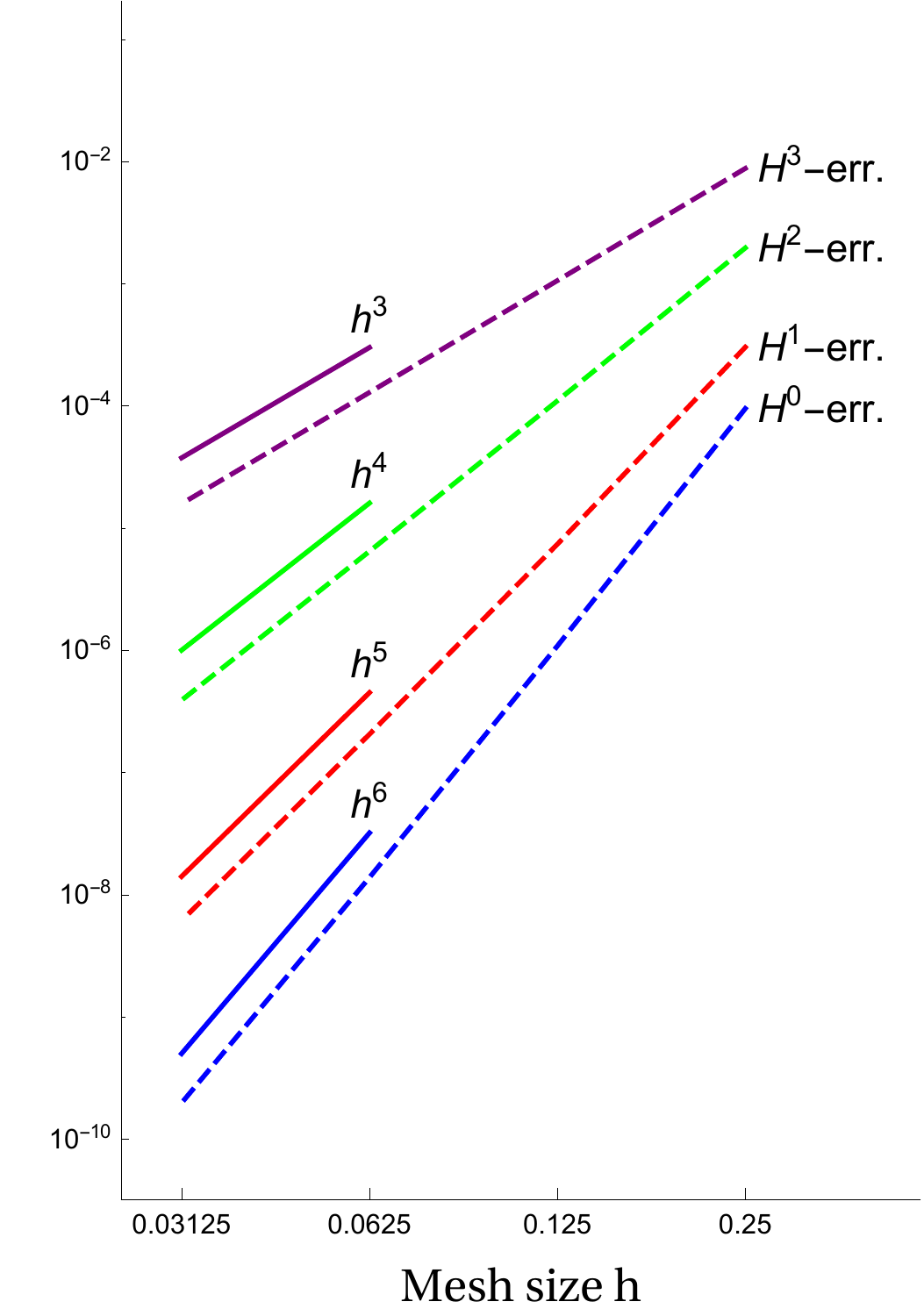} &  
\includegraphics[width=4.8cm,clip]{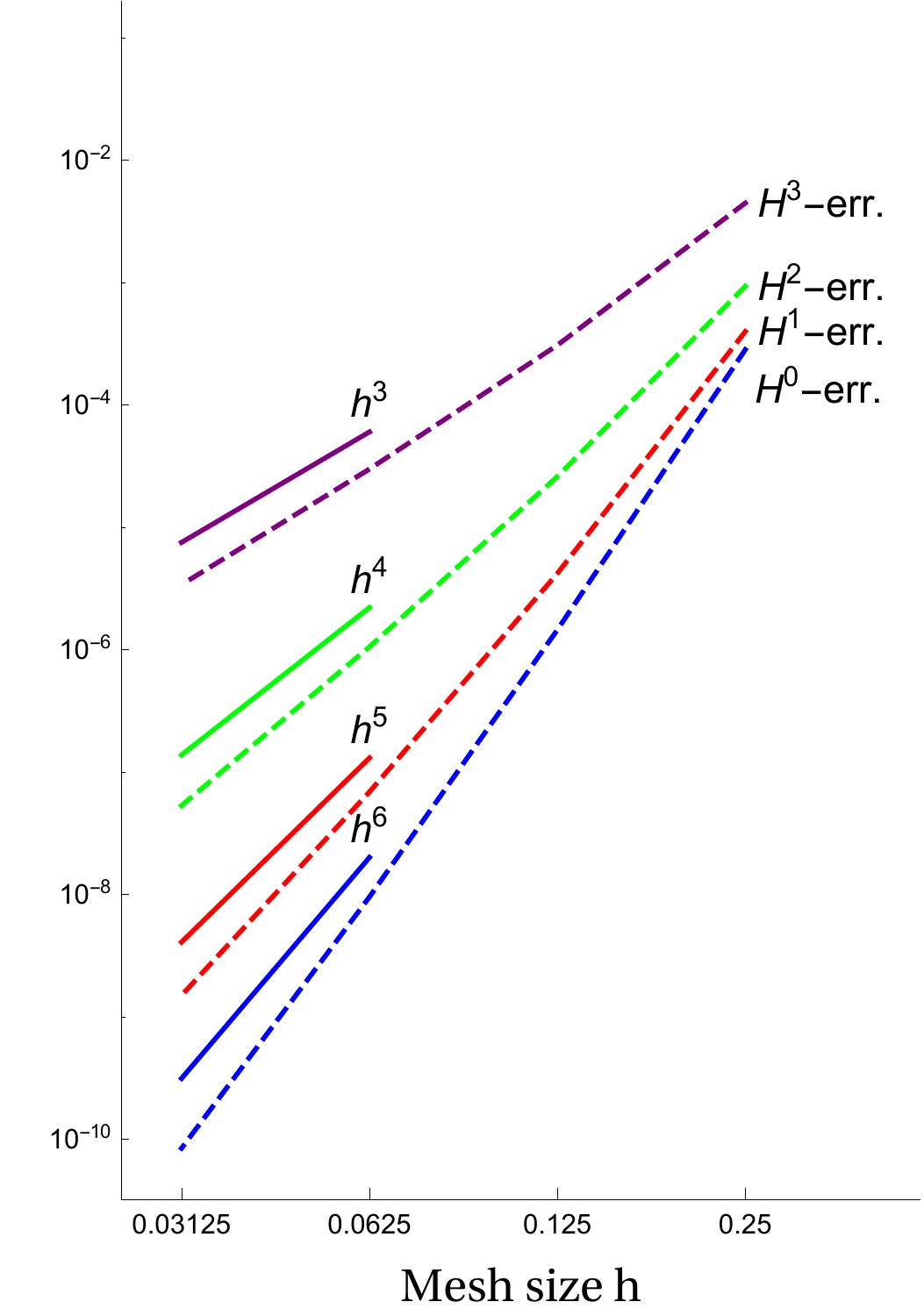} & 
\includegraphics[width=4.8cm,clip]{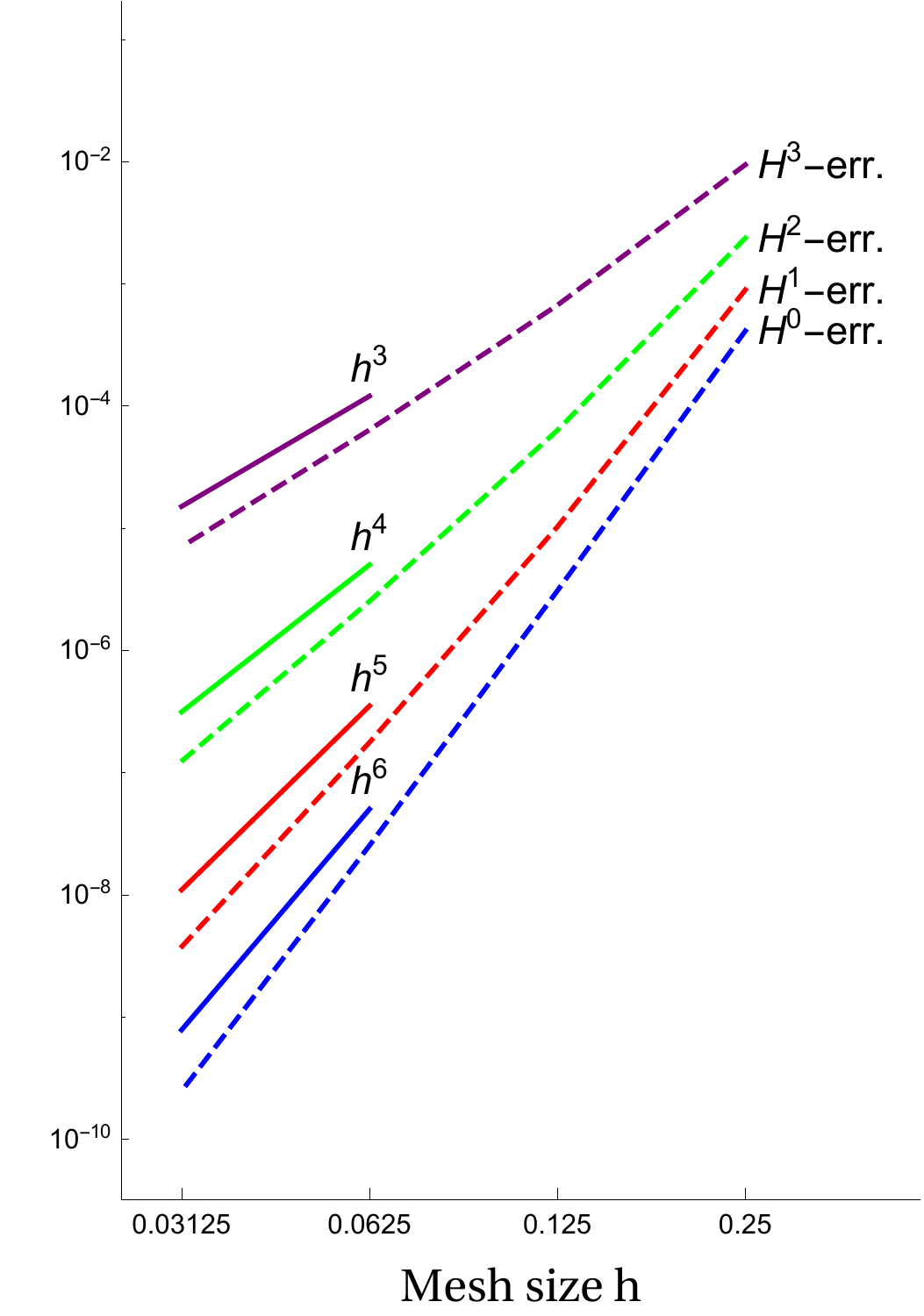} \\
\multicolumn{3}{c}{Relative $H^i$-errors, $i=0,\ldots,3$} \\[0.3cm]
\includegraphics[width=4.8cm,clip]{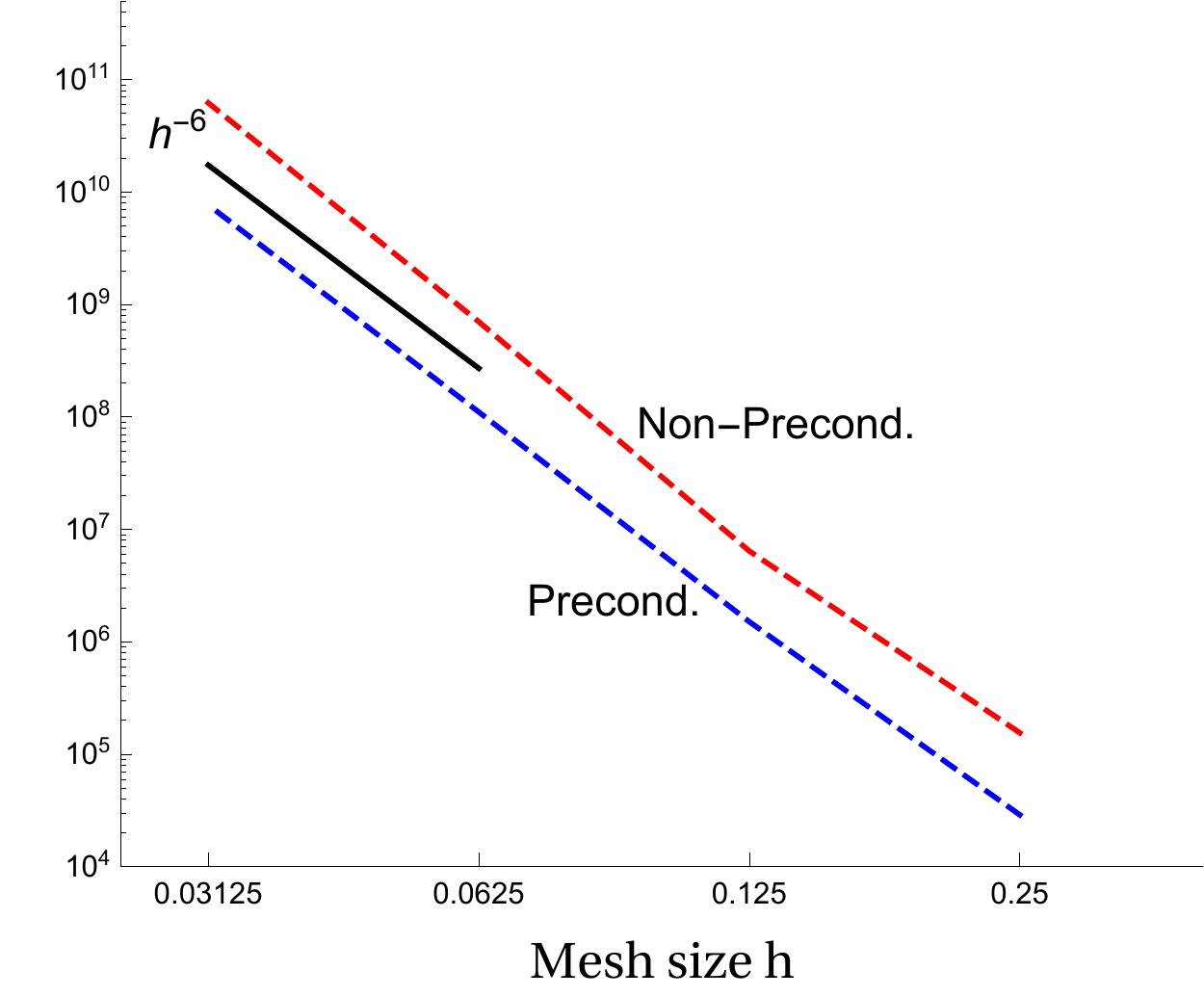} &  
\includegraphics[width=4.8cm,clip]{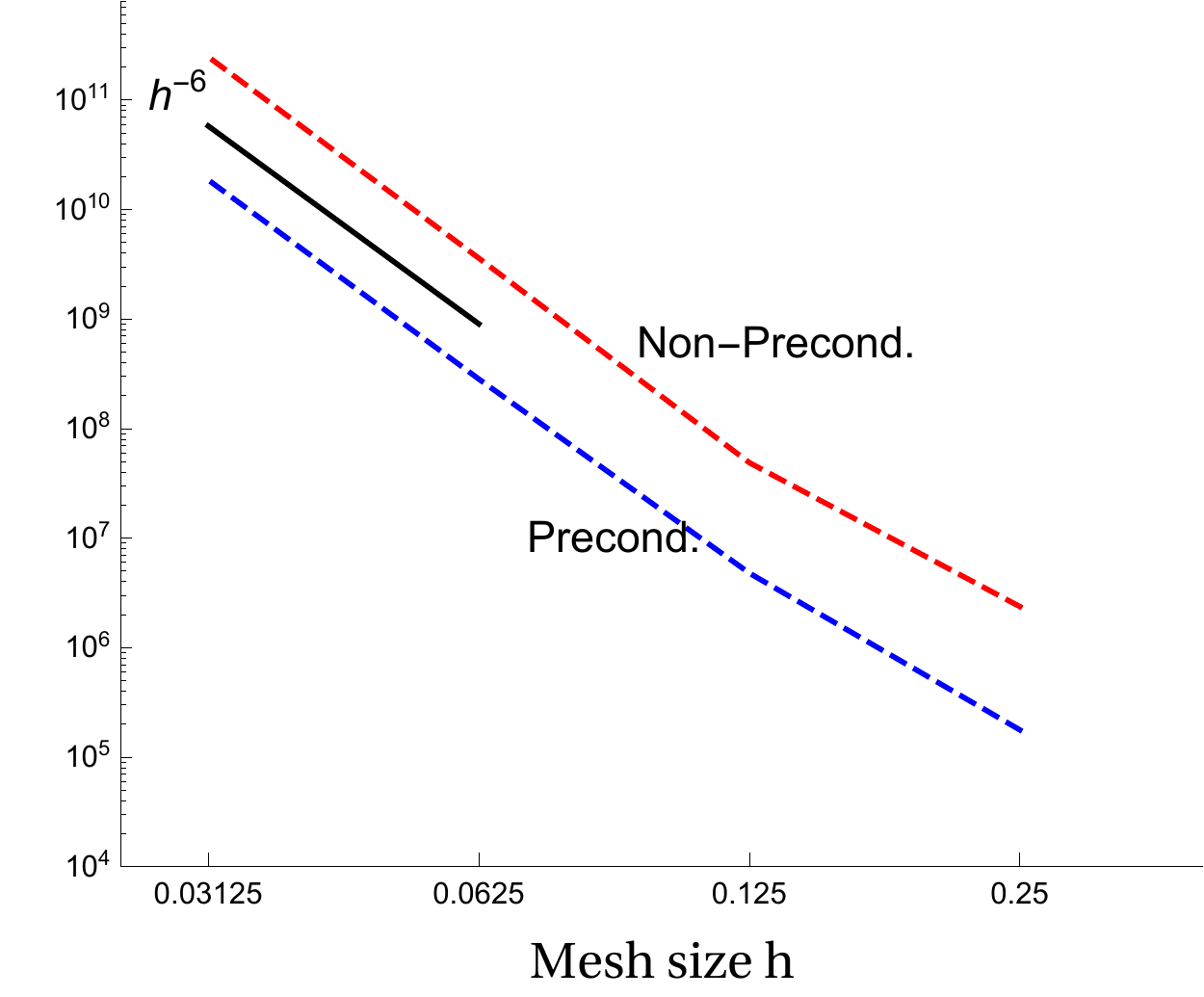} & 
\includegraphics[width=4.8cm,clip]{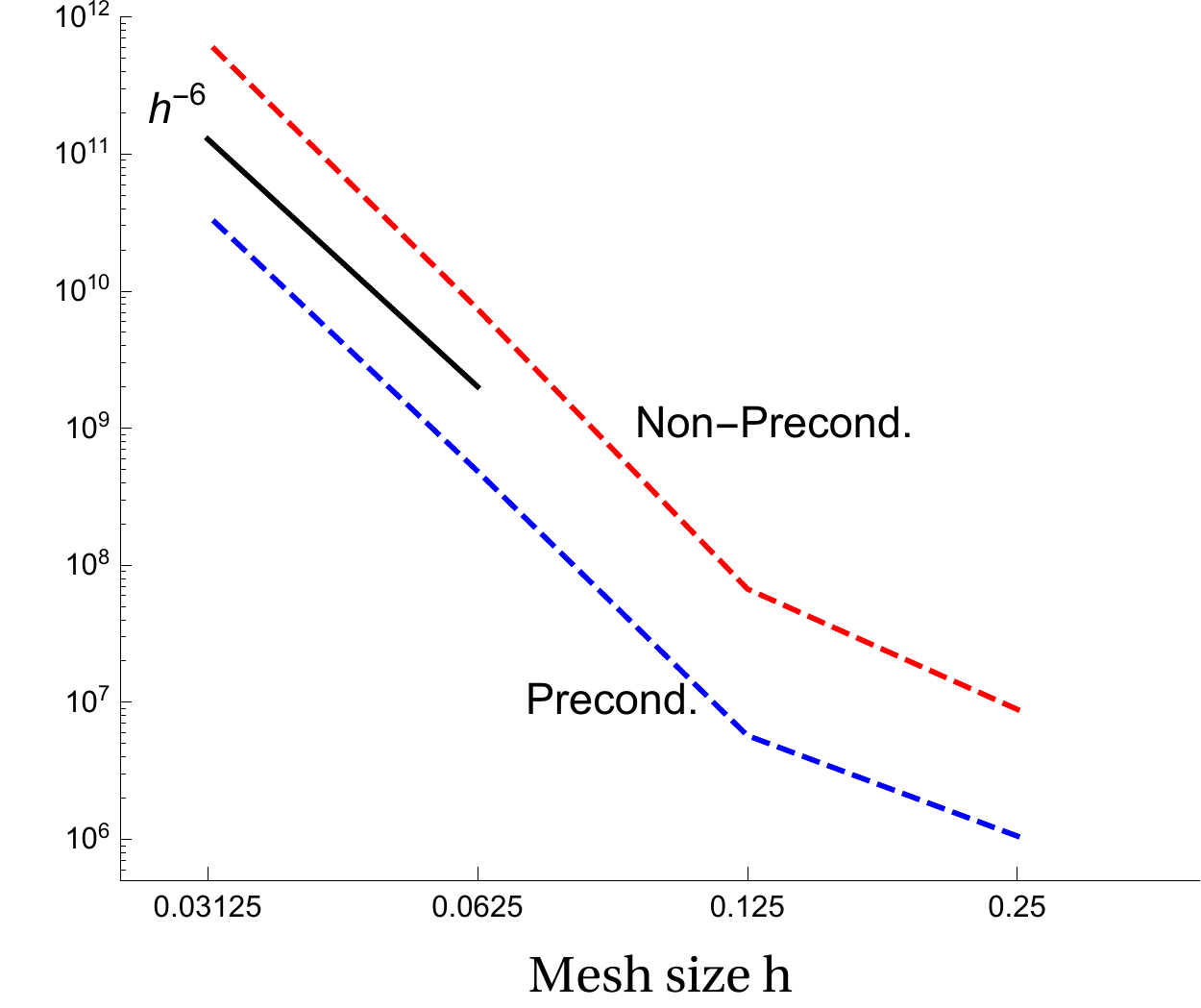} \\
\multicolumn{3}{c}{Condition numbers~$\kappa$ of the {\MK preconditioned and non-preconditioned} stiffness matrices~$S$} 
\end{tabular}
\caption{Solving the triharmonic equation over different multi-patch domains~$\Omega$ (cf. Example~\ref{ex:example1}).}
\label{fig:example1}
\end{figure}
\end{ex}

\begin{ex} \label{ex:example2} 
We consider the bilinearly parameterized five-patch domain with four extraordinary vertices of valency~$3$, which is visualized in Fig.~\ref{fig:example2} (first row). 
{\MK For the mesh-sizes $h=\frac{1}{k+1}$, $k \in \{3,7,15,31 \}$, nested isogeometric spline spaces $\W$ of degree~$p=5,6$ and regularity~$r=2$ (for $p=5,6$) and $r=3$ (for $p=6$) are 
generated. As in Example~\ref{ex:example1}, the construction of the space~$\W$ has to be slightly changed for the case $p=5$, $r=2$ and $h=\frac{1}{4}$. The resulting spaces} 
are used to solve the triharmonic equation~\eqref{eq:triharmonic_problem} with the homogeneous boundary conditions~\eqref{eq:triharmonic_problem_boundary}. 
We use for testing the right side function~$f$ which is obtained by the exact solution
\[
 u(\ab{x}) = (\frac{1}{20000}x_2 (\frac{405}{8} - \frac{27 x_1}{8} - x_2)(\frac{425}{38} + \frac{4 x_1}{19} - x_2)(\frac{23 x_1}{3} -
    x_2))^3 ,
\]
see Fig.~\ref{fig:example2} (first row). {\MK The resulting relative $H^i$-errors are of order $\mathcal{O}(h^{p+1-i})$,} and 
the estimated growth rates of the diagonally scaled stiffness matrices~$S$ are of order~$\mathcal{O}(h^{-6})$. {\MK As in Example~\ref{ex:example1}, we also present 
the condition numbers of the non-preconditioned stiffness matrices~$S$, see Fig.~\ref{fig:example2}, which are again slightly higher than for the preconditioned case 
(i.e. using diagonal scaling) but still seems to grow of order~$\mathcal{O}(h^{-6})$. This indicates again that the constructed basis functions are well-conditioned.}
\begin{figure}[htp]
\centering \footnotesize
\begin{tabular}{cc}
\includegraphics[width=6.6cm,clip]{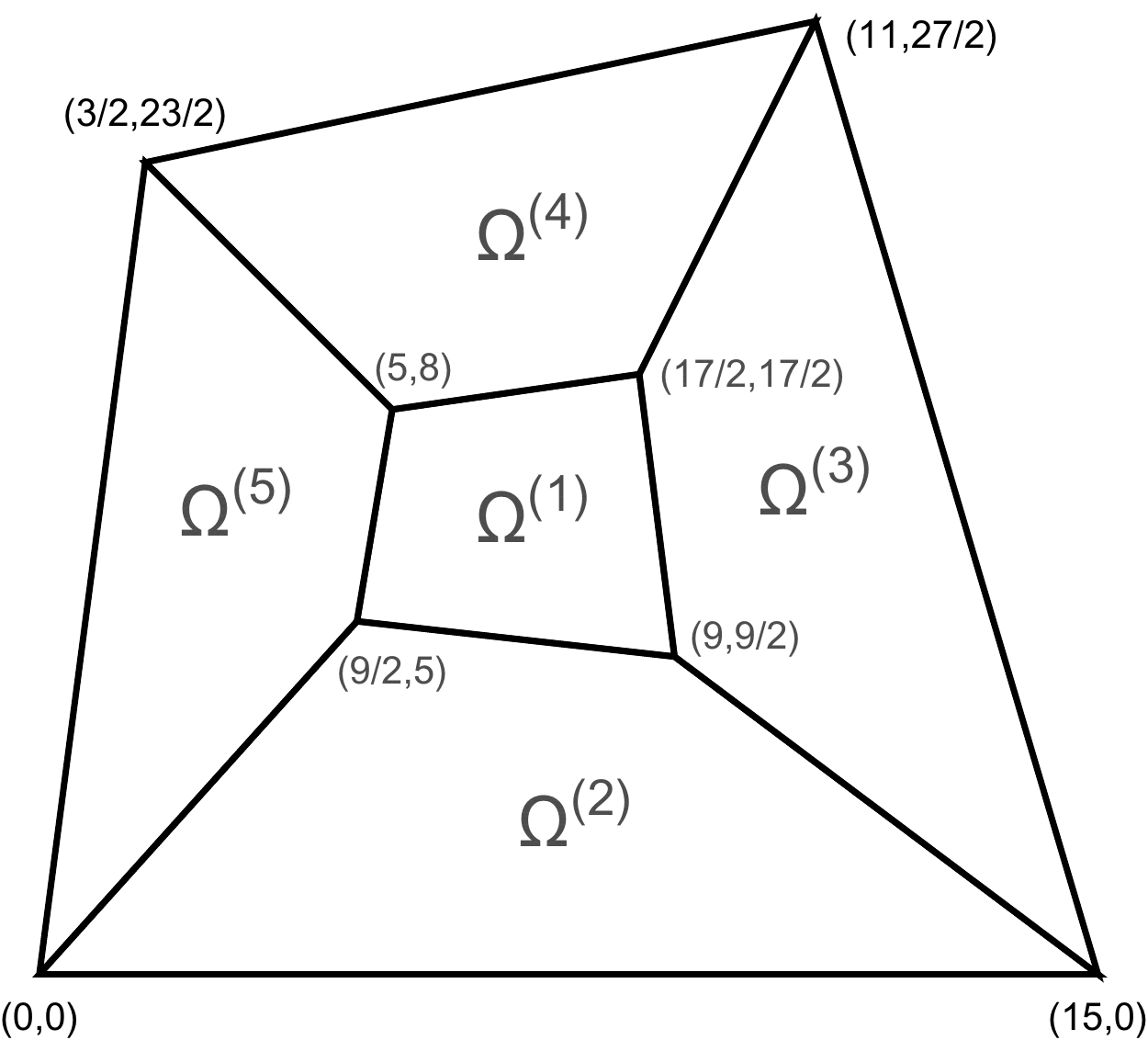} &  \includegraphics[width=6.6cm,clip]{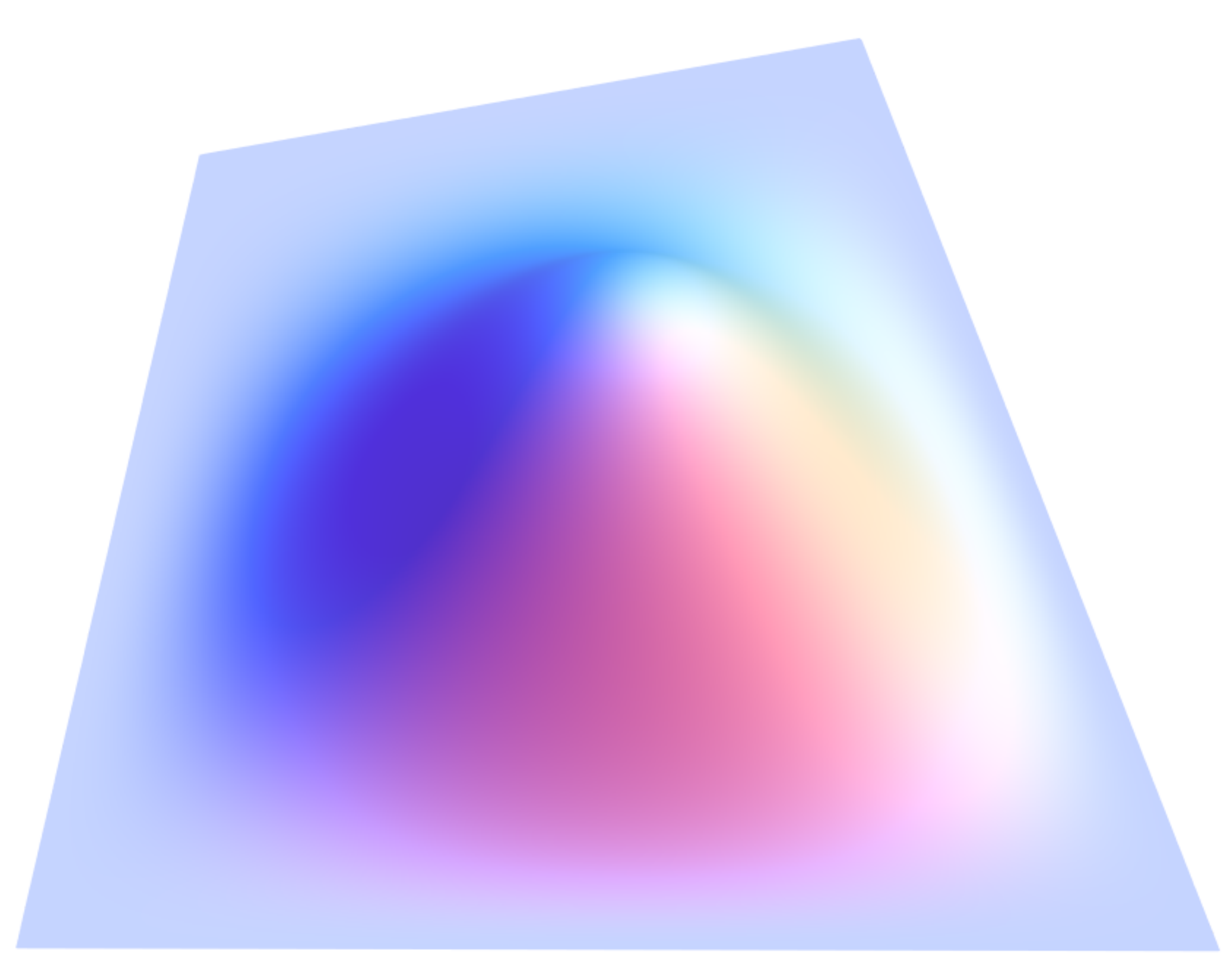} \\
Computational domain & Exact solution
\end{tabular}
\begin{tabular}{ccc}
\includegraphics[width=4.9cm,clip]{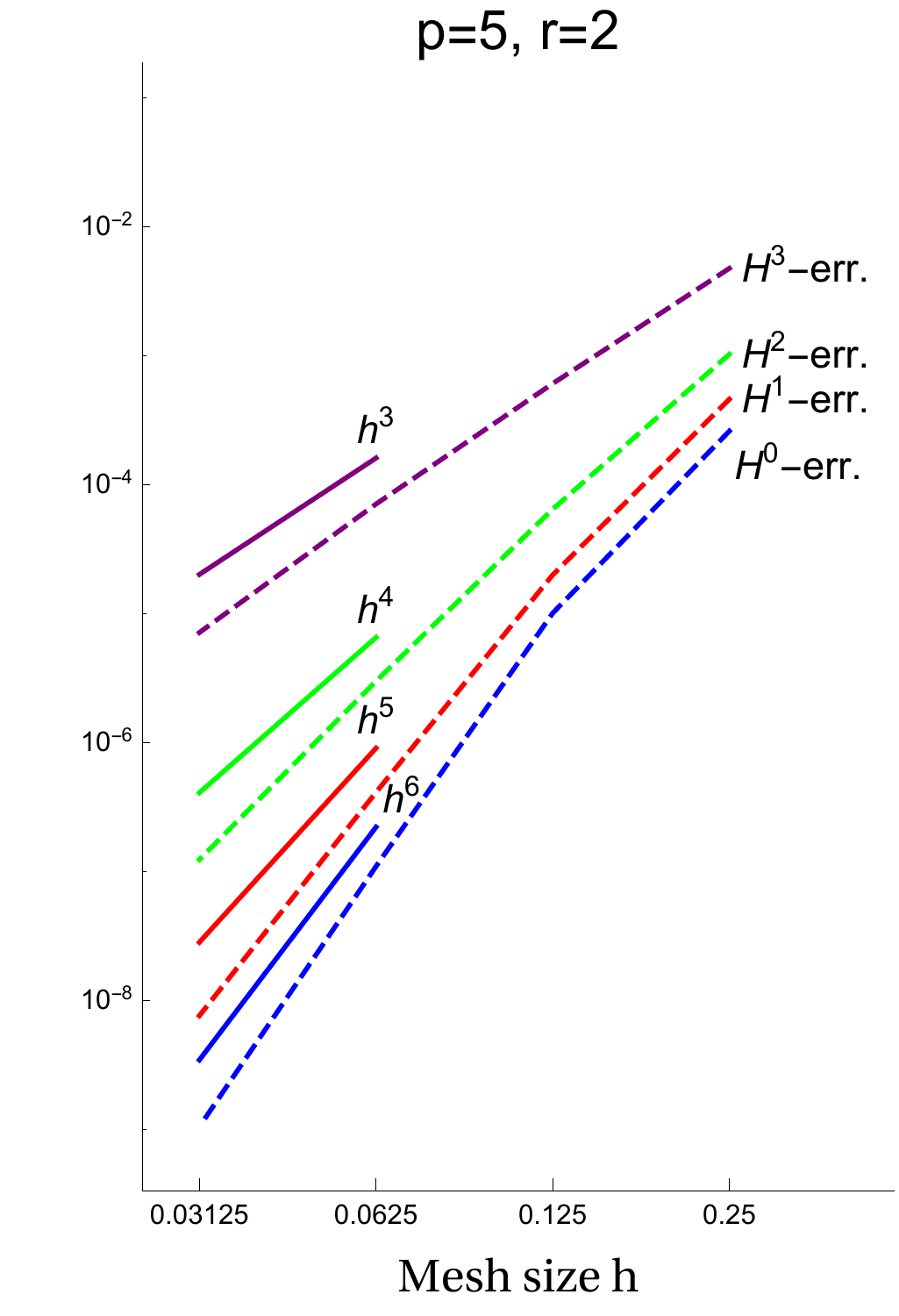} &
\includegraphics[width=4.9cm,clip]{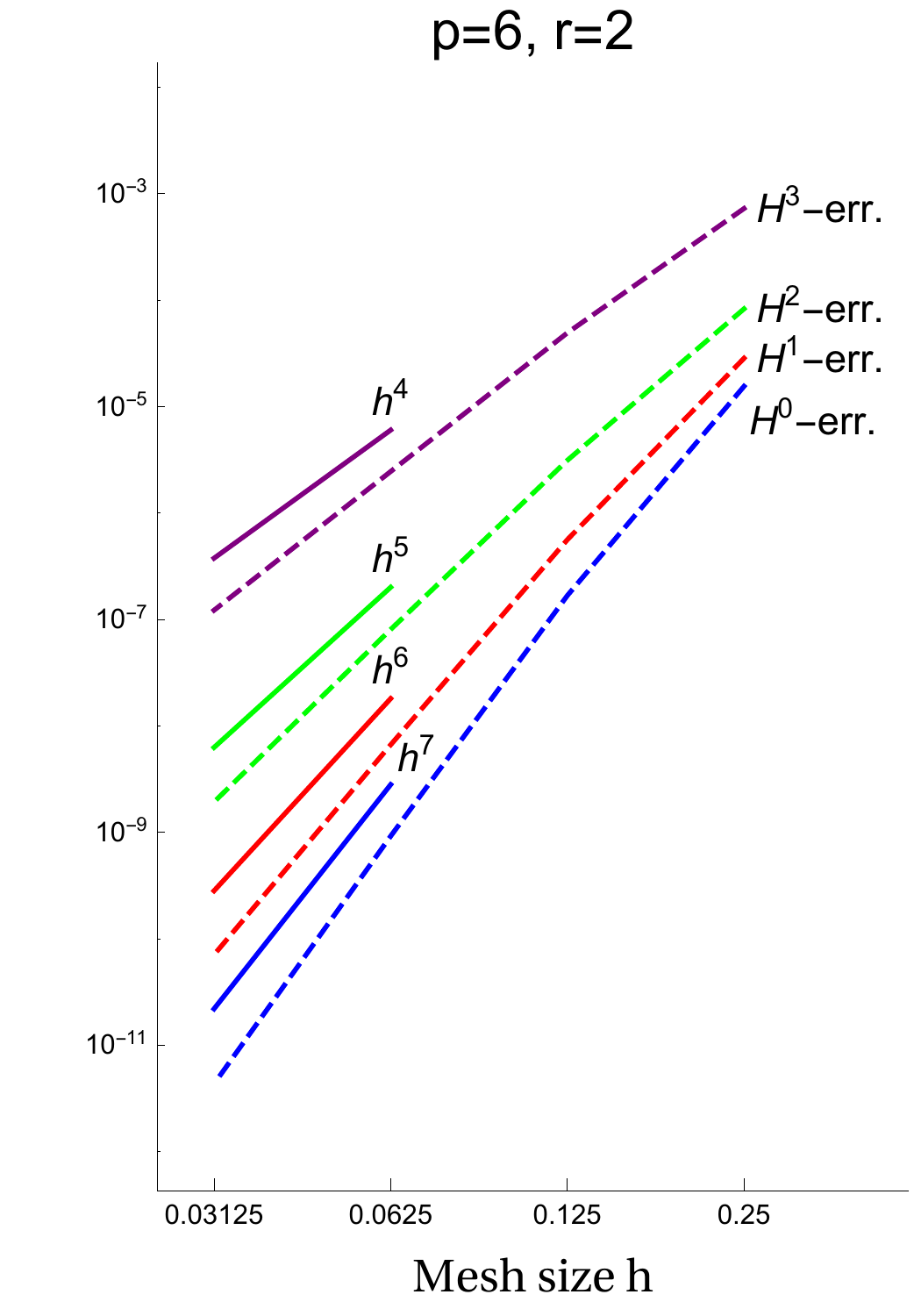} &
\includegraphics[width=4.9cm,clip]{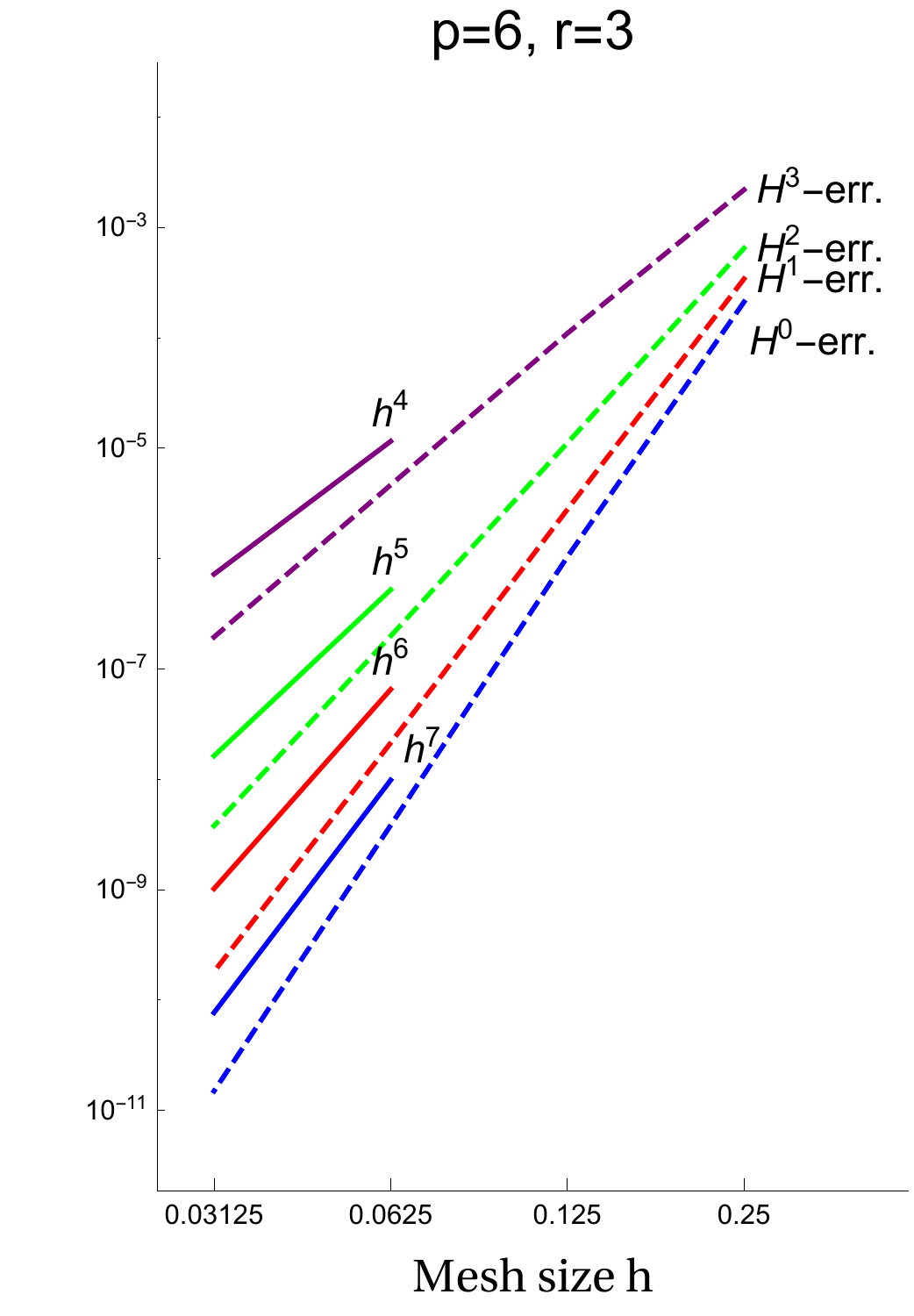}\\
\multicolumn{3}{c}{Relative $H^i$-errors, $i=0,\ldots,3$} \\
\includegraphics[width=4.9cm,clip]{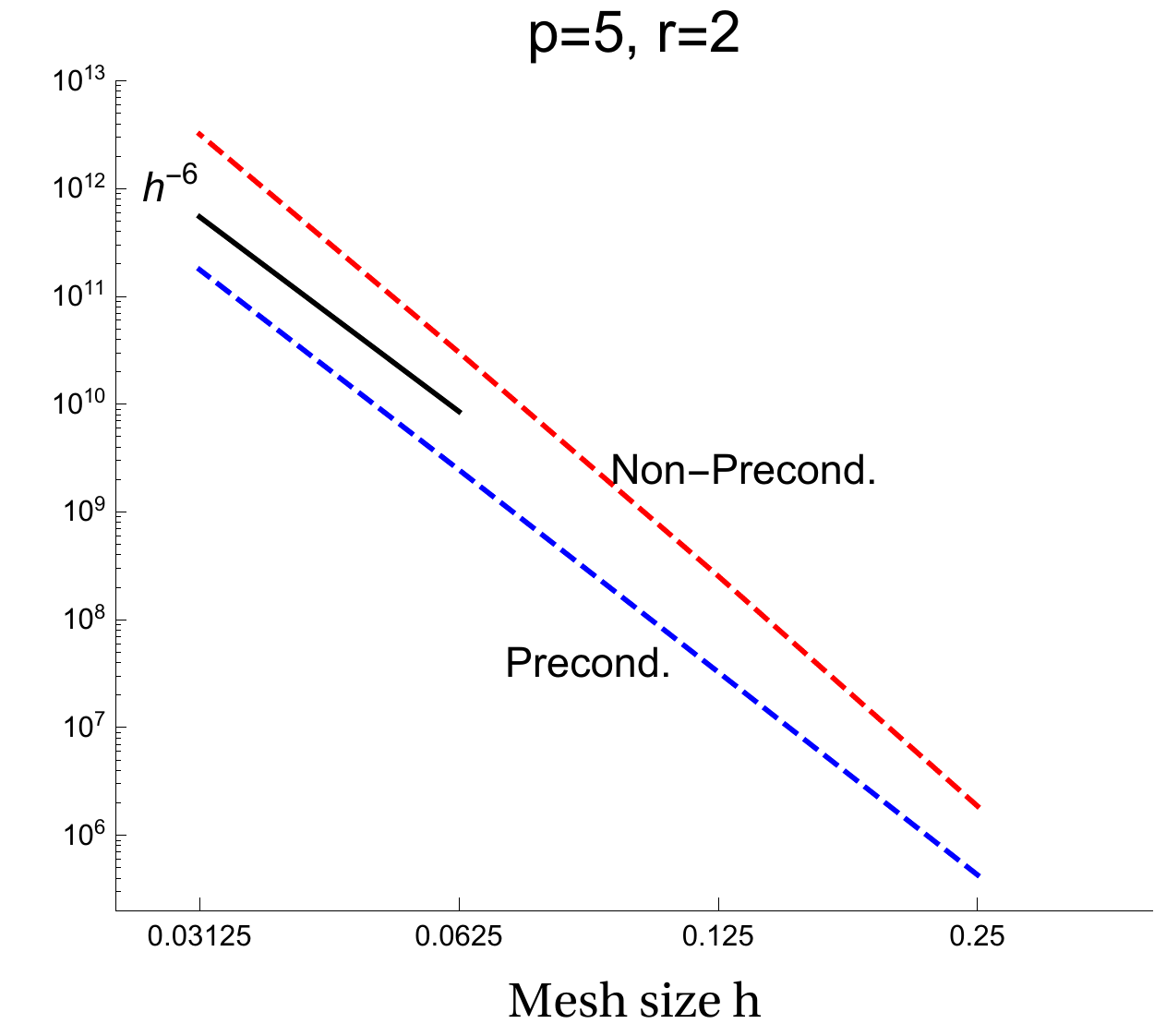} &
\includegraphics[width=4.9cm,clip]{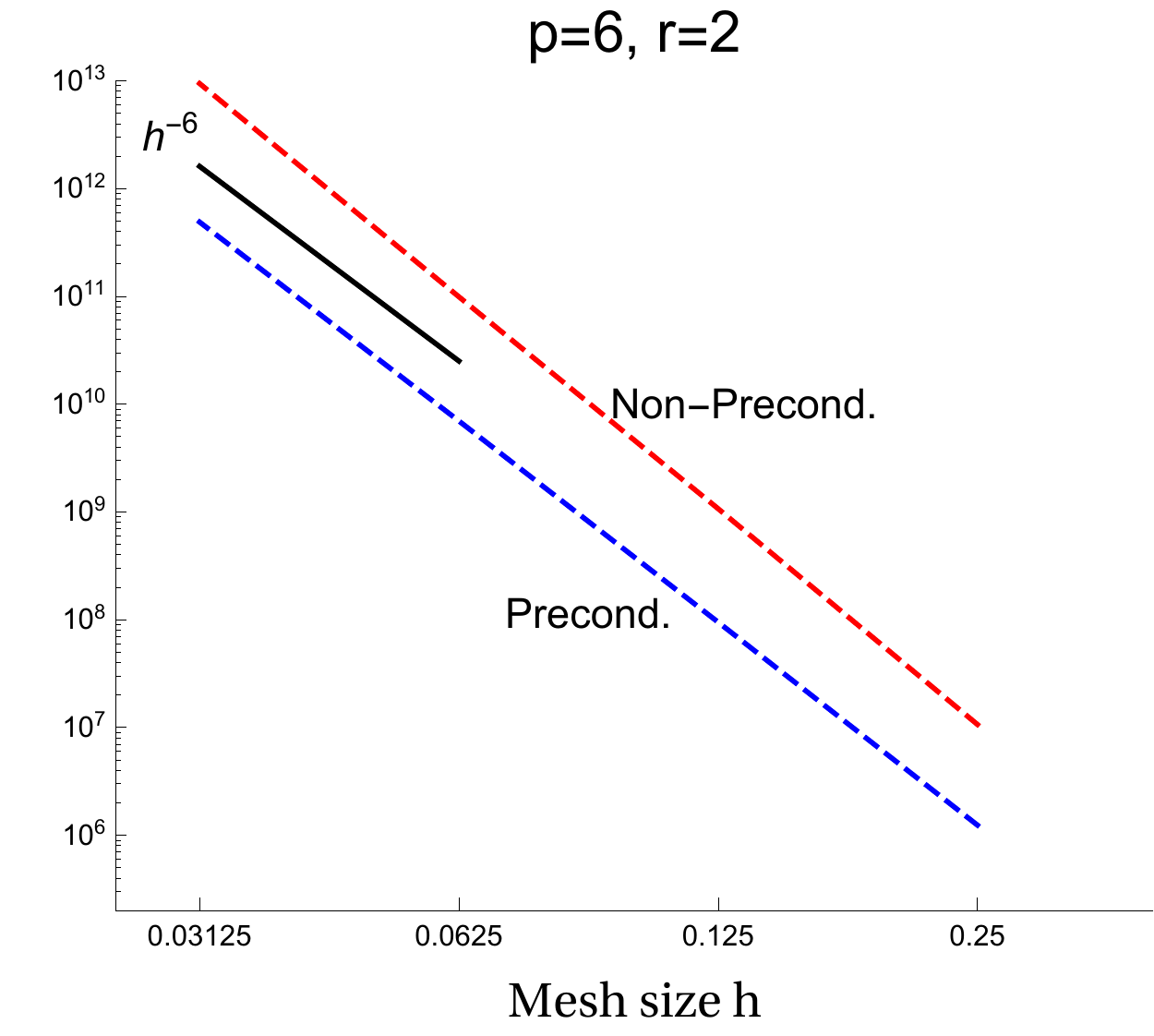} &
\includegraphics[width=4.9cm,clip]{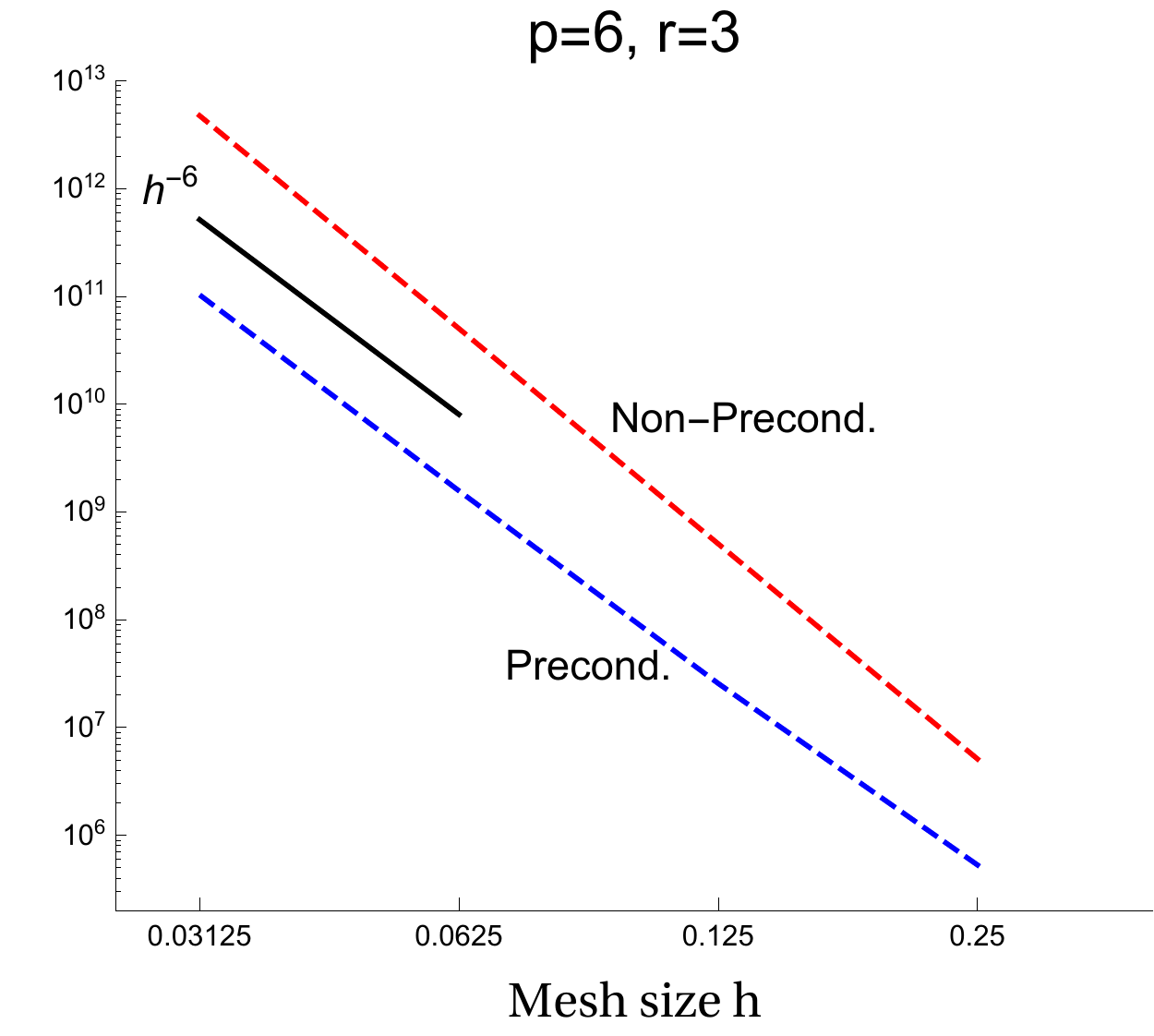}\\
\multicolumn{3}{c}{{\MK Condition numbers~$\kappa$ of preconditioned and non-preconditioned stiffness matrices~$S$}} 
\end{tabular}
\caption{Solving the triharmonic equation over the given multi-patch domain~$\Omega$ (cf. Example~\ref{ex:example2}).}
\label{fig:example2}
\end{figure}
\end{ex}

\section{Conclusion} \label{sec:conclusion}

We described a method for solving the triharmonic equation over bilinearly parameterized planar multi-patch domains. The presented approach is based on the concept of IGA and uses 
as discretization space~$\W$ a space of globally $C^{2}$-smooth isogeometric functions. The discretization space~$\W$ is the span of three different types of {\MK basis} functions 
called patch, edge and vertex functions. All of these functions possess a simple representation with small local supports, can be uniformly generated for all 
possible multi-patch configurations, {\VV and numerical examples indicate {\MK that} they are well-conditioned}. The numerical results obtained by solving the triharmonic equation over 
different bilinear multi-patch domains using $h$-refinement demonstrate the potential of our approach.  

The paper leaves several open questions which are worth to study. A first possible topic for future research could be the study of a priori error estimates for the triharmonic equation 
over multi-patch domains under $h$-refinement (similar to the ones in \cite{TaDe14} for single patch domains), and the theoretical investigation of the approximation properties of the 
discretization space~$\W$. {\MK Another topic could be the detailed study of the dimension of the space~$\W$ to get an explicit dimension formula. In \cite{KaVi17b}, the case of the 
entire $C^2$-smooth space~$\V$ was investigated, and the obtained formula there provides an upper bound for the dimension of~$\W$. Like in~\cite{KaVi17b} for the case of~$\V$, 
the dimension of the space~$\W$ is just the sum of the dimensions of the single subspaces (i.e. patch, edge and vertex subspaces). While the numbers of basis functions for the patch 
subspaces~$\mathcal{W}_{0h;\Omega^{(\ell)}}$ and for the edge subspaces~$\mathcal{W}_{0h;\Gamma^{(s)}}$ are explicitly given, the computation of the numbers of 
basis functions for the vertex subspaces~$\mathcal{W}_{0h;\bfm{v}^{(\rho)}}$ still deserves further investigation.

Moreover, one could consider} further $6$-th order PDEs for which the use of the discretization space~$\W$ could be suitable, since these problems require functions 
of $C^2$-smoothness. Possible examples are
the Kirchhoff plate model based on the Mindlin's gradient elasticity theory~\cite{Niiranen2016}, the Phase-field crystal equation \cite{BaDe15, Gomez2012} and 
the gradient-enhanced continuum damage model~\cite{GradientDamageModels}.
The extension of our approach to more general multi-patch domains, such as e.g., bilinear-like planar domains, shells or volumetric domains could be considered, too.

\paragraph*{\bf Acknowledgment}

{\MK The authors wish to thank the anonymous reviewers for their comments that helped to improve the paper.} 
V.~Vitrih was partially supported by the Slovenian Research Agency (research program P1-0285).
This support is gratefully acknowledged.

\appendix

\section{Proof of Proposition~\ref{thm:mainC4}} \label{app:proofProposition}

The proof will be mainly based on the concept of blossoming.
Let $q \in \mathcal{S}_{h}^{p,r}([0,1])$, and let $t_0^{p,r}, t_1^{p,r},\ldots,t_{2p+1+k(p-r)}^{p,r}$ be the corresponding knots of the spline space $\mathcal{S}_{h}^{p,r}([0,1])$.
Then there exists a unique function $\mathcal{Q}^{p,r}:{\MK \R^p \to \R}$, called the \emph{blossom} of $q$, which is symmetric, multi-affine and fulfill 
$\mathcal{Q}^{p,r}(\xi,\xi,\ldots,\xi) = q(\xi)$. These properties imply that the control points of $q$ can be written as
$$
d_\iota = \mathcal{Q}^{p,r}(t_{\iota+1}^{p,r},t_{\iota+2}^{p,r},\ldots,t_{\iota+p}^{p,r}), \quad \iota=0,1,\ldots,p+k(p-r).
$$
Blossoming is a simple approach, which can be used amongst others to perform knot insertion for a spline function or to multiply two spline functions. 
For more details about the concept of blossoming we refer to e.g.~\cite{GO03, Ra89, Se93}. 

The following two lemmas will be needed.
\begin{lem}  \label{lem:knotInsertion}
Let $N_j^{p,r+1}(\xi) = \sum_{ \iota=0}^{p+k(p-r)} \widetilde{d}_\iota N_\iota^{p,r}(\xi)$. Then
$ 
\widetilde{d}_\iota = 0
$
for $ \iota<j.$
\end{lem}
\begin{pf}
Let $d_\iota$ be control points of $N_j^{p,r+1} \in \mathcal{S}_{h}^{p,r+1}([0,1])$, i.e., $d_\iota = \delta_{j,\iota}$. Moreover let $ \widetilde{d}_\iota $ denote control 
points of $N_j^{p,r+1}$ represented in the space $\mathcal{S}_{h}^{p,r}([0,1])$. Then (see e.g. \cite{GO03})
$$
  \widetilde{d}_\iota = \mathcal{Q}^{p,r+1} (t_{\iota+1}^{p,r},\ldots, t_{\iota+p}^{p,r}).
$$
Since 
$$
 {d}_\iota = \mathcal{Q}^{p,r+1} (t_{\iota+1}^{p,r+1},\ldots, t_{\iota+p}^{p,r+1}) 
 \quad {\rm and} \quad
 t_{\iota+p}^{p,r+1} \geq t_{\iota+p}^{p,r},
$$
it follows that 
$\widetilde{d}_\iota = \sum_{m \leq \iota} c_m d_m$, $c_m \in \R$, which implies $\widetilde{d}_\iota = 0$ for $\iota<j$.
\qed
\end{pf}
\begin{lem}  \label{lem:multiplication}
Let $(\omega_0 (1-\xi) + \omega_1 \xi) \, N_j^{p-1,r}(\xi) =  \sum_{ \iota=0}^{p+k(p-r)} \widehat{d}_\iota N_\iota^{p,r}(\xi)$. Then
$ 
\widehat{d}_\iota = 0
$
for $ \iota<j.$
\end{lem}
\begin{pf}
Let $d_\iota$ denote control points of $ N_j^{p-1,r} \in \mathcal{S}_{h}^{p-1,r}([0,1])$, i.e., $d_\iota = \delta_{j,\iota}$,
and let $\mathcal{Q}^{p-1,r}$ denote its blossom. Moreover let $\widehat{d}_\iota$ denote the control points of $(\omega_0 (1-\xi) + \omega_1 \xi) \, N_j^{p-1,r}(\xi) 
\in \mathcal{S}_{h}^{p,r}([0,1])$.
Then (see e.g. \cite{GO03})
$$
  \widehat{d}_\iota = \frac{1}{p} \sum_{m=1}^p \mathcal{Q}^{p-1,r} (t_{\iota+1}^{p,r},\ldots,t_{\iota+m-1}^{p,r} ,t_{\iota+m+1}^{p,r},\ldots, t_{\iota+p}^{p,r}) 
  \left(\omega_0 (1-t_{\iota+m}^{p,r}) + \omega_1 t_{\iota+m}^{p,r} \right).
$$
We have to prove that $\widehat{d}_\iota = \sum_{n \leq \iota} c_n d_n$, $c_n \in \R$. Since
$d_\iota= \mathcal{Q}^{p-1,r} (t_{\iota+1}^{p-1,r},\ldots,t_{\iota+p}^{p-1,r})$ and $t_{\iota+p}^{p,r} \leq t_{\iota+p}^{p-1,r}$, it follows that
$ \mathcal{Q}^{p-1,r} (t_{\iota+1}^{p,r},\ldots,t_{\iota+m-1}^{p,r} ,t_{\iota+m+1}^{p,r},\ldots, t_{\iota+p}^{p,r}) $ does not involve $d_n$, $n>\iota$. Therefore $\widehat{d}_\iota $ is 
independent of $d_n$, $n>\iota$, implying  $\widehat{d}_\iota = 0$ for $ \iota<j$. \qed
\end{pf}

Proof of Proposition~\ref{thm:mainC4}:
Recall \eqref{eq:basisFunctionsGenericG}. We first observe that the first summation in \eqref{eq:defgljS} follows directly from \eqref{eq:Mj}. 
It remains to prove that the only nonzero coefficients $d_{\Gamma^{(s)};m,n}^{(\tau)}$ might be the ones with $n \geq \max(0, i+j-m)$ and $n \leq \min(d-1,d-n_i+j-i+m)$. 
 
The lower bound follows immediately by using
$$ 
\left( N_j^{p,r}(\xi) \right) '  = \sum_{\iota=j-1}^j d_\iota \, N_\iota^{p-1,r-1}(\xi), \quad 
 \left( N_j^{p,r}(\xi) \right) ''  = \sum_{\iota=j-2}^j c_\iota \, N_\iota^{p-2,r-2}(\xi),
 $$
 and by Lemma~\ref{lem:knotInsertion} and  Lemma~\ref{lem:multiplication}. The upper bound can be shown by first considering the function 
 $$
   \widehat{g}^{(\tau)}_{\Gamma^{(s)};i,j}(\xi_1^{(\tau)},\xi_2^{(\tau)}) = g^{(\tau)}_{\Gamma^{(s)};i,n_i-1-j} (\xi_1^{(\tau)},1-\xi_2^{(\tau)}),
$$
which possesses again the lower bound $n \geq \max(0, i+j-m)$ for possible nonzero coefficients $\widehat{d}_{\Gamma^{(s)};m,n}^{(\tau)}$. This directly implies the upper 
bound $n \leq \min(d-1,d-n_i+j-i+m)$ for possible nonzero coefficients $d_{\Gamma^{(s)};m,n}^{(\tau)}$ of the function~$g^{(\tau)}_{\Gamma^{(s)};i,j}$.
\qed

\end{document}